\documentclass[10pt]{article}
\usepackage[utf8]{inputenc}
\usepackage[letterpaper,top=1in,bottom=1in,left=1.2in,right=1.2in]{geometry}
\usepackage[english]{babel}
\usepackage{amsmath}
\usepackage{amsthm}
\usepackage{bm}
\usepackage{amssymb}
\usepackage{dsfont}
\usepackage{commath}
\usepackage{mathtools}
\usepackage{float}
\usepackage{algorithm}
\usepackage{algpseudocode}

\usepackage[numbers]{natbib}
\usepackage{graphicx}

\usepackage{tikz}
\usetikzlibrary{positioning}

\usepackage{mathrsfs}

\newcommand{\cA}{\mathcal{A}}

\newcommand{\cD}{\mathcal{D}}

\newcommand{\cP}{\mathcal{P}}

\newcommand{\EE}{\mathbb{E}}

\newcommand{\NN}{\mathbb{N}}

\newcommand{\PP}{\mathbb{P}}

\newcommand{\RR}{\mathbb{R}}

\newcommand{\bone}{\mathbf{1}}

\newcommand{\GaussCDF}{\Phi_{\mathrm G}}
\newcommand{\GaussTail}{\overline{\Phi}_{\mathrm G}}
\newcommand{\GaussDensity}{\varphi_{\mathrm G}}

\theoremstyle{plain}
\newtheorem{theorem}{Theorem}[section]
\newtheorem{corollary}[theorem]{Corollary}
\newtheorem{lemma}[theorem]{Lemma}
\newtheorem{proposition}[theorem]{Proposition}

\newtheorem{assumption}[theorem]{Assumption}

\theoremstyle{definition}
\newtheorem{remark}[theorem]{Remark}
\newtheorem{definition}[theorem]{Definition}

\numberwithin{equation}{section}

\usepackage{fancyhdr}
\usepackage[hidelinks]{hyperref}
\title{Optimal Loss Allocation and Reflected Free Boundaries in Systemic Risk\footnote{Research of QY was
partially supported by the National Science Foundation grant DMS 2406762.}\footnote{Work was partially done while both authors were doctoral students at Princeton University.}}
\author{ }
\author{Yucheng Guo\footnote{Department of Mathematical Sciences, Carnegie Mellon University, Pittsburgh, PA, 15213, USA, email: 
{\tt yuchengg@andrew.cmu.edu}. }
\and Qinxin Yan\footnote{Industrial Engineering \& Operations Research, University of California Berkeley, Berkeley, CA, 94720, USA, email: 
{\tt qinxinyan@berkeley.edu}. }}

\date{}
\usepackage{oubraces}

\begin{document}

\maketitle

\vspace{0mm}
\begin{abstract}
We study a singular mean-field control problem in which a central planner allocates losses generated by bank defaults to maximize terminal system health. Every generated loss must be borne by the surviving banks, without external capital. We prove that a cutoff, or “taxing-the-richest” rule is optimal for every non-decreasing concave terminal payoff, including the survival indicator. We construct the continuous-time optimizer by solving a time-discretized problem and passing to the limit, and establish its uniqueness in law for every non-constant payoff. The optimal flow is characterized by a novel reflected free-boundary problem. We establish well-posedness in a probabilistic solution class and prove positive-time smoothness of both the free boundary and the density of surviving banks. The main difficulty is the dependence of the limiting feedback on a moving support endpoint, which is unstable under Wasserstein convergence. We address this through stochastic-order comparisons, compactness estimates for the cutoff boundaries, and an augmented call-function formulation.

\end{abstract}
\vspace{1mm}

\vspace{3pt}

\def\d{\mathrm{d}}

\section{Introduction}

\subsection{The problem and main results}
The default of a financial institution can impose losses on the rest of the system and weaken institutions that remain solvent. When these losses must be borne within the system, their allocation becomes a policy question in its own right. We study a central planner who assigns losses generated by defaults among surviving banks, with the aim of preserving the system's future health. The planner must allocate every loss in full and has no access to external capital. Our main result identifies a simple optimal rule: concentrate the burden on the currently healthiest banks. This rule is optimal for an entire class of increasing concave welfare criteria, including the proportion of banks that survive to a given horizon.

We model a representative bank by its distance to default. Before default, its state evolves as a Brownian motion minus the cumulative loss assigned to it; after default, it is absorbed at zero (and has no further role in the system). Thus, the health (net asset) of the representative bank can be described as
\[
    X_t=X_0+B_{t\wedge\tau}-L_{t\wedge\tau},\quad \tau=\inf\{t\geq0:X_t=0\},
\]
where $B$ is a standard Brownian motion independent of $X_0$, and $L$ is an adapted, non-decreasing allocation process that stops after default. Each default generates a loss of size $\alpha>0$ per unit of population mass to be allocated, and admissibility requires the exact aggregate balance
\[
    \mathbb E[L_t]=\alpha\mathbb P[\tau\leq t],\quad 0\leq t\leq T.
\]
Allocation jumps may not send a solvent bank to or through zero. The allocation rule can affect future defaults and therefore the amount of loss subsequently generated. The full admissible class, allowing adapted controls on a variable probability space, is specified in Sections~\ref{sec:discrete} and~\ref{sec:continuous}.

For a fixed horizon $T>0$, the objective is to maximize
\[
\EE[f(X_T)]=\int_{[0,\infty)}f(x)\,\mu_T(dx),\quad \mu_t=\mathcal L(X_t),
\]
over all admissible allocations. Here $f$ is finite, non-decreasing, and concave on $[0,\infty)$, and may have an upward right jump at zero. In particular, $f=\mathbf 1_{(0,\infty)}$ gives the survival probability. We use the following standing assumption.

\begin{assumption}\label{ass:initial-law}
The initial law $\mu_0\in\mathcal P_1([0,\infty))$ is compactly supported and satisfies
\[
    \mu_0(\{0\})=0,\quad \int_{[0,\infty)}x\,\mu_0(dx)>\alpha.
\]
\end{assumption}
No density or boundary compatibility condition is imposed on $\mu_0$. The strict mean condition ensures a positive reserve margin: exact loss balance gives $\mathbb E[X_t]=\mathbb E[X_0]-\alpha\mathbb P[\tau\leq t]>0$.

\paragraph{Optimal loss allocation.}
We first separate Brownian evolution from allocation on a time grid. At each allocation date, all banks above an endogenous cutoff are reduced to that cutoff, while banks below it are left unchanged. The cutoff is chosen to match exactly the losses generated since the preceding allocation date. Proposition~\ref{prop:DiscreteOptimal1} identifies this rule as the one-step maximizer in increasing-concave order, and Theorem~\ref{thm:discrete-cutoff-optimality} establishes its dynamic optimality. The same cutoff rule therefore optimizes every payoff in the stated class.

As the time step tends to zero, the discrete values converge to the continuous-time value, and the discrete optimal controls yield an admissible optimal limit (Theorem~\ref{thm:VDist_to_Cont_Conv}). In this limit, the cutoff becomes a deterministic moving upper boundary $\Lambda$: surviving banks are reflected downward at $\Lambda$ and absorbed at zero. The resulting control is optimal in the full admissible class. Moreover, for every nonconstant payoff $f$, both the optimal law flow and the joint law of the optimal state and allocation process are unique (Theorem~\ref{thm:continuous-optimal-control-uniqueness}). This uniqueness includes the survival payoff and does not require strict concavity.

\paragraph{A well-posed and smooth free boundary.}
The optimal dynamics lead to a novel parabolic free-boundary problem. For positive times, let $u(t,\cdot)$ denote the density of the surviving population, whose total mass is $\mathbb P[\tau>t]$. The density and the upper boundary satisfy
\begin{equation}\label{eq:FBP1}
\begin{aligned}
    \partial_tu(t,x)&=\frac12\partial_{xx}u(t,x),
        &&0<x<\Lambda_t,\\
    u(t,0)&=0,\\
    \frac12\partial_xu(t,\Lambda_t)&=-\dot\Lambda_tu(t,\Lambda_t),\\
    u(t,\Lambda_t)&=\alpha\partial_xu(t,0).
\end{aligned}
\end{equation}
The first two boundary conditions describe absorption at zero and reflection at the moving upper edge. The last condition couples the two ends of the interval: the allocation at the reflecting boundary must fund the losses generated by the default flux at the origin.

We first establish well-posedness in a probabilistic solution class that does not assume a smooth boundary. The construction in Section~\ref{sec:continuous} provides existence, and an augmented call-function formulation yields stability and uniqueness (Theorem~\ref{thm:pairwise-FBP-stability} and Corollary~\ref{cor:probabilistic-FBP-uniqueness}). We then prove that $\Lambda$ is actually smooth at every positive time and that $u$ is smooth up to both spatial boundaries, so that \eqref{eq:FBP1} holds classically (Theorem~\ref{thm:positive-time-classical-regularity} and Remark~\ref{rem:classical-regularity-time-endpoints}). Therefore, classical free-boundary dynamics emerge even from singular initial laws. The stability estimates also give Lipschitz continuity of the value function in a distance that records initial call functions and absorption probabilities under uncontrolled Brownian evolution (Proposition~\ref{prop:survival-value-stability}).

The main analytical difficulty appears as the allocation time step vanishes: the loss budgets become small, and the cutoffs approach a support endpoint, which is not controlled by ordinary Wasserstein convergence. We combine stochastic-order comparisons with compactness estimates for the cutoff boundaries to construct the optimal limit. The augmented call function then encodes both the distribution of surviving reserves and the accumulated defaults, making it possible to prove uniqueness before establishing boundary regularity.

\paragraph{Organization of the paper.}
In the remainder of this section, we briefly discuss the connection of our model to the literature. Section~\ref{sec:discrete} solves the discrete-time control problem. Section~\ref{sec:continuous} constructs an optimal continuous-time control by passage to the limit. Section~\ref{sec:continuous-2} establishes well-posedness and positive-time classical regularity of the associated free boundary. Section~\ref{sec:Back2MFC} proves stability of the value function and uniqueness of the optimal control. Auxiliary stochastic-order results are collected in the Appendix Section \ref{sec:appendix}.

\subsection{Related literature}

\paragraph{Systemic risk and loss allocation.}
Financial-network models describe how contractual obligations transmit distress, beginning with the clearing framework of Eisenberg and Noe~\cite{EisenbergNoe2001}. Dynamic interbank models also connect reserve fluctuations and lending decisions to systemic risk; see, for example, the mean-field game of Carmona, Fouque, and Sun~\cite{CarmonaFouqueSun2015}. Particularly close to our setting are absorbed mean-field models in which defaults impose losses on surviving institutions. Hambly, Ledger, and S\o jmark~\cite{HamblyLedgerSojmark2019} analyze the resulting positive feedback and blow-ups, while common noise and heterogeneous contagion are studied in \cite{HamblySojmark2019,FeinsteinSojmark2021}. These models make the dependence of future defaults on the transmission of current losses explicit.

Centralized intervention in contagious default systems has been studied through capital injections by Cuchiero, Reisinger, and Rigger~\cite{CuchieroReisingerRigger2024} and Hambly and Jettkant~\cite{HamblyJettkant2025}. Survival-oriented budget allocation is considered by Bayraktar, Guo, Tang, and Zhang~\cite{BayraktarGuoTangZhang2023}. Our planner instead chooses how to distribute the losses generated within the system. Each realized loss must be funded by the surviving population, and its distribution influences subsequent defaults. This isolates the allocation of endogenous losses as a control problem and identifies a rule that maximizes an entire class of terminal welfare measures.

\paragraph{Mean-field control and games.}
Our problem belongs to centralized control of McKean--Vlasov dynamics: the planner optimizes the population's welfare while accounting for the effect of intervention on its law. The distinction from mean-field games is developed by Carmona, Delarue, and Lachapelle~\cite{CarmonaDelarueLachapelle2013}; dynamic programming on the space of probability measures is established, among other works, by Pham and Wei~\cite{PhamWei2017}. Singular interventions have been studied in mean-field games by Fu and Horst~\cite{FuHorst2017}, in mean-field control with multidimensional singular controls by Denkert and Horst~\cite{DenkertHorst2025}, and under dynamic state-control-law constraints by Bo, Wang, and Yu~\cite{BoWangYu2026}. Endogenous reflection boundaries also arise in the interbank singular-control games of Cont, Guo, and Xu~\cite{ContGuoXu2021}.

The specific difficulty here comes from combining absorption with an exact balance between cumulative intervention and default mass. The default mass and the optimal support boundary are unstable under ordinary Wasserstein convergence. General existence or dynamic-programming results do not by themselves identify the cutoff optimizer or determine its free boundary. Our approach uses the order structure of the allocation problem to solve the discrete dynamics and then analyze the singular limit. It also identifies the same unique optimal control for every nonconstant payoff in the prescribed class.

\paragraph{Parabolic free-boundary problems.}
The connection between singular mean-field interactions and Stefan problems is central to the probabilistic analysis of default contagion. Delarue, Nadtochiy, and Shkolnikov~\cite{DelarueNadtochiyShkolnikov2022} establish global physical solutions, regularity, and uniqueness for the supercooled Stefan problem in the presence of blow-ups. The bailout model of \cite{CuchieroReisingerRigger2024} leads to a drift-controlled version of that problem. In our system, absorption occurs at a fixed endpoint and the moving endpoint is reflecting. The boundary condition $u(t,\Lambda_t)=\alpha\partial_xu(t,0)$ couples the reflecting density to the default flux at the opposite endpoint. Reflection in prescribed time-dependent intervals is studied by Burdzy, Kang, and Ramanan~\cite{BurdzyKangRamanan2009}; here in this paper, the boundary must itself be determined consistently with this coupling. The probabilistic and augmented call-function formulations allow us to establish well-posedness under Assumption~\ref{ass:initial-law} and subsequently recover smooth classical solutions at positive times.

\paragraph{Particle systems with selection.}
Extremal intervention also produces free boundaries in diffusing particle systems with selection. Carinci, De Masi, Giardin\`a, and Presutti~\cite{CarinciDeMasiGiardinaPresutti2014} construct generalized free-boundary solutions through mass-transport inequalities and barriers obtained by alternating diffusion with removal and reinjection. Berestycki, Brunet, Nolen, and Penington~\cite{BerestyckiBrunetNolenPenington2020} study the free-boundary problem associated with branching Brownian motion with spatial selection. Atar~\cite{Atar2025} develops a weak formulation in terms of a density and a selection measure, encoding extremal removal without assuming a regular interface. These works provide close precedents for the use of order, splitting approximations, and weak formulations to analyze moving support boundaries. Our cutoff transports the upper tail to a threshold while retaining its surviving mass, and the amount allocated is determined by defaults at the origin. Its limiting boundary therefore reflects surviving states and has a positive density trace. The approximation studied here is the vanishing allocation-time step.

\paragraph{Stochastic order and optimal allocation.}
The cutoff argument is also related to majorization and stochastic dominance; see Shaked and Shanthikumar~\cite{ShakedShanthikumar2007}. For a fixed loss budget, capping the largest reserves produces the maximal feasible law in increasing-concave order. The dynamic issue is that the allocation changes future defaults and hence future budgets. The comparisons in Sections~\ref{sec:discrete} and~\ref{sec:continuous} preserve the relevant order through this feedback, extending a static allocation principle to the full control problem.
\section{Discrete-Time Cutoff Control}
\label{sec:discrete}
In this section, we study the time-discretized version of the mean-field control problem. Let the time step $\Delta>0$ be such that $M:=T/\Delta$ is an integer, and write $t_n=n\Delta$. With the aim to de-couple the Brownian evolution of individual banks and the allocation control, we assume that the Brownian motion and the default (absorption at $0$) still evolves in continuous time, while losses generated by defaults during $(t_n,t_{n+1}]$ are deferred to and allocated at $t_{n+1}$. 

\subsection{Admissible controls and the law-valued state}

We first give a process-level formulation. An admissible system is defined on a filtered probability space $(\Omega,\mathcal F,(\mathcal F_t)_{0\leq t\leq T},\PP)$ satisfying the usual conditions. It carries an $\mathcal F_0$-measurable initial state $X_0$ with $\mathcal L(X_0)=\mu_0$ and an $(\mathcal F_t)$-Brownian motion $B$ independent of $X_0$.  The space may also carry independent auxiliary random variables, independent of $(X_0,B)$ and revealed only at allocation dates. These variables allow for relaxed, randomized allocation kernels. It turns out that, as we will see in Section \ref{subsec:DiscreteOptimal}, the cutoff optimizer constructed below does not use this randomization.

Suppose that $X_{t_n}$ is the post-allocation state at time $t_n$. On $[t_n,t_{n+1})$, let
\[
    \widehat\tau_{n+1}:=\inf\left\{t\in[t_n,t_{n+1}]:X_{t_n}+B_t -B_{t_n}\leq0\right\},
\]
with the convention that $\inf\varnothing=\infty$, and set
\[
    X_t:=
    \begin{cases}
    X_{t_n}+B_t-B_{t_n},&t<\widehat\tau_{n+1},\\
    0,&t\geq\widehat\tau_{n+1},
    \end{cases}
    \qquad t\in[t_n,t_{n+1}),
\]
and
\[
    \widehat X_{t_{n+1}}:=X_{t_{n+1}-}=
    \begin{cases}
    X_{t_n}+B_{t_{n+1}}-B_{t_n},
        &\widehat\tau_{n+1}>t_{n+1},\\
    0,  &\widehat\tau_{n+1}\leq t_{n+1}.
    \end{cases}
\]
An allocation $A_{n+1}$ is an $\mathcal F_{t_{n+1}}$-measurable random variable and satisfies
\begin{equation}
\label{eq:discrete-no-allocation-default}
    A_{n+1}=0\ \text{on }\{\widehat X_{t_{n+1}}=0\},\quad0\leq A_{n+1}<\widehat X_{t_{n+1}}\ \text{on }\{\widehat\tau_{n+1}>t_{n+1}\}.
\end{equation}
The post-allocation state is
\[
    X_{t_{n+1}}=\widehat X_{t_{n+1}}-A_{n+1}.
\]
After a bank defaults, it remains at zero and receives no later allocation. With
\[
    L_t:=\sum_{k=1}^M A_k\mathbf 1_{\{t_k\leq t\}},
\]
the cumulative allocated loss is thus adapted, nonnegative, non-decreasing, c\`adl\`ag, integrable, and satisfies $L_t=L_{t\wedge\tau}$, where
\[
    \tau:=\inf\{t\geq0:X_t=0\}.
\]

The aggregate-loss constraint is imposed as an equality at allocation dates:
\begin{equation}
\label{eq:discrete-exact-balance}
    \EE[L_{t_n}]=\alpha\PP[\tau\leq t_n],
    \qquad n=0,\ldots,M.
\end{equation}
Defaults occurring between two grid times are therefore outstanding until the
next allocation date.  At $t_0=0$, the standing assumption $\mu_0(\{0\})=0$ and the convention $L_0=0$ make \eqref{eq:discrete-exact-balance} automatic. Thus $\mu_0$ is already the
post-allocation law at the initial date.

This is a weak/open-loop admissible class: the probability space is part of the control, and allocations may use the information in $\mathcal F_{t_n}$. Because future Brownian increments are independent of $\mathcal F_{t_n}$ and the objective depends only on the terminal law, every one-step open-loop allocation can be represented by its conditional allocation kernel given the current pre-allocation state. Conversely, every such Borel kernel can be realized using the auxiliary randomization. Thus, there is no loss in using relaxed Markov law-feedback controls for the Bellman recursion below.

Let $\mu_n=\mathcal L(X_{t_n})$ and let $\widehat\mu_{n+1}=S_\Delta\mu_n$ be the pre-allocation law, where $S_t$ is the absorbed Brownian semigroup defined below. If $q_\mu:=\mu(\{0\})$ denotes the mass of $\mu$ at $x=0$, then \eqref{eq:discrete-exact-balance} is equivalent to
\begin{equation}
\label{eq:discrete-incremental-balance}
\EE[A_{n+1}]=c_\Delta(\mu_n):=\alpha\bigl(q_{S_\Delta\mu_n}-q_{\mu_n}\bigr).
\end{equation}

For a continuation problem restarted at $t_n$ from a post-allocation law $\mu$, the atom $q_\mu$ records defaults whose losses have already been funded. The future allocations therefore satisfy the shifted balance
\begin{equation}
\label{eq:discrete-shifted-balance}
\EE\left[\sum_{k=n+1}^j A_k\right]=\alpha\bigl(q_{\mu_j}-q_\mu\bigr),\quad j=n+1,\ldots,M,
\end{equation}
where $\mu_j=\mathcal L(X_{t_j})$. This is the continuation convention used in the Bellman recursion. For $n=0$, $\mu=\mu_0$, and $q_{\mu_0}=0$, it coincides with \eqref{eq:discrete-exact-balance}.

\subsection{Stochastic orders and the cutoff operator}

Throughout the paper,
\[
    \mathbb R_+:=[0,\infty),\quad\mathcal P_1(\mathbb R_+):=\left\{\mu\in\mathcal P(\mathbb R_+):\int_{\mathbb R_+}x\,\mu(\d x)<\infty\right\}.
\]
We write $\GaussCDF$ and $\GaussDensity$ for the standard Gaussian CDF and density, and $\GaussTail:=1-\GaussCDF$ for its upper tail. For $\mu\in\mathcal P_1(\mathbb R_+)$, we define the following objects, the meaning of which will become evident as we proceed with the analysis:
\begin{equation}
\label{eq:call-put-notation}
\begin{aligned}
    m_\mu&:=\int_{\mathbb R_+}x\,\mu(\d x),
    &q_\mu&:=\mu(\{0\}),
    &h_\mu&:=\alpha q_\mu,\\
    C_\mu(a)&:=\int_{\mathbb R_+}(x-a)^+\,\mu(\d x),
    &D_\mu(a)&:=\int_{\mathbb R_+}(x\wedge a)\,\mu(\d x),
    &H_\mu(a)&:=C_\mu(a)+h_\mu .
\end{aligned}
\end{equation}
Note that $D_\mu(a)=m_\mu-C_\mu(a)$.

\begin{definition}
\label{def:stochastic-orders}
For $\mu,\nu\in\mathcal P_1(\mathbb R_+)$:
\begin{enumerate}
\item $\mu\succsim_1\nu$ denotes first-order stochastic dominance:
\[
    \int f\,\d\mu\geq\int f\,\d\nu
\]
for every non-decreasing Borel function $f$ of at most linear growth.
\item $\mu\succsim_2\nu$ denotes second-order stochastic dominance:
\[
    \int f\,\d\mu\geq\int f\,\d\nu
\]
for every finite non-decreasing concave function $f:\mathbb R_+\to\mathbb R$ of at most linear growth. 
\end{enumerate}
\end{definition}

\begin{remark}
We briefly comment on the set of testing functions $f$ in item 2 of Definition~\ref{def:stochastic-orders}. By allowing a possible upward jump at $x=0$, we include the survival payoff function $\bone_{\{x>0\}}$. Indeed, every such $f$ has the decomposition
\begin{equation}
\label{eq:payoff-jump-decomposition}
    f(x)=f(0)+J_f\mathbf 1_{\{x>0\}}+\int_0^x g(s)\,\d s,\quad J_f:=f(0+)-f(0)\geq0,
\end{equation}
where $g$ is non-negative and non-increasing. Equivalently, for a right-continuous version of $g$, there are $\ell\geq0$ and a positive Stieltjes measure $\gamma$ such that
\begin{equation}
\label{eq:payoff-mixture-decomposition}
f(x)=f(0)+J_f\mathbf 1_{\{x>0\}}+\ell x+\int_{(0,\infty)}(x\wedge a)\,\gamma(\d a).
\end{equation}
It follows that
\begin{equation}
\label{eq:integrated-quantile-characterization}
\mu\succsim_2\nu\quad\Longleftrightarrow\quad D_\mu(a)\geq D_\nu(a)\quad\text{for every }a>0.
\end{equation}
The survival term can be recovered from
\[
\mathbf 1_{\{x>0\}}=\lim_{a\downarrow0}\frac{x\wedge a}{a}.
\]
In particular, $f=\mathbf 1_{(0,\infty)}$ corresponds to $J_f=1$ and $g=0$.
\end{remark}

For $0\leq c<m_\mu$, define the exact-budget action set
\begin{equation}
\label{eq:law-action-set}
\mathcal A(\mu,c):=\left\{\nu\in\mathcal P_1(\mathbb R_+):\mu\succsim_1\nu,\quad q_\nu=q_\mu,\quad m_\nu=m_\mu-c\right\}.
\end{equation}
The equality of the atoms at zero is the law-level form of the no-allocation-induced-default condition.

The cutoff operator has the following precise domain. Set $T_0\mu=\mu$. If $0<c<m_\mu$, the call function is continuous and strictly decreasing while positive, so there is a unique $\kappa_c(\mu)>0$ such that
\begin{equation}
\label{eq:cutoff-level}
    C_\mu(\kappa_c(\mu))=c.
\end{equation}
Define
\begin{equation}
\label{eq:tax-operator}
T_c\mu:=\bigl(x\mapsto x\wedge\kappa_c(\mu)\bigr)_\#\mu.
\end{equation}
We leave $T_c\mu$ undefined for $c\ge m_\mu$.

\begin{proposition}[Cutoff maximality]
\label{prop:DiscreteOptimal1}
Let $\mu\in\mathcal P_1(\mathbb R_+)$ and $0\leq c<m_\mu$. Then $T_c\mu\in\mathcal A(\mu,c)$. Moreover, $T_c\mu\succsim_2\nu$ for every $\nu\in\mathcal P_1(\mathbb R_+)$ satisfying $\mu\succsim_1\nu$ and $m_\nu\leq m_\mu-c$.
\end{proposition}

\begin{proof}
The case $c=0$ is trivial. Suppose henceforth that $0<c<m_\mu$. Under the pointwise coupling $X\mapsto X\wedge\kappa_c(\mu)$, the pre-allocation state is at least the post-allocation state, so it is clear that $\mu\succsim_1 T_c\mu$. It lowers the mean by $C_\mu(\kappa_c(\mu))=c$, hence belongs to $\mathcal A(\mu,c)$. For every $a\geq0$,
\begin{equation}
\label{eq:tax-call-identity}
    C_{T_c\mu}(a)=\bigl(C_\mu(a)-c\bigr)^+,
\end{equation}
and therefore
\begin{equation}
\label{eq:tax-stop-loss-identity}
    D_{T_c\mu}(a)=\min\{D_\mu(a),m_\mu-c\}.
\end{equation}
If $\mu\succsim_1\nu$, then the increasing function $x\mapsto x\wedge a$ gives $D_\mu(a)\geq D_\nu(a)$. Also $D_\nu(a)\leq m_\nu\leq m_\mu-c$. Hence,
\[
    D_{T_c\mu}(a)\geq D_\nu(a)\quad\forall a>0.
\]
The characterization \eqref{eq:integrated-quantile-characterization} proves the result. 
\end{proof}

For $t\geq0$, let $S_t\mu$ be the law at time $t$ of Brownian motion on $\mathbb R_+$ absorbed on hitting zero, with initial law $\mu$. Equivalently, 
\begin{equation*}
S_t\mu(\d x):=\int\PP\Big[y+B_t\in\d x,\inf_{s\in[0,t]}(y+B_s)>0\Big]\,\mu(\d y)+\int\PP\Big[\inf_{s\in[0,t]}(y+B_s)\leq0\Big]\,\mu(\d y)\cdot\delta_0(\d x).
\end{equation*}

\begin{lemma}[Order preservation]
\label{lem:order-preservation}
Let $t>0$.
\begin{enumerate}
\item If $\mu\succsim_1\nu$, then $S_t\mu\succsim_1S_t\nu$.
\item If $\mu\succsim_2\nu$, then $S_t\mu\succsim_2S_t\nu$.
\item If $\mu\succsim_1\nu$ and $0\leq c<\min\{m_\mu,m_\nu\}$, then $T_c\mu\succsim_1T_c\nu$.
\item If $\mu\succsim_2\nu$ and $0\leq c<\min\{m_\mu,m_\nu\}$, then $T_c\mu\succsim_2T_c\nu$.
\end{enumerate}
\end{lemma}

\begin{proof}
\emph{Absorbed semigroup and first-order dominance.}
If $x\geq y$, we couple the two Brownian trajectories $B^x$, $B^y$ with absorption at $0$ starting from $x$ and $y$.  Before the smaller process is absorbed, their difference is $x-y$, and after that time the smaller process is zero. Thus, $B_t^x\geq B_t^y$ for every $t$. Applying the monotone coupling characterization of first-order dominance proves the first claim.

\emph{Absorbed semigroup and second-order dominance.}
Let $S_tf(x):=\EE[f(B_t^x)]$. The preceding coupling shows that $S_tf$ is non-decreasing when $f$ is non-decreasing. It remains to prove concavity for the basis in \eqref{eq:payoff-mixture-decomposition}. Optional stopping gives
$S_t\mathrm{id}(x)=x$.  For the endpoint jump,
\[
    S_t\mathbf 1_{(0,\infty)}(x)=2\GaussCDF\left(\frac{x}{\sqrt t}\right)-1,\quad x>0,
\]
with value zero at $x=0$.  Its second derivative is
\[
    -\frac{2x}{t^{3/2}}\GaussDensity\left(\frac{x}{\sqrt t}\right)\leq0,\quad\forall x>0.
\]
The density of Brownian motion absorbed at $0$ is
\[
    p_t^0(x,y)=\frac1{\sqrt{2\pi t}}\left(e^{-(x-y)^2/(2t)}-e^{-(x+y)^2/(2t)}\right),\quad x,y>0.
\]
For $a>0$, set
\[
    v_a(t,x):=S_t\bigl((\,\cdot-a)^+\bigr)(x)=\int_a^\infty(y-a)p_t^0(x,y)\,\d y.
\]
Since $\partial_{xx}p_t^0=\partial_{yy}p_t^0$, integrating-by-part twice gives
\[
    \partial_{xx}v_a(t,x)=p_t^0(x,a)\geq0.
\]
As $S_t(x\wedge a)=x-v_a(t,x)$, this function is concave; it is non-decreasing by the pathwise coupling. The positive mixture \eqref{eq:payoff-mixture-decomposition} now shows that $S_tf$ is non-decreasing and concave. Therefore
\[
    \int f\,\d(S_t\mu)=\int S_tf\,\d\mu\geq\int S_tf\,\d\nu=\int f\,\d(S_t\nu),
\]
which proves the second claim.

\emph{Cutoff operator and first-order dominance.}
If $\mu\succsim_1\nu$, then $C_\mu(a)\geq C_\nu(a)$ for every $a$. Hence $\kappa_c(\mu)\geq\kappa_c(\nu)$. Couple $X\sim\mu$ and $Y\sim\nu$ so that $X\geq Y$ almost surely. Then
\[
    X\wedge\kappa_c(\mu)\geq Y\wedge\kappa_c(\nu)\quad\text{almost surely},
\]
which proves the third claim.

\emph{Cutoff operator and second-order dominance.}
First, we suppose that $m_\mu=m_\nu$. Then $\mu\succsim_2\nu$ is equivalent to $C_\mu(a)\leq C_\nu(a)$ for every $a$. Identity \eqref{eq:tax-call-identity} gives
\[
    C_{T_c\mu}(a)\leq C_{T_c\nu}(a),
\]
and the two post-tax means are equal, so $T_c\mu\succsim_2T_c\nu$.

For unequal means, we apply Lemma~\ref{lem:icv-decomposition}: there is $\eta\in\mathcal P_1(\mathbb R_+)$ such that $\eta\succsim_1\nu$ and that $\eta$ dominates $\mu$ in convex order. Equivalently, $m_\eta=m_\mu$ and $\mu\succsim_2\eta$. The equal-mean case and the third claim yield
\[
    T_c\mu\succsim_2 T_c\eta\succsim_1 T_c\nu,
\]
which proves the last claim.
\end{proof}

\begin{lemma}[Correct coupling direction for an allocation]
\label{lem:allocation-coupling}
Let $\mu,\nu\in\mathcal P_1(\mathbb R_+)$ and $0\leq c<m_\mu$.  Then $\nu\in\mathcal A(\mu,c)$ if and only if there is a coupling $(X,Y)$ such that
\[
    X\sim\mu,\quad Y\sim\nu,\quad X\geq Y\ \text{almost surely},\quad \EE[X-Y]=c,
\]
and $X=0$ if and only if $Y=0$ almost surely. Thus, $X$ is the pre-allocation state and the dominated variable $Y$ is the post-allocation state. The allocated loss $A=X-Y$ satisfies
\[
    A=0\ \text{on }\{X=0\},\quad 0\leq A<X\ \text{on }\{X>0\}.
\]
\end{lemma}

\begin{proof}
If $\nu\in\mathcal A(\mu,c)$, the monotone quantile coupling gives $X\geq Y$ almost surely. Since $q_\mu=q_\nu$, $Y=0$ must almost surely imply $X=0$, so $X=0$ if and only if $Y=0$. The mean identity gives $\EE[X-Y]=c$. The strict inequality $X-Y<X$ on $\{X>0\}$ follows from $Y>0$ there. Conversely, the stated coupling implies $\mu\succsim_1\nu$, equality of the atoms at zero, and $m_\nu=m_\mu-c$.  

\end{proof}

\subsection{Bellman recursion and optimality}\label{subsec:DiscreteOptimal}
For clarifying the comparison analyzed in this section, we define the total-resource functional and the order
\begin{equation}
\label{eq:augmented-order}
\mathcal R(\mu):=m_\mu+\alpha q_\mu=H_\mu(0),\qquad\mu\succsim_\alpha\nu\quad\Longleftrightarrow\quad\mu\succsim_2\nu\ \text{and}\,\,\mathcal R(\mu)\geq\mathcal R(\nu).
\end{equation}
For a general $\mu\in\cP(\RR_+)$ (treated as a post-allocation distribution), set
\begin{equation}
\label{eq:cutoff-transition}
c_\Delta(\mu):=\alpha(q_{S_\Delta\mu}-q_\mu),\quad\mathcal G_\Delta(\mu):=T_{c_\Delta(\mu)}S_\Delta\mu.
\end{equation}

\begin{lemma}[Exact-balance invariant and one-step order]
\label{lem:exact-balance-transition}
If $\mathcal R(\mu)>\alpha$, then $0<c_\Delta(\mu)<m_\mu=m_{S_\Delta\mu}$, so $\mathcal G_\Delta(\mu)$ is an admissible strict-domain cutoff update. It satisfies
\[
q_{\mathcal G_\Delta(\mu)}=q_{S_\Delta\mu},\quad\mathcal R(\mathcal G_\Delta(\mu))=\mathcal R(\mu).
\]
Every distribution in $\mathcal A(S_\Delta\mu,c_\Delta(\mu))$ has the same value of $\mathcal R$. Moreover, if both $\mathcal R(\mu)>\alpha$ and $\mathcal R(\nu)>\alpha$, then
\[
\mu\succsim_\alpha\nu\quad\Longrightarrow\quad\mathcal G_\Delta(\mu)\succsim_\alpha\mathcal G_\Delta(\nu).
\]
\end{lemma}

\begin{proof}
Absorbed Brownian motion preserves the mean by optional stopping. Note also that $\mathcal R(\mu)>\alpha$ implies $q_\mu<1$. From every $x>0$, Brownian motion starting at $x$ has strictly positive probability of hitting zero before time $\Delta$, so $q_{S_\Delta\mu}>q_\mu$ and $c_\Delta(\mu)>0$. Since also $q_{S_\Delta\mu}\leq1$,
\[
m_\mu-c_\Delta(\mu)=m_\mu+\alpha q_\mu-\alpha q_{S_\Delta\mu}= \mathcal R(\mu)-\alpha q_{S_\Delta\mu}\geq \mathcal R(\mu)-\alpha>0.
\]
Thus, the cutoff is in the strict domain and preserves the atom at zero. Exact balance gives the asserted resource identity. The same calculation applies to every distribution in the action set $\mathcal A(S_\Delta\mu,c_\Delta(\mu))$.

Now suppose $\mathcal R(\nu)>\alpha$ and $\mu\succsim_\alpha\nu$. Lemma~\ref{lem:order-preservation} gives
\[
    S_\Delta\mu\succsim_2S_\Delta\nu,\quad q_{S_\Delta\mu}\leq q_{S_\Delta\nu},
\]
where the second inequality uses the survival payoff $\mathbf 1_{(0,\infty)}$. By \eqref{eq:tax-stop-loss-identity},
\[
    D_{\mathcal G_\Delta(\mu)}(a)=\min\left\{D_{S_\Delta\mu}(a),\mathcal R(\mu)-\alpha q_{S_\Delta\mu}\right\}.
\]
Both entries in this minimum dominate the corresponding entries for $\nu$. Hence
\[
D_{\mathcal G_\Delta(\mu)}(a)\geq D_{\mathcal G_\Delta(\nu)}(a)\quad\text{for all }a>0.
\]
The characterization \eqref{eq:integrated-quantile-characterization} and resource invariance prove the claim.
\end{proof}

Let $f$ be a finite non-decreasing concave payoff of at most linear growth, with the endpoint convention in Definition~\ref{def:stochastic-orders}. For post-allocation laws with $\mathcal R(\mu)>\alpha$, we recursively define the value functions
\begin{equation}
\label{eq:discrete-bellman-recursion}
\begin{aligned}
    V_f^\Delta(T,\mu)&:=\int_{\mathbb R_+}f(x)\,\mu(\d x),\\
    V_f^\Delta(t_n,\mu)&:=\sup_{\rho\in\mathcal A(S_\Delta\mu,c_\Delta(\mu))}V_f^\Delta(t_{n+1},\rho),\quad n=M-1,\ldots,0.
\end{aligned}
\end{equation}
This is the Bellman recursion for the shifted weak/open-loop continuation problem \eqref{eq:discrete-shifted-balance}, after Markovization by conditional allocation kernels.

\medskip

\begin{theorem}[Discrete cutoff optimality]
\label{thm:discrete-cutoff-optimality}
On the domain
\[
    \mathfrak D_\alpha
    :=
    \{\mu\in\mathcal P_1(\mathbb R_+):\mathcal R(\mu)>\alpha\},
\]
the following assertions hold for every $n=0,\ldots,M$:
\begin{enumerate}
\item $V_f^\Delta(t_n,\cdot)$ is non-decreasing with respect to $\succsim_\alpha$;
\item the supremum in \eqref{eq:discrete-bellman-recursion} is attained and
\[
V_f^\Delta(t_n,\mu)=V_f^\Delta\bigl(t_{n+1},\mathcal G_\Delta(\mu)\bigr),\quad n<M;
\]
\item the deterministic cutoff law-feedback policy
\[
    \mu_{n+1}=\mathcal G_\Delta(\mu_n)
\]
is an optimizer for the shifted continuation problem. In particular, it is an optimizer for the original problem started from $\mu_0$.
\end{enumerate}
\end{theorem}

\begin{proof}
We use backward induction and prove the first two assertions simultaneously. At $n=M$, monotonicity follows from the definition of $\succsim_2$.

Assume that $V_{f,n+1}^\Delta$ is non-decreasing for $\succsim_\alpha$. If $\rho\in\mathcal A(S_\Delta\mu,c_\Delta(\mu))$, then Proposition~\ref{prop:DiscreteOptimal1} gives $\mathcal G_\Delta(\mu)\succsim_2\rho$. The two laws have the same mean and the same atom at zero, hence the same resource $\mathcal R$. Therefore, $\mathcal G_\Delta(\mu)\succsim_\alpha\rho$, and the induction hypothesis gives
\[
    V_f^\Delta\bigl(t_{n+1},\mathcal G_\Delta(\mu)\bigr)\geq V_f^\Delta(t_{n+1},\rho).
\]
Since the cutoff law itself belongs to the action set, it attains the Bellman supremum. If $\mu,\nu\in\mathfrak D_\alpha$ and $\mu\succsim_\alpha\nu$, Lemma~\ref{lem:exact-balance-transition} and the induction hypothesis yield
\[
\begin{aligned}
V_f^\Delta(t_n,\mu)=V_f^\Delta\bigl(t_{n+1},\mathcal G_\Delta(\mu)\bigr)\geq V_f^\Delta\bigl(t_{n+1},\mathcal G_\Delta(\nu)\bigr)=V_f^\Delta(t_n,\nu).
\end{aligned}
\]
This closes the induction.

Finally, on the strict domain $\mathfrak D_\alpha$,
\[
    \kappa_c(\mu)=\inf\{a\ge0:C_\mu(a)\leq c\}.
\]
The map $(\mu,a)\mapsto C_\mu(a)$ is continuous under the $W_1$ topology in $\mu$ and the usual topology in $a$, so the displayed formula makes $(\mu,c)\mapsto\kappa_c(\mu)$ Borel on the positive strict domain. Moreover, $\mu\mapsto S_\Delta\mu$ is $W_1$-continuous, $\mu\mapsto q_\mu$ is Borel, and
\[
    q_{S_\Delta\mu}=\int_{\mathbb R_+}\PP\left[\inf_{0\leq s\leq\Delta}(x+B_s)\leq0\right]\mu(\d x)
\]
is continuous because the hitting-probability kernel is bounded and continuous. Hence, $c_\Delta$, $\mathcal G_\Delta$, and $(x,\mu)\mapsto(x-\kappa_{c_\Delta(\mu)}(S_\Delta\mu))^+$ are Borel. The positivity established in Lemma~\ref{lem:exact-balance-transition} ensures that the positive-cutoff formula applies throughout $\mathfrak D_\alpha$. Iterating this selector realizes the cutoff law sequence and proves the genuine optimizer statement.

\end{proof}

\section{Continuous-Time Optimality via Time Discretization}
\label{sec:continuous}

\subsection{Continuous-time controls and the main result}
The continuous model uses the same exact aggregate-loss balance as the discrete model. We first define the process-level formulation of the weak/open-loop control problem. The compactness and identification needed to connect it to the discrete cutoff systems are addressed separately.

Fix a law $\mu_0\in\mathcal P_1(\mathbb R_+)$. As in the discrete case in Section \ref{sec:discrete}, an admissible system is defined on a filtered probability space $\bigl(\Omega,\mathcal F,(\mathcal F_t)_{0\leq t\leq T},\PP\bigr)$ satisfying the usual conditions. It carries an $\mathcal F_0$-measurable initial state $X_0$ with $\mathcal L(X_0)=\mu_0$ and an $(\mathcal F_t)$-Brownian motion $B$ independent of $X_0$. The space may also carry independent auxiliary random variables, independent of $(X_0,B)$, allowing potential randomness on the allocation control.

The cumulative allocation $L$ is $(\mathcal F_t)$-adapted, non-negative, non-decreasing, c\`adl\`ag, and satisfies
\begin{equation}
\label{eq:continuous-control-regularity}
L_{0-}=L_0=0,\quad\EE[L_T]<\infty.
\end{equation}
Define the unabsorbed reserve and default time by
\begin{equation}
\label{eq:continuous-unabsorbed-state}
\widehat X_t:=X_0+B_t-L_t,\quad\tau:=\inf\big\{t\in[0,T]:\widehat X_t\leq0\big\},
\end{equation}
with $\inf\varnothing=\infty$. The allocation is required to stop at default:
\begin{equation}
\label{eq:continuous-control-stopping}
L_t=L_{t\wedge\tau},\quad t\in[0,T].
\end{equation}
Allocation jumps may not default a surviving bank: almost surely, for every $t\in[0,T]$,
\begin{equation}
\label{eq:continuous-no-allocation-default}
\begin{aligned}
    \Delta L_t&=0,&&\text{on }\{t\leq\tau,\ \widehat X_{t-}=0\},\\
    0\leq\Delta L_t&<\widehat X_{t-},&&\text{on }\{t\leq\tau,\ \widehat X_{t-}>0\}.
\end{aligned}
\end{equation}
Here, the no-allocation-induced-default convention means precisely that an allocation jump cannot send a positive reserve to or through zero. Since Brownian paths are continuous and every downward jump of $\widehat X$ equals $-\Delta L$, condition \eqref{eq:continuous-no-allocation-default} implies $\widehat X_\tau=0$ on $\{\tau\leq T\}$.

The absorbed state is
\begin{equation}
\label{eq:continuous-absorbed-state}
    X_t:=
    \begin{cases}
        \widehat X_t,&t<\tau,\\
        0,&t\geq\tau.
    \end{cases}
\end{equation}
Equivalently,
\begin{equation}
\label{eq:continuous-stopped-dynamics}
    X_t=X_0+B_{t\wedge\tau}-L_{t\wedge\tau},
    \quad t\in[0,T],
\end{equation}
so the process remains at zero after default. We also note that $\tau$ also equals the first zero-hitting time of $X$.

The aggregate relation is imposed again as the exact balance
\begin{equation}
\label{eq:continuous-shifted-balance}
\EE[L_t]=\alpha\PP[\tau\leq t],\quad\forall t\in[0,T].
\end{equation}
In particular, the convention $L_{0-}=0$ and exact balance force $L_0=0$ in the original problem.

This is a weak/open-loop class: the filtered probability space is part of the control, and $L$ need not be Markov or law-feedback. Exact balance makes the current law sufficient for the resource bookkeeping. Indeed, if $\mu_t=\mathcal L(X_t)$, then optional stopping in
\eqref{eq:continuous-stopped-dynamics} gives
\begin{equation}
\label{eq:continuous-resource-invariant}
m_{\mu_t}+\alpha q_{\mu_t}=m_{\mu_0}+\alpha q_{\mu_0},\quad t\in[0,T].
\end{equation}

Let $\mathfrak A(\mu_0)$ denote the class of admissible system $(X,B,L,\tau)$ above. For a payoff $f$ as in item 2 of Definition~\ref{def:stochastic-orders}, define
\begin{equation}
\label{eq:continuous-value}
V_f(0,\mu):=\sup_{\mathfrak A(\mu)}\EE[f(X_T)].
\end{equation}
For the survival payoff, retain the abbreviation
\[
    V(0,\mu):=V_{\mathbf 1_{(0,\infty)}}(0,\mu).
\]

We can also restart the control problem at some time $r\in[0,T]$ as in Section \ref{sec:discrete}.


\begin{theorem}\label{thm:VDist_to_Cont_Conv}
Suppose that Assumption~\ref{ass:initial-law} holds. Then
\begin{align*}
V_f(0,\mu_0)=\lim_{\Delta\downarrow0}V_f^{\Delta}(0,\mu_0)
\end{align*}
for any payoff function $f$ as in item 2 of Definition \ref{def:stochastic-orders}.
\end{theorem}

The proof has two parts. First, we compare the discrete cutoff flow with an arbitrary continuous-time admissible control and obtain $V_f^\Delta\ge V_f$; this part is easier as, intuitively, delaying the default loss benefits the overall health of the system. For the reverse inequality, we realize the discrete cutoff flow by stopped reflection, establish positive-time compactness of the associated boundaries, and identify every subsequential limit as an admissible optimal continuous-time control.

\subsection{Discrete cutoff controls dominate continuous-time competitors}
Fix an arbitrary admissible continuous-time system $(\widetilde X,\widetilde B,\widetilde L,\widetilde\tau)\in\mathfrak A(\mu_0)$, and write $t_n:=n\Delta$, $\widetilde\mu_t:=\mathcal L(\widetilde X_t)$. The exact shifted balance and the stopped state equation imply the resource identity
\begin{equation}\label{eq:lemma32-resource-invariance}
R(\widetilde\mu_{t_n})=m_{\widetilde\mu_{t_n}}+\alpha q_{\widetilde\mu_{t_n}}=R(\mu_0)=:R_0,\quad n=0,\ldots,M.
\end{equation}
\begin{lemma}
The following one-step comparison holds:
\begin{equation}\label{eq:lemma32-one-step}
\mathcal G_\Delta(\widetilde\mu_{t_n})\succsim_\alpha\widetilde\mu_{t_{n+1}},\quad n=0,\ldots,M-1.
\end{equation}    
\end{lemma}
\begin{proof}
Fix \(n\), and, on the same probability space, switch off all allocation after \(t_n\). More precisely, define
\[
    \widehat\tau_{n+1}:=\inf\left\{s\in[t_n,t_{n+1}]:\widetilde X_{t_n}+\widetilde B_s-\widetilde B_{t_n}\le 0\right\},
\]
with the convention \(\inf\varnothing=\infty\), and set
\[
    \widehat X_{n+1}:=\bigl(\widetilde X_{t_n}+\widetilde B_{t_{n+1}}-\widetilde B_{t_n}\bigr)\mathbf 1_{\{\widehat\tau_{n+1}>t_{n+1}\}}.
\]
Because \(\widetilde X_{t_n}\) is \(\mathcal F_{t_n}\)-measurable and the future Brownian increments are independent of \(\mathcal F_{t_n}\), we see that $\mathcal L(\widehat X_{n+1})=S_\Delta\widetilde\mu_{t_n}$. We claim that
\begin{equation}\label{eq:lemma32-pathwise-comparison}
\widehat X_{n+1}\ge\widetilde X_{t_{n+1}}\quad\text{almost surely}.
\end{equation}
Indeed, on \(\{\widetilde\tau\le t_{n+1}\}\), the right-hand side is zero, so the claim is immediate. On \(\{\widetilde\tau>t_{n+1}\}\), for every \(s\in[t_n,t_{n+1}]\),
\[
    \widetilde X_s=\widetilde X_{t_n}+\widetilde B_s-\widetilde B_{t_n}-\bigl(\widetilde L_s-\widetilde L_{t_n}\bigr)>0.
\]
Since \(\widetilde L_s-\widetilde L_{t_n}\ge0\), the path $\widetilde X_{t_n}+\widetilde B_s-\widetilde B_{t_n}$ also remains positive throughout the interval. Thus, \(\widehat\tau_{n+1}>t_{n+1}\), and
\[
\widehat X_{n+1}=\widetilde X_{t_{n+1}}+\widetilde L_{t_{n+1}}-\widetilde L_{t_n}\ge\widetilde X_{t_{n+1}}.
\]
This proves \eqref{eq:lemma32-pathwise-comparison}, and consequently
\begin{equation}\label{eq:lemma32-first-order-comparison}
S_\Delta\widetilde\mu_{t_n}\succsim_1\widetilde\mu_{t_{n+1}}.
\end{equation}
In particular, applying this dominance to \(\mathbf 1_{(0,\infty)}\) gives
\begin{equation}\label{eq:lemma32-default-comparison}
q_{S_\Delta\widetilde\mu_{t_n}}\le q_{\widetilde\mu_{t_{n+1}}}.
\end{equation}
Let
\[ c_{n+1}^0:=c_\Delta(\widetilde\mu_{t_n})=\alpha\bigl(q_{S_\Delta\widetilde\mu_{t_n}}-q_{\widetilde\mu_{t_n}}\bigr)\ge0
\]
be the genuine discrete-time budget associated with \(\widetilde\mu_{t_n}\). Absorbed Brownian motion preserves the first moment, so $m_{S_\Delta\widetilde\mu_{t_n}}=m_{\widetilde\mu_{t_n}}$. Using \eqref{eq:lemma32-resource-invariance}, we obtain
\begin{align}
m_{S_\Delta\widetilde\mu_{t_n}}-c_{n+1}^0=R_0-\alpha q_{S_\Delta\widetilde\mu_{t_n}}\ge R_0-\alpha>0.\label{eq:lemma32-strict-domain}
\end{align}
Thus $0\le c_{n+1}^0<m_{S_\Delta\widetilde\mu_{t_n}}$, and $T_{c_{n+1}^0}S_\Delta\widetilde\mu_{t_n}\in\cA(\widetilde\mu_{t_n},c_{n+1}^0)$. Moreover, by \eqref{eq:lemma32-resource-invariance} and \eqref{eq:lemma32-default-comparison},
\begin{align}
m_{\widetilde\mu_{t_{n+1}}}&=R_0-\alpha q_{\widetilde\mu_{t_{n+1}}}\le R_0-\alpha q_{S_\Delta\widetilde\mu_{t_n}}=m_{S_\Delta\widetilde\mu_{t_n}}-c_{n+1}^0.\label{eq:lemma32-mean-comparison}
\end{align}
Applying Proposition~\ref{prop:DiscreteOptimal1} to the pre-allocation law \(S_\Delta\widetilde\mu_{t_n}\), the budget \(c_{n+1}^0\), and the comparison law \(\widetilde\mu_{t_{n+1}}\), using \eqref{eq:lemma32-first-order-comparison} and \eqref{eq:lemma32-mean-comparison}, yields
\[
\mathcal G_\Delta(\widetilde\mu_{t_n})=T_{c_{n+1}^0}S_\Delta\widetilde\mu_{t_n}\succsim_2\widetilde\mu_{t_{n+1}}.
\]
By Lemma~\ref{lem:exact-balance-transition} and \eqref{eq:lemma32-resource-invariance},
\[
R\bigl(\mathcal G_\Delta(\widetilde\mu_{t_n})\bigr)=R(\widetilde\mu_{t_n})=R(\widetilde\mu_{t_{n+1}})=R_0.
\]
Therefore, $\mathcal G_\Delta(\widetilde\mu_{t_n})\succsim_\alpha\widetilde\mu_{t_{n+1}}$, which concludes the proof of \eqref{eq:lemma32-one-step}.
\end{proof}

Now, let \((\mu_{t_n}^\Delta)_{n=0}^M\) denote the discrete-time optimally controlled flow:
\[
    \mu_0^\Delta:=\mu_0,\quad\mu_{t_{n+1}}^\Delta:=\mathcal G_\Delta(\mu_{t_n}^\Delta).
\]
\begin{lemma}
It holds that
\begin{equation}\label{eq:lemma32-inductive-comparison}
\mu_{t_n}^\Delta\succsim_\alpha\widetilde\mu_{t_n},\quad n=0,\ldots,M.
\end{equation}
\end{lemma}
\begin{proof}
This is equality when \(n=0\). Suppose it holds for some \(n<M\). Lemma~\ref{lem:exact-balance-transition} and \eqref{eq:lemma32-one-step} then give
\[
\mu_{t_{n+1}}^\Delta=\mathcal G_\Delta(\mu_{t_n}^\Delta)\succsim_\alpha \mathcal G_\Delta(\widetilde\mu_{t_n})\succsim_\alpha\widetilde\mu_{t_{n+1}}.
\]
The claim follows by induction. 
\end{proof}
In particular, $\mu_T^\Delta\succsim_2\widetilde\mu_T$. Since the continuous-time admissible system was arbitrary, we have thus established
\begin{corollary}
Let \(f\) be any payoff in Definition~\ref{def:stochastic-orders}, item 2. It holds that
\[
    V_f^\Delta(0,\mu)\ge V_f(0,\mu).
\]
In particular, taking \(f=\mathbf 1_{(0,\infty)}\) gives
\[
    V^\Delta(0,\mu)\ge V(0,\mu).
\]
\end{corollary}

\subsubsection{Pathwise realization of the discrete cutoff flow}
For clarity, we re-index and interpolate the optimally controlled probability flow in the $\Delta$-discretized problem to a continuous-time flow:
\begin{align*}
\mu_{(n+1)\Delta}^\Delta&:=\mathcal G_\Delta(\mu_{n\Delta}^\Delta),\quad n\ge0,\\
\mu_t^\Delta&:=S_{t-n\Delta}\mu_{n\Delta}^\Delta,\quad t\in[n\Delta,(n+1)\Delta).
\end{align*}
Then $V_f^\Delta(0,\mu_0)=\langle\mu_T^\Delta,f\rangle$. For $n\ge1$, let $\kappa_n^\Delta$ be the tax threshold corresponding to the cutoff operation $T_{c_n^\Delta}S_\Delta\mu_{(n-1)\Delta}^\Delta$, and set
\[
    \kappa_0^\Delta:=\max\operatorname{supp}\mu_0.
\]
Define the auxiliary boundary by
\[
    \Lambda_t^\Delta:=\kappa_n^\Delta,\quad t\in[n\Delta,(n+1)\Delta),
\]
with constant extension after $T$ whenever this is needed below.

We now give the stopped pathwise realization of this flow. Let $X_0\sim\mu_0$ and let $B$ be a Brownian motion independent of $X_0$. Define
\begin{align}
K_t^\Delta&:=\max_{j:\,j\Delta\le t}\bigl(X_0+B_{j\Delta}-\kappa_j^\Delta\bigr)^+,\label{eq:discrete-regulator-step-one}\\
\widehat X_t^\Delta&:=X_0+B_t-K_t^\Delta,\quad\tau^\Delta:=\inf\{t\in[0,T]:\widehat X_t^\Delta\le0\},\label{eq:discrete-unabsorbed-step-one}
\end{align}
where we follow the convention that $\inf\varnothing:=\infty$, and set
\begin{equation}\label{eq:discrete-stopped-step-one}
L_t^\Delta:=K_{t\wedge\tau^\Delta}^\Delta,
\quad X_t^\Delta:=
\begin{cases}
\widehat X_t^\Delta,&t<\tau^\Delta,\\
0,&t\ge\tau^\Delta.
\end{cases}
\end{equation}
It can be seen, jointly with the resource identity and the strict positivity of the cutoffs, that
\[
    \mathcal L(X_{n\Delta}^\Delta)=\mu_{n\Delta}^\Delta,\quad\mathcal L(X_t^\Delta)=\mu_t^\Delta,\quad t\in[n\Delta,(n+1)\Delta).
\]

\subsection{Tightness of discrete-time systems}
The goal of this subsection is to establish
\begin{proposition}\label{prop:boundary-precompactness}
Under Assumption \ref{ass:initial-law}, the family of auxiliary boundaries $\{\Lambda^\Delta:\Delta>0,\,T/\Delta\in\NN\}$ is pre-compact in $\cD((0,T+1])$ endowed with the local $M_1$ topology.    
\end{proposition}
\begin{remark}
\begin{enumerate}
    \item The extension of $\Lambda^\Delta$ beyond time $T$ and studying them in the space $\cD((0,T+1])$ instead of $\cD((0,T])$ get rid of the need to analyze the limit of $\Lambda_t^\Delta$ as $t\uparrow T$.
    \item In view of Proposition \ref{prop:FullInitialBoundaryTrace} in Section \ref{sec:continuous-2}, we might be able to upgrade the pre-compactness of tha family of auxiliary boundaries to the space $\cD([0,T+1])$. We do not need such a strengthend version as for now. 
\end{enumerate}    
\end{remark}
We start with proving some auxiliary technical results for the discrete-time systems.

\begin{lemma}\label{lem:cutoff-lower-bound}
Under Assumption \ref{ass:initial-law}, set $r_*:=m_0-\alpha>0$. It holds for all $\Delta>0$ and $t\ge0$ that $\Lambda_t^\Delta\ge r_*$.
\end{lemma}
\begin{proof}
Because every cutoff is strictly positive, default occurs by continuous crossing, and the stopped system satisfies
\[
    X_t^\Delta=X_0+B_{t\wedge\tau^\Delta}-L_t^\Delta,\quad 0\le t\le T.
\]    
Equations \eqref{eq:discrete-regulator-step-one}-\eqref{eq:discrete-stopped-step-one} imply that $X_{n\Delta}^\Delta\leq\kappa_n^\Delta$ for $n=0,1,...,M$. Taking expectation on both sides of the above equation for $t=n\Delta$, and using optional stopping, we obtain
\begin{align*}
\kappa_n^\Delta\ge\EE[X_{n\Delta}^\Delta]=\EE[X_0]-\EE[L_{n\Delta}^\Delta]= m_0-\alpha\PP[\tau^\Delta\leq n\Delta]\ge m_0-\alpha=r_*.    
\end{align*}
The desired inequality then readily follows from the definition of $\Lambda^\Delta$.
\end{proof}


We next prove that defaults occur at a uniformly positive rate on every positive-time interval, in the sense of the following lemma.

\begin{lemma}\label{lem:uniform-default-rate}
Under Assumption~\ref{ass:initial-law}, for every $\varepsilon\in(0,T)$ there exists a constants $\beta_\varepsilon>0$ such that
\[
 q_{n+1}^\Delta-q_n^\Delta\geq\beta_\varepsilon\Delta,
\]
for every admissible mesh $\Delta$ and every $n$ satisfying
\[
 \varepsilon\leq n\Delta<(n+1)\Delta\leq T.
\]
\end{lemma}

\begin{proof}
Choose $0<a<R<\infty$ such that $p_0:=\mu_0([a,R])>0$, and then fix $0<r<A<\min\{a,r_*\}$. Define 
\[
\sigma^\Delta:=\inf\{t\geq0:X_t^\Delta\leq r\},\quad G^\Delta:=\Big\{a\leq X_0\leq R,\ \sigma^\Delta\leq\frac\varepsilon2\Big\}.
\]
Writing $Z_t:=X_0+B_t$, the reflection principle gives, on $\{a\leq X_0\leq R\}$,
\[
 \PP\left[\inf_{0\leq t\leq\frac\varepsilon2}Z_t\leq r\,\middle|\,\mathcal F_0\right]=2\overline\Phi_G\left(\frac{X_0-r}{\sqrt{\varepsilon/2}}\right)\geq\rho_\varepsilon,
\]
where
\[
 \rho_\varepsilon:=2\overline\Phi_G\left(\frac{R-r}{\sqrt{\varepsilon/2}}\right)>0.
\]
Before default, $X_t^\Delta=Z_t-L_t^\Delta\leq Z_t$. Moreover, no allocation jump can cross the level $r$, since every nonzero allocation jump lands at a cutoff at least $r_*>r$. Thus, the controlled path reaches $r$ whenever the free path does, possibly earlier, and its first crossing of $r$ is continuous. Consequently,
\[
 \PP[G^\Delta]\geq p_0\rho_\varepsilon\quad\text{and}\quad X_{\sigma^\Delta}^\Delta=r\quad\text{on }G^\Delta.
\]

Let $p_A$ be the Dirichlet heat kernel on $(0,A)$ and set
\[
 g_{r,A}(u):=\frac12\partial_y p_A(u,r,0+),\quad u>0.
\]
That is, $g_{r,A}$ is the density of exit through zero for Brownian motion started from $r$ and killed upon leaving $(0,A)$. Positivity of $p_A$, parabolic boundary regularity, and the Hopf lemma imply that $g_{r,A}$ is continuous and strictly positive on $(0,\infty)$. Hence
\[
 \underline g_\varepsilon:=\min_{u\in[\varepsilon/2,T]}g_{r,A}(u)>0.
\]

Fix $n$ with $\varepsilon\leq n\Delta<(n+1)\Delta\leq T$. On the event $G^\Delta$,
\[
 [n\Delta-\sigma^\Delta,(n+1)\Delta-\sigma^\Delta]\subset[\varepsilon/2,T].
\]
After time $\sigma^\Delta$, restrict to Brownian paths that remain in $(0,A)$ until they exit through zero. Along such paths no allocation acts, since $A<r_*$ lies below every cutoff. The strong Markov property at $\sigma^\Delta$ (formally, at $\sigma^\Delta\wedge\frac\varepsilon2$) therefore shows that this killed-Brownian flux is a subflux of the actual defaults. Consequently,
\begin{align*}
q_{n+1}^\Delta-q_n^\Delta=\PP[n\Delta<\tau^\Delta\leq (n+1)\Delta]\geq\EE\left[\bone_{G^\Delta}\int_{n\Delta-\sigma^\Delta}^{(n+1)\Delta-\sigma^\Delta}g_{r,A}(u)\,du\right]\geq p_0\rho_\varepsilon\underline g_\varepsilon\,\Delta.
\end{align*}
The conclusion follows with $\beta_\varepsilon:=p_0\rho_\varepsilon\underline g_\varepsilon$.
\end{proof}

We next derive a one-sided modulus for the cutoff threshold in the discrete-time systems.
\begin{lemma}\label{lem:cutoff-modulus}
For $\gamma>0$, define
\[
\Omega_\gamma(h):=\inf_{\theta>0}\left\{\frac{\log(1+\theta/\gamma)}{2\theta}+\theta h\right\},\quad h\geq0.
\]
Let $\Delta$ be an admissible mesh and let $0\leq n<m\leq M$. If $c_{j+1}^\Delta\geq\gamma\Delta$, $j=n,\ldots,m-1$, then $\kappa_m^\Delta-\kappa_n^\Delta\leq\Omega_\gamma((m-n)\Delta)$. Moreover, $\Omega_\gamma$ is non-decreasing and $\Omega_\gamma(h)\to0$ as $h\downarrow0$. Consequently, under Assumption~\ref{ass:initial-law}, for every $\varepsilon\in(0,T)$,
\[
 \kappa_m^\Delta-\kappa_n^\Delta\leq\Omega_{\gamma_\varepsilon}((m-n)\Delta),\quad\varepsilon\leq n\Delta<m\Delta\leq T,
\]
where $\gamma_\varepsilon:=\alpha\beta_\varepsilon$ and $\beta_\varepsilon$ is given by Lemma \ref{lem:uniform-default-rate}.
\end{lemma}

\begin{proof}
For $j\geq0$, write
\[
 C_j^\Delta(x):=\int_{\mathbb R_+}(y-x)^+\,\mu_{j\Delta}^\Delta(dy),\quad x\in\mathbb R,
\]
and let $P_t$ denote the unrestricted Gaussian heat semigroup.

We first observe that
\[
C_{S_\Delta\mu_{j\Delta}^\Delta}(x)\leq P_\Delta C_j^\Delta(x),\quad x\in\mathbb R.
\]
Indeed, couple the absorbed and unrestricted Brownian motions. If $\tau_0$ is the hitting time of zero, then bounded optional sampling gives
\[
X+B_{\Delta\wedge\tau_0}=\EE\left[X+B_\Delta\mid\mathcal F_{\Delta\wedge\tau_0}\right].
\]
Jensen's inequality for $z\mapsto(z-x)^+$, followed by Brownian symmetry, proves the claim.

The cutoff identity also holds on the whole real line:
\[
C_{j+1}^\Delta(x)=\left(C_{S_\Delta\mu_{j\Delta}^\Delta}(x)-c_{j+1}^\Delta\right)^+,\quad x\in\mathbb R.
\]
For $x<\kappa_{j+1}^\Delta$, this follows from
\[
(\min\{y,\kappa_{j+1}^\Delta\}-x)^+=(y-x)^+-(y-\kappa_{j+1}^\Delta)^+,
\]
while for $x\geq\kappa_{j+1}^\Delta$ both sides vanish. Consequently, if $c_{j+1}^\Delta\geq\gamma\Delta$, then
\[
C_{j+1}^\Delta\leq\mathcal Q_{\gamma,\Delta}C_j^\Delta,\quad\mathcal Q_{\gamma,\Delta}f:=(P_\Delta f-\gamma\Delta)^+,
\]
pointwise on $\mathbb R$. Notice also that $\mathcal Q_{\gamma,\Delta}$ is order preserving. Since $\mu_{n\Delta}^\Delta$ is supported in
$[0,\kappa_n^\Delta]$,
\[
C_n^\Delta(x)\leq(\kappa_n^\Delta-x)^+,\qquad x\in\mathbb R.
\]
Fix $\theta>0$ and define
\[
F_\theta(y):=\frac{\gamma}{2\theta^2}\big(e^{2\theta y}-1-2\theta y\big),\quad b_\theta:=\frac{1}{2\theta}\log\left(1+\frac{\theta}{\gamma}\right).
\]
Then $F_\theta$ is nonnegative and convex,
\[
F_\theta(0)=F_\theta'(0)=0,\quad F_\theta'(b_\theta)=1,
\]
and hence
\[
F_\theta(y)\geq(y-b_\theta)^+,\quad y\geq0.
\]
Therefore,
\[
W_\theta(t,x):=F_\theta\left((\kappa_n^\Delta+b_\theta+\theta t-x)^+\right)
\]
satisfies
\[
C_n^\Delta(x)\leq W_\theta(0,x),\quad x\in\mathbb R.
\]

A direct calculation gives
\[
\left(\partial_t-\frac12\partial_{xx}\right)W_\theta=-\gamma\,\bone_{\{x<\kappa_n^\Delta+b_\theta+\theta t\}}\geq-\gamma
\]
in the distributional sense. There is no interface measure because $F_\theta(0)=F_\theta'(0)=0$. The heat-semigroup comparison therefore yields
\[
 \mathcal Q_{\gamma,\Delta}W_\theta(t,\cdot)\leq W_\theta(t+\Delta,\cdot).
\]
By monotonicity of $\mathcal Q_{\gamma,\Delta}$, induction gives
\[
C_{n+\ell}^\Delta(x)\leq W_\theta(\ell\Delta,x),\quad0\leq\ell\leq m-n.
\]

The right-hand side vanishes whenever $x\geq\kappa_n^\Delta+b_\theta+\theta\ell\Delta$. Hence,
\[
\mathrm{supp}\,\mu_{(n+\ell)\Delta}^\Delta\subset\left[0,\kappa_n^\Delta+b_\theta+\theta\ell\Delta\right].
\]
Since $c_m^\Delta\geq\gamma\Delta>0$, the cutoff operation at time $m\Delta$ places positive mass at $\kappa_m^\Delta$. Thus
\[
 \kappa_m^\Delta=\max\mathrm{supp}\,\mu_{m\Delta}^\Delta.
\]
Taking $\ell=m-n$ gives
\[
\kappa_m^\Delta-\kappa_n^\Delta\leq\frac{\log(1+\theta/\gamma)}{2\theta}+\theta(m-n)\Delta.
\]
Minimizing over $\theta>0$ proves the desired bound. Finally, $\Omega_\gamma$ is non-decreasing because each function inside the infimum is non-decreasing in $h$. For every fixed $\theta>0$,
\[
0\leq\limsup_{h\downarrow0}\Omega_\gamma(h)\leq\frac{\log(1+\theta/\gamma)}{2\theta}.
\]
Letting $\theta\to\infty$ proves $\Omega_\gamma(h)\to0$. For the final assertion, Lemma \ref{lem:uniform-default-rate} and the exact balance give
\[
c_{j+1}^\Delta=\alpha(q_{j+1}^\Delta-q_j^\Delta)\geq\alpha\beta_\varepsilon\Delta=\gamma_\varepsilon\Delta,\quad j=n,\ldots,m-1.
\]
The preceding deterministic comparison applies.
\end{proof}

\begin{proof}[Proof of Proposition~\ref{prop:boundary-precompactness}]
\emph{Compact containment.}
Fix $\eta\in(0,T]$ and consider sufficiently small admissible meshes such that $\Delta\leq\eta/16$. Set
\[
    p_\Delta:=\left\lceil\frac{\eta}{4\Delta}\right\rceil,\quad q_\Delta:=\left\lfloor\frac{\eta}{2\Delta}\right\rfloor.
\]
Then
\[
    p_\Delta\Delta\geq\frac{\eta}{4},\quad q_\Delta\Delta\leq\frac{\eta}{2},\quad(q_\Delta-p_\Delta)\Delta\geq\frac{\eta}{4}-2\Delta\geq\frac{\eta}{8}.
\]
Let
\[
    \underline{\kappa}^{\Delta}:=\min_{p_\Delta<j\leq q_\Delta}\kappa_j^\Delta,\quad\gamma:=\gamma_{\eta/4}.
\]
By Lemma~\ref{lem:uniform-default-rate} and the exact balance identity,
\[
\mathbb{E}\!\left[L_{q_\Delta\Delta}^\Delta-L_{p_\Delta\Delta}^\Delta\right]=\sum_{j=p_\Delta}^{q_\Delta-1}c_{j+1}^\Delta\geq\gamma(q_\Delta-p_\Delta)\Delta\geq\frac{\gamma\eta}{8}.
\]
Let
\[
    x_*:=\max\mathrm{supp}\,\mu_0,\quad Z_T:=x_*+\sup_{0\leq s\leq T}B_s.
\]
The pathwise regulator representation gives
\[
\begin{aligned}
L_{q_\Delta\Delta}^\Delta-L_{p_\Delta\Delta}^\Delta&\leq\mathbf{1}_{\{\tau^\Delta>p_\Delta\Delta\}}\left(X_{p_\Delta\Delta}^\Delta+\sup_{0\leq u\leq(q_\Delta-p_\Delta)\Delta}(B_{p_\Delta\Delta+u}-B_{p_\Delta\Delta})-\underline{\kappa}^\Delta\right)^+ \\
&\leq (Z_T-\underline{\kappa}^\Delta)^+ .
\end{aligned}
\]
Indeed, on the event $\{\tau^\Delta>p_\Delta\Delta\}$,
\[
    X_{p_\Delta\Delta}^\Delta+\sup_{0\leq u\leq(q_\Delta-p_\Delta)\Delta}(B_{p_\Delta\Delta+u}-B_{p_\Delta\Delta})=X_0-K_{p_\Delta\Delta}^\Delta+\sup_{p_\Delta\Delta\leq s\leq q_\Delta\Delta}B_s\leq Z_T,
\]
whereas the regulator increment vanishes on $\{\tau^\Delta\leq p_\Delta\Delta\}$. Since $Z_T\in L^1$, we can choose $R_\eta<\infty$ sufficiently large such that
\[
\mathbb{E}\big[(Z_T-R_\eta)^+\big]<\frac{\gamma\eta}{8}.
\]
The preceding estimates then imply $\underline{\kappa}^\Delta<R_\eta$. Now, choose
\[
k_\Delta\in\operatorname*{arg\,min}_{p_\Delta<j\leq q_\Delta}\kappa_j^\Delta.
\]
For every $m$ satisfying $\eta\leq m\Delta\leq T$, we have $k_\Delta\Delta\leq\eta/2<m\Delta$. Hence, Lemma~\ref{lem:cutoff-modulus} yields
\[
\kappa_m^\Delta\leq\kappa_{k_\Delta}^\Delta+\Omega_{\gamma_{\eta/4}}\big((m-k_\Delta)\Delta\big)\leq R_\eta+\Omega_{\gamma_{\eta/4}}(T).
\]
Consequently,
\begin{equation}\label{eq:cutoff-compact-containment}
\limsup_{\Delta\downarrow0}\sup_{\eta\leq m\Delta\leq T}\kappa_m^\Delta<\infty.
\end{equation}

\emph{Local $M_1$-precompactness.}
Fix $\varepsilon\in(0,T)$ and let $\Delta\leq\varepsilon/2$. For $\varepsilon\leq s\leq t\leq T$, set
\[
    i:=\left\lfloor\frac{s}{\Delta}\right\rfloor,\quad j:=\left\lfloor\frac{t}{\Delta}\right\rfloor.
\]
Then
\[
    i\Delta\geq\frac{\varepsilon}{2},\quad(j-i)\Delta\leq t-s+\Delta.
\]
Lemma~\ref{lem:cutoff-modulus} therefore gives
\begin{equation}\label{eq:cutoff-modulus-continuous-time}
\Lambda_t^\Delta-\Lambda_s^\Delta\leq\Omega_{\gamma_{\varepsilon/2}}(t-s+\Delta),\quad\varepsilon\leq s\leq t\leq T.
\end{equation}
Moreover, applying \eqref{eq:cutoff-compact-containment} with $\eta=\varepsilon/2$ gives
\begin{equation}\label{eq:cutoff-compact-containment-continuous-time}
\limsup_{\Delta\downarrow0}\sup_{\varepsilon\leq t\leq T}\Lambda_t^\Delta<\infty.    
\end{equation}

Extend each boundary constantly beyond $T$ by setting
\[
    \Lambda_t^\Delta:=\Lambda_T^\Delta,\quad T<t\leq T+1.
\]
For a scalar càdlàg path $x$, $[a,b]\subset(0,T+1]$, and $\delta>0$, the $M_1$ modulus is defined by (cf. \cite[Chapter 12]{whitt2002stochastic})
\[
    w_{M_1}(x;[a,b],\delta):=\sup_{\substack{a\leq t_1\leq t_2\leq t_3\leq b\\t_3-t_1\leq\delta}}
    \operatorname{dist}\bigl(x_{t_2},[x_{t_1},x_{t_3}]\bigr),
\]
where $[u,v]$ indeed means $[u\wedge v,u\vee v]$. The elementary inequality
\[
    \operatorname{dist}(y,[x,z])\leq (y-x)^+\vee(z-y)^+
\]
and \eqref{eq:cutoff-modulus-continuous-time} imply, for $a<T$ and all sufficiently small $\Delta$, that
\[
w_{M_1}\bigl(\Lambda^\Delta;[a,b],\delta\bigr)\leq\Omega_{\gamma_{a/2}}(\delta+\Delta).
\]
If $a\geq T$, the left-hand side vanishes. Therefore,
\[
\lim_{\delta\downarrow0}\limsup_{\Delta\downarrow0}w_{M_1}\bigl(\Lambda^\Delta;[a,b],\delta\bigr)=0.
\]


The local scalar $M_1$ compactness criterion \cite[Section 12.5]{whitt2002stochastic} now shows that the family $\{\Lambda^\Delta:\Delta>0,T/\Delta\in\NN\}$ as $\Delta_k\downarrow0$ is pre-compact in $M_1$ locally on $(0,T+1]$.
\end{proof}

\subsection{Limiting boundary and control in continuous time}

\begin{proof}[Proof of Theorem~\ref{thm:VDist_to_Cont_Conv}]
Fix \(f\) as in item 2 of Definition~\ref{def:stochastic-orders}, and choose admissible meshes \(\Delta_k\downarrow0\) such that
\begin{equation}\label{eq:limsup-sequence-value}
V_f^{\Delta_k}(0,\mu_0)\to\limsup_{\Delta\downarrow0}V_f^\Delta(0,\mu_0).
\end{equation}
Put \(M_k:=T/\Delta_k\), write \(\Lambda^k:=\Lambda^{\Delta_k}|_{[0,T]}\), set \(\Lambda_T^k:=\kappa_{M_k}^{\Delta_k}\), and denote its constant extension to \([0,T+1]\) by \(\overline\Lambda^k\).

\paragraph{Extraction of the boundary.}
Proposition~\ref{prop:boundary-precompactness} gives, along a subsequence, a deterministic \(\overline\Lambda\in\cD((0,T+1])\) such that
\[
\overline\Lambda^k\to\overline\Lambda\quad\text{locally in \(M_1\) on \((0,T+1]\)}.
\]
The limit is constant on \([T,T+1]\). Thus, with \(\Lambda:=\overline\Lambda|_{(0,T]}\), the standard extension--restriction property of \(M_1\) gives
\begin{equation}\label{eq:local-M1-boundary-convergence}
\Lambda^k\to\Lambda\quad\text{in \(M_1([\varepsilon,T])\)}
\end{equation}
whenever \(\varepsilon>0\) is a continuity point of \(\Lambda\). Moreover, Lemma~\ref{lem:cutoff-lower-bound} and passage to the limit give, with \(r_*:=m_0-\alpha>0\),
\begin{equation}\label{eq:limiting-boundary-lower-bound}
\Lambda_t^k\geq r_*\quad(0\leq t\leq T),\quad\Lambda_t\geq r_*\quad(0<t\leq T).
\end{equation}
Use a common \(X_0\sim\mu_0\) and Brownian motion \(B\), independent of \(X_0\), and set
\[
Y_t:=X_0+B_t,\quad\pi_k(t):=\Delta_k\Big\lfloor\frac{t}{\Delta_k}\Big\rfloor,\quad Y_t^k:=Y_{\pi_k(t)}.
\]
The realization in \eqref{eq:discrete-regulator-step-one}-- \eqref{eq:discrete-stopped-step-one} reads
\begin{align}
    K_t^k&:=\sup_{0\leq s\leq t}(Y_s^k-\Lambda_s^k)^+,
    &Z_t^k&:=Y_t-K_t^k,
    \label{eq:discrete-unstopped-regulator}\\
    \tau^k&:=\inf\{t\in[0,T]:Z_t^k\leq0\},
    &L_t^k&:=K_{t\wedge\tau^k}^k,
    \quad X_t^k:=Z_t^k\bone_{\{t<\tau^k\}},
    \label{eq:discrete-stopped-system}
\end{align}
where \(\inf\varnothing=\infty\). In particular, \(\mathcal L(X_t^k)=\mu_t^{\Delta_k}\). Exact balance at grid times and constancy between grid times yield
\begin{equation}\label{eq:discrete-exact-balance-all-times}
\EE[L_t^k]=\alpha\PP[\tau^k\leq\pi_k(t)],\quad 0\leq t\leq T.
\end{equation}

\paragraph{The initial layer and the reflection limit.}
For \(0<r<r_*\), every positive tax leaves the reserve at a cutoff at least \(r_*\), whereas between tax times the regulator is constant. Consequently,
\[
\{\tau^k\leq h\}\subset\{X_0\leq r\}\cup\{\operatorname{osc}_{[0,h]}B\geq r\}.
\]
The reflection principle therefore gives
\begin{equation}\label{eq:early-default-probability}
\sup_k\PP[\tau^k\leq h]\leq\rho_{\mu_0}(h):=\inf_{0<r<r_*}\Big\{\mu_0([0,r])+4\exp\Big(-\frac{r^2}{8h}\Big)\Big\},
\end{equation}
and \(\rho_{\mu_0}(h)\to0\) as \(h\downarrow0\), since \(\mu_0(\{0\})=0\). Let \(H_T:=X_0+\sup_{s\leq T}|B_s|\in L^1\). Then \(K_t^k\leq H_T\), and
\[
    K_h^k\leq L_h^k+H_T\bone_{\{\tau^k\leq h\}}.
\]
Together with \eqref{eq:discrete-exact-balance-all-times}, this gives
\begin{equation}\label{eq:initial-regulator-bound}
\sup_k\EE[K_h^k]\leq\alpha\rho_{\mu_0}(h)+\sup_{\PP(A)\leq\rho_{\mu_0}(h)}\EE[H_T\bone_A]=:\eta_{\mu_0}(h),
\end{equation}
where \(\eta_{\mu_0}(h)\to0\) as \(h\downarrow0\). Define
\begin{equation}\label{eq:limiting-unstopped-regulator}
K_0:=0,\quad K_t:=\sup_{0<s\leq t}(Y_s-\Lambda_s)^+,\quad 0<t\leq T.
\end{equation}
For a continuity point \(h>0\) of \(\Lambda\), set
\[
K_t^{k,h}:=\sup_{h\leq s\leq t}(Y_s^k-\Lambda_s^k)^+,\quad K_t^h:=\sup_{h\leq s\leq t}(Y_s-\Lambda_s)^+.
\]
Addition of a continuous path and the positive running-supremum map are continuous in scalar \(M_1\); for the latter, take running suprema of the spatial components of ordered parametric representations. Since \(Y^k\to Y\) uniformly almost surely,
\begin{equation}\label{eq:truncated-regulator-convergence}
K^{k,h}\to K^h\quad\text{almost surely in \(M_1([h,T])\)}.
\end{equation}
If \(0<r<h\) are continuity points of \(\Lambda\), then \(K_h^{k,r}\to K_h^r\) almost surely. Fatou's lemma, \eqref{eq:initial-regulator-bound}, and then monotone convergence as \(r\downarrow0\) show that $\EE[K_h]\leq\eta_{\mu_0}(h)$. Letting \(h\downarrow0\) through continuity points gives
\begin{equation}\label{eq:no-initial-impulse}
    K_{0+}=0\quad\text{almost surely}.
\end{equation}
Thus \(K\in\cD([0,T])\). Now, for continuity points \(0<h<\delta\) of \(\Lambda\),
\[
\begin{aligned}
d_{M_1,[\delta,T]}(K^k,K)&\leq K_h^k+d_{M_1,[\delta,T]}(K^{k,h},K^h)+K_h .
\end{aligned}
\]
Equations \eqref{eq:initial-regulator-bound} and \eqref{eq:truncated-regulator-convergence}, followed by \(h\downarrow0\), imply
\begin{equation}\label{eq:local-regulator-convergence}
K^k\to K\quad\text{in probability in \(M_1([\delta,T])\)}
\end{equation}
for every positive continuity point \(\delta\) of \(\Lambda\).

\paragraph{Absorption and exact balance.}
Set
\begin{equation}\label{eq:limiting-stopped-system}
Z_t:=Y_t-K_t,\quad\tau:=\inf\{t\in[0,T]:Z_t\leq0\},\quad L_t:=K_{t\wedge\tau},\quad X_t:=Z_t\bone_{\{t<\tau\}},
\end{equation}
again with \(\inf\varnothing=\infty\). From \eqref{eq:local-regulator-convergence},
\begin{equation}\label{eq:local-unabsorbed-state-convergence}
Z^k\to Z\quad\text{locally in \(M_1\) on \((0,T]\), in probability}.
\end{equation}
Moreover, \eqref{eq:no-initial-impulse} gives \(Z_{0+}=X_0>0\) almost surely. Hence \(\tau>0\) almost surely and \(\PP[\tau\leq h]\to0\) as \(h\downarrow0\). Since \(\Lambda\) is deterministic, \(K,Z\) are adapted and \(\tau\) is a stopping time. At every jump time of \(K\),
\begin{equation}\label{eq:limiting-reflection-jump}
\Delta K_t=(Z_{t-}-\Lambda_t)^+,\quad Z_t=Z_{t-}\wedge\Lambda_t.
\end{equation}
The measure \(dK\) is supported, with the appropriate one-sided value at jumps, on the contact set \(Z=\Lambda\). Hence \(K\) is constant on any interval on which \(Z_s\vee Z_{s-}<r_*\) for every \(s\). Since \(Z\) has no upward jumps and \eqref{eq:limiting-reflection-jump} excludes a regulator jump to or through zero,
\[
Z_{\tau-}=Z_\tau=0,\quad \Delta K_\tau=0\quad\text{on }\{\tau\leq T\}.
\]
On \(\{\tau<T\}\), strong Markov and the local constancy of \(K\) give
\begin{equation}\label{eq:strict-crossing-after-default}
    \inf_{\tau<s<(\tau+\varepsilon)\wedge T}Z_s<0
    \quad\text{for every }\varepsilon>0
\end{equation}
almost surely. The same local constancy shows that
\begin{equation}\label{eq:default-time-no-atoms}
    \PP[\tau=t]=0,\qquad 0<t\leq T.
\end{equation}
Indeed, on \(\{\tau=t\}\), \(K\) is constant on \([t-1/n,t]\) for some \(n>1/t\), so
\[
\{\tau=t\}\subset\bigcup_{n>1/t}\{Z_{t-1/n}+B_t-B_{t-1/n}=0\}.
\]
Each event on the right is null after conditioning on \(\mathcal F_{t-1/n}\).

We claim that, whenever \(t_k\to t\in(0,T]\),
\begin{equation}\label{eq:default-time-convergence}
\tau^k\wedge t_k\to\tau\wedge t\quad\text{and}\quad\bone_{\{\tau^k\leq t_k\}}\to\bone_{\{\tau\leq t\}}, \quad\text{in probability}.
\end{equation}
For completeness, fix a countable dense set \(\mathcal D\subset(0,T)\) of continuity points of \(\Lambda\). From every subsequence, a diagonal extraction in \eqref{eq:local-unabsorbed-state-convergence} gives almost-sure \(M_1\)-convergence on every \([a,T]\), \(a\in\mathcal D\). The completed graph of \(Z\) is uniformly separated from zero on compact intervals strictly before \(\tau\), because \(Z\) is positive there and has no upward jumps. This gives the lower bound for the hitting times; \eqref{eq:strict-crossing-after-default} gives the upper bound through a negative continuity point immediately after any hit before \(T\). If \(\tau=\infty\), the completed graph on \([a,T]\) is uniformly positive. All bracketing points here are chosen from \(\mathcal D\). Finally let \(a\downarrow0\), using \eqref{eq:early-default-probability} and \(\PP[\tau\leq a]\to0\). This proves convergence of the capped hitting times and hence the first line of \eqref{eq:default-time-convergence}. For the indicators, on \(\{\tau<t\}\) choose a negative continuity point \(s\in\mathcal D\cap(\tau,t)\); then \(\tau^k<t_k\) eventually. On \(\{\tau>t\}\), when \(t<T\), choose \(a,b\in\mathcal D\) with \(a<t<b<\tau\); completed-graph positivity on \([a,b]\), together with the initial-layer bound, excludes \(\tau^k\leq t_k\). When \(t=T\) and \(\tau=\infty\), use \([a,T]\) instead. Finally \eqref{eq:default-time-no-atoms} removes \(\{\tau=t\}\).

If non-decreasing \(k_n\to k\) in \(M_1\), \(s_n\to s\), and \(k\) is continuous at \(s\), then \(k_n(s_n)\to k(s)\). Applying this standard fact to \eqref{eq:local-regulator-convergence} and \eqref{eq:default-time-convergence}, noting that \(K\) is continuous at every continuity point of \(\Lambda\) and at \(\tau\), gives
\begin{equation}\label{eq:stopped-state-pointwise-convergence}
L_t^k\to L_t\quad\text{in probability}
\end{equation}
at every continuity point \(t\in(0,T)\) of \(\Lambda\). At time \(T\), use \(\tau^k\to\tau\) and moving-time evaluation on \(\{\tau<T\}\), and endpoint \(M_1\)-convergence on \(\{\tau>T\}\). Together with the indicator convergence, this gives
\begin{equation}\label{eq:terminal-state-and-default-convergence}
L_T^k\to L_T,\quad X_T^k\to X_T,\quad\bone_{\{\tau^k\leq T\}}\to\bone_{\{\tau\leq T\}}\quad\text{in probability},
\end{equation}
where \(\{\tau=T\}\) is null by \eqref{eq:default-time-no-atoms}. Because \(L_t^k\leq H_T\), the family is uniformly integrable. Therefore, at every continuity point \(t\in(0,T)\) of \(\Lambda\),
\[
\EE[L_t]=\lim_k\EE[L_t^k]=\alpha\lim_k\PP[\tau^k\leq\pi_k(t)]=\alpha\PP[\tau\leq t].
\]
The terminal convergence gives the same identity at \(T\). The two sides are right-continuous, the right-hand side is continuous on \((0,T]\) by \eqref{eq:default-time-no-atoms}, and the continuity points of \(\Lambda\) are dense. Hence, \begin{equation}\label{eq:limiting-exact-balance}
\EE[L_t]=\alpha\PP[\tau\leq t],\quad 0\leq t\leq T.
\end{equation}
The process \(L\) is adapted, non-negative, non-decreasing and c\`adl\`ag, is constant after \(\tau\), and satisfies
\[
L_0=0,\quad\EE[L_T]\leq\EE[H_T]<\infty,\quad X_t=X_0+B_{t\wedge\tau}-L_t.
\]
Here \(K_t\leq H_T\), and hence \(L_t\leq H_T\), by \eqref{eq:limiting-boundary-lower-bound} and \eqref{eq:limiting-unstopped-regulator}. The unabsorbed reserve \(X_0+B-L\) agrees with \(Z\) up to \(\tau\), so its first zero-hitting time is \(\tau\). Moreover, every nonzero allocation jump before default satisfies
\begin{equation}\label{eq:no-control-induced-default-limit}
\Delta L_t=X_{t-}-\Lambda_t\leq X_{t-}-r_*<X_{t-},
\end{equation}
while \(\Delta L_\tau=0\). Thus, \((X,B,L,\tau)\in\mathfrak A(\mu_0)\).

\paragraph{Terminal payoff and conclusion.}
To account for the possible jump of \(f\) at zero, write
\begin{equation}\label{eq:payoff-decomposition-at-zero}
J_f:=f(0+)-f(0),\quad\overline f(x):=f(x)-f(0)-J_f\bone_{(0,\infty)}(x).
\end{equation}
Then \(f=f(0)+J_f\bone_{(0,\infty)}+\overline f\), where \(\overline f\) is continuous and has at most linear growth. Since
\[
    0\leq X_T^k\leq(X_0+B_T)^+,
\]
and
\[
|\overline f(X_T^k)|\leq C\bigl(1+(X_0+B_T)^+\bigr),
\]
the family \(\{\overline f(X_T^k):k\geq1\}\) is uniformly integrable. Equation~\eqref{eq:terminal-state-and-default-convergence} therefore implies, using \(\bone_{\{X_T^k>0\}}=1-\bone_{\{\tau^k\leq T\}}\),
\begin{equation}\label{eq:terminal-payoff-convergence}
\EE[f(X_T^k)]\to\EE[f(X_T)].
\end{equation}
The discrete cutoff system is optimal, while the limiting system is admissible. Hence,
\begin{equation}\label{eq:limsup-value-upper-bound}
V_f(0,\mu_0)\geq\EE[f(X_T)]=\lim_{k\to\infty}V_f^{\Delta_k}(0,\mu_0)=\limsup_{\Delta\downarrow0}V_f^\Delta(0,\mu_0).
\end{equation}
Together with \(V_f^\Delta(0,\mu_0)\geq V_f(0,\mu_0)\), proved in the preceding comparison subsection, this yields
\begin{equation}\label{eq:discrete-continuous-value-convergence}
\lim_{\Delta\downarrow0}V_f^\Delta(0,\mu_0)=V_f(0,\mu_0),
\end{equation}
and completes the proof.
\end{proof}
\section{The Associated Free Boundary Problem}
\label{sec:continuous-2}
This section studies the probabilistic free boundary formulation satisfied by the limiting optimal flow constructed in Section~\ref{sec:continuous}. We introduce the probabilistic notion, establish its structural properties, and derive an augmented call-function formulation that yields stability and uniqueness. We then prove positive-time classical regularity, recovering the free boundary system~\eqref{eq:FBP1} from the probabilistic construction. The section concludes with uniqueness of the continuous-time optimal flow.

\subsection{Definition of probabilistic solutions}
\begin{definition}[Probabilistic solution with a local boundary]
\label{def:FBPProbSol}
Fix $\mu_0\in\mathcal P_1(\mathbb R_+)$ with $\mu_0(\{0\})=0$. Let $\Lambda\in D((0,T];(0,\infty))$ be deterministic, let $X_0\sim\mu_0$, and let $B$ be a Brownian motion independent of $X_0$. Set $K_0:=0$, $\widetilde X_0:=X_0$, and, for $t>0$, define
\[
K_t:=\sup_{0<s\le t}(X_0+B_s-\Lambda_s)^+,\quad\widetilde X_t:=X_0+B_t-K_t.
\]
Let $\tau:=\inf\{t\in[0,T]:\widetilde X_t\le0\}$ with the convention that $\inf\varnothing:=\infty$, and set
\[
L_t:=K_{t\wedge\tau},\quad X_t:=X_0+B_{t\wedge\tau}-L_t,\quad 0\le t\le T.
\]
We call $(X,\Lambda)$ a probabilistic solution to the free boundary problem~\eqref{eq:FBP1} if
\[
\mathbb E[L_t]=\alpha\mathbb P[\tau\le t]\,\quad 0\le t\le T.
\]
\end{definition}

\begin{lemma}[Classical solutions induce probabilistic solutions]\label{lem:ClassicalImpliesProbabilistic}
Suppose $(u,\Lambda)$ is a classical solution to the free boundary problem~\eqref{eq:FBP1} on $[0,T]$, with $\Lambda\in C^1([0,T];(0,\infty))$, and suppose that
\begin{align*}
u(0,\cdot)\geq0,\quad\operatorname{supp}u(0,\cdot)\subset[0,\Lambda_0],\quad\int_0^{\Lambda_0}u(0,x)\,\mathrm{d}x=1.
\end{align*}
Let $X_0$ be a random variable with density $u(0,\cdot)$, and let $B$ be a Brownian motion independent of $X_0$. Define
\begin{align*}
K_t&:=\sup_{s\in[0,t]}(X_0+B_s-\Lambda_s)^+,\quad\widetilde X_t:=X_0+B_t-K_t,\\
\tau&:=\inf\{t\in[0,T]:\,\widetilde X_t\leq0\},
\quad \inf\varnothing:=\infty,\\
L_t&:=K_{t\wedge\tau},\quad X_t:=
\begin{cases}
\widetilde X_t,&t<\tau,\\
0,&t\geq\tau.
\end{cases}
\end{align*}
Then $(X,\Lambda)$ is a probabilistic solution in the sense of Definition~\ref{def:FBPProbSol}. Moreover, for every $t\in[0,T]$,
\begin{align*}
\PP[X_t\in\mathrm{d}x]=\PP[\tau\leq t]\,\delta_0(\mathrm{d}x)+u(t,x)\mathbf{1}_{(0,\Lambda_t)}(x)\,\mathrm{d}x.
\end{align*}
\end{lemma}

\begin{proof}
Since $X_0\le\Lambda_0$ almost surely, including $s=0$ in the definition of $K$ does not change the regulator from Definition~\ref{def:FBPProbSol}. Continuity of $B$ and $\Lambda$ gives $K_{0+}=L_{0+}=0$. The running-supremum formula also shows that $K$ is continuous and non-decreasing, that $\widetilde X_t\le\Lambda_t$, and that $dK_t$ is carried by $\{\widetilde X_t=\Lambda_t\}$. Because $X_0>0$ almost surely, $\widetilde X_\tau=0$ on the event $\{\tau\le T\}$. Consequently,
\[
    X_t=\widetilde X_{t\wedge\tau}=X_0+B_{t\wedge\tau}-K_{t\wedge\tau}=X_0+B_{t\wedge\tau}-L_t.
\]

Fix $t\in(0,T]$. Let $g$ be smooth, vanish near zero, and be constant near $\Lambda_t$. Let $\psi$ solve the backward mixed-boundary problem
\[
\begin{cases}
\partial_s\psi+\frac12\partial_{xx}\psi=0,&0\le s<t,\quad 0<x<\Lambda_s,\\
\psi(s,0)=0,\qquad \partial_x\psi(s,\Lambda_s)=0,&0\le s\le t,\\
\psi(t,x)=g(x),&0\le x\le\Lambda_t.
\end{cases}
\]
Applying It\^o's formula to $s\mapsto\psi(s,\widetilde X_s)$ up to $t\wedge\tau$ gives
\[
\begin{aligned}
\psi(t\wedge\tau,\widetilde X_{t\wedge\tau})&=\psi(0,X_0)+\int_0^{t\wedge\tau}\left(\partial_s\psi+\frac12\partial_{xx}\psi\right)(s,\widetilde X_s)\,ds  \\
&\quad+\int_0^{t\wedge\tau}\partial_x\psi(s,\widetilde X_s)\,dB_s-\int_0^{t\wedge\tau}\partial_x\psi(s,\widetilde X_s)\,dK_s .
\end{aligned}
\]
The first integral vanishes by the backward equation. Since $dK_s$ is carried by $\{\widetilde X_s=\Lambda_s\}$,
\[
\int_0^{t\wedge\tau}\partial_x\psi(s,\widetilde X_s)\,dK_s=\int_0^{t\wedge\tau}\partial_x\psi(s,\Lambda_s)\,dK_s=0.
\]
Moreover, $\psi(t\wedge\tau,\widetilde X_{t\wedge\tau})=g(X_t)$, because it equals $g(\widetilde X_t)=g(X_t)$ on $\{\tau>t\}$, whereas on $\{\tau\le t\}$ both sides equal zero by $\widetilde X_\tau=X_t=0$ and $\psi(\tau,0)=g(0)=0$. Taking expectations therefore yields
\[
\mathbb E[g(X_t)]=\mathbb E[\psi(0,X_0)]=\int_0^{\Lambda_0}\psi(0,x)u(0,x)\,dx.
\]
On the other hand, Green's identity and the boundary conditions in
\eqref{eq:FBP1} give
\[
\frac{\d}{\d s}\int_0^{\Lambda_s}\psi(s,x)u(s,x)\,dx=\dot\Lambda_s\psi(s,\Lambda_s)u(s,\Lambda_s)+\frac12\bigl[\psi\,\partial_xu-(\partial_x\psi)u\bigr]_{x=0}^{x=\Lambda_s}=0.
\]
Therefore,
\[
\mathbb E[g(X_t)]=\int_0^{\Lambda_t}g(x)u(t,x)\,dx.
\]
Such functions $g$ are uniformly dense among continuous functions on $[0,\Lambda_t]$ that vanish at zero. Hence
\[
\mathbb P[X_t\in dx]=q_t\delta_0(dx)+u(t,x)\mathbf 1_{(0,\Lambda_t)}(x)\,dx,
\]
where
\[
q_t=\mathbb P[\tau\le t]=1-\int_0^{\Lambda_t}u(t,x)\,dx.
\]

It remains to verify the exact balance equation. Put
\[
    M_t:=\int_0^{\Lambda_t}u(t,x)\,dx,\quad m_t:=\int_0^{\Lambda_t}x\,u(t,x)\,dx.
\]
Leibniz' rule, the heat equation, and the first upper-boundary condition give
\[
\dot M_t=-\frac12\partial_xu(t,0),\quad\dot m_t=-\frac12u(t,\Lambda_t).
\]
Since $M_0=1$,
\[
    q_t=\frac12\int_0^t\partial_xu(s,0)\,ds.
\]
Furthermore, optional stopping and the law identity give
\[
\mathbb E[L_t]=\mathbb E[X_0]-\mathbb E[X_t]=m_0-m_t=\frac12\int_0^t u(s,\Lambda_s)\,ds.
\]
Using $u(s,\Lambda_s)=\alpha\partial_xu(s,0)$, we conclude that
\[
\mathbb E[L_t]=\frac{\alpha}{2}\int_0^t\partial_xu(s,0)\,ds=\alpha q_t=\alpha\mathbb P[\tau\le t].
\]
Therefore, $(X,\Lambda)$ is a probabilistic solution, and the asserted law formula follows.
\end{proof}

\subsection{Analytical properties of probabilistic solutions}
Throughout this subsection, let $(X,\Lambda)$ be a probabilistic solution on $[0,T]$ in the sense of Definition~\ref{def:FBPProbSol}, let
\[
\mu_t:=\mathcal L(X_t),\quad\tau:=\inf\{t>0:X_t\leq0\},
\]
and suppose that $\EE[X_0]>\alpha$. Under Assumption \ref{ass:initial-law}, $r_*:=\EE[X_0]-\alpha>0$. In a similar fashion to Lemma \ref{lem:cutoff-lower-bound}, we can prove
\begin{lemma}[Uniform positivity of the free boundary]
\label{lem:boundary-uniform-positivity}
Under Assumption \ref{ass:initial-law}, $\Lambda_t\ge r_*$ for all $t\in(0,T]$.
\end{lemma}

\begin{proof}
Optional stopping gives $\mathbb{E}[B_{t\wedge\tau}]=0$. Using the exact balance equation $\mathbb{E}[L_t]=\alpha\mathbb{P}[\tau\leq t]$, we obtain
\begin{equation}
\label{eq:mean-state-lower-bound}
\mathbb{E}[X_t]=\mathbb{E}[X_0]-\mathbb{E}[L_t]=\mathbb{E}[X_0]-\alpha\mathbb{P}[\tau\leq t]\geq r_*.
\end{equation}
On the other hand, the Skorokhod representation implies that $0\leq X_t\leq \Lambda_t$ a.s. Consequently,
\begin{equation}
\label{eq:boundary-uniform-lower-bound}
\Lambda_t\geq \mathbb{E}[X_t]\geq r_*.
\end{equation}
\end{proof}

\begin{lemma}[Interior reachability]\label{lem:interior-reachability}
Let $0<r<t\le T$, and let $J$ be a nonempty open interval such that $J\subset\big(0,\inf_{s\in[r,t]}\Lambda_s\big)$. For every $x\in(0,\Lambda_r]$, the reflected process restarted from $x$ at time $r$ satisfies
\[
    \mathbb P_{r,x}[X_t\in J,\ \tau>t]>0.
\]
Moreover, if $J\subset\big(0,\inf_{s\in[r,t)}\Lambda_s\big)$, then
\[
    \mathbb P_{r,x}[X_{t-}\in J,\ \tau\ge t]>0.
\]
\end{lemma}

\begin{proof}
Set
\[
h:=t-r,\quad W_u:=B_{r+u}-B_r,\quad\lambda(u):=\Lambda_{r+u},\quad b:=\inf_{0\le u\le h}\lambda(u)>0.
\]
For $\varphi\in D([0,h])$, write
\[
\Gamma_\lambda(\varphi)(u):=\varphi(u)-\sup_{0\le v\le u}\bigl(\varphi(v)-\lambda(v)\bigr)^+ .
\]
Thus, the restarted, unabsorbed reflected trajectory is $Y:=\Gamma_\lambda(x+W)$. For fixed $\lambda$, the Skorokhod map satisfies
\begin{equation}\label{eq:SkorokhodLipschitz}
\|\Gamma_\lambda(\varphi)-\Gamma_\lambda(\widetilde\varphi)\|_\infty\le2\|\varphi-\widetilde\varphi\|_\infty.                 
\end{equation}
Fix $z\in J$. For $\theta\in\mathbb R$, define
\[
\varphi^\theta(u):=x+\theta u,\quad y^\theta:=\Gamma_\lambda(\varphi^\theta),\quad F(\theta):=y^\theta(h).
\]
Estimate~\eqref{eq:SkorokhodLipschitz} shows that $F$ is continuous. Choose $\theta_-<0$ so that $x+\theta_-h<z$. Since the regulator is non-negative,
\[
    F(\theta_-)\le x+\theta_-h<z.
\]
On the other hand, for $\theta_+\ge0$ sufficiently large, $x+\theta_+h>z$. Since $\varphi^{\theta_+}$ is non-decreasing and $\lambda\ge b$,
\[
\begin{aligned}
F(\theta_+)\ge x+\theta_+h-\bigl(x+\theta_+h-b\bigr)^+=\min\{x+\theta_+h,b\}>z.
\end{aligned}
\]
The intermediate value theorem therefore gives $\theta_*\in (\theta_-,\theta_+)$ such that
\[
    y^{\theta_*}(h)=F(\theta_*)=z.
\]
The whole path $y^{\theta_*}$ stays strictly positive. Indeed, if $\theta_*<0$, then $\varphi^{\theta_*}$ is non-increasing while its regulator is non-decreasing; hence $y^{\theta_*}$ is non-increasing and
\[
    y^{\theta_*}(u)\ge y^{\theta_*}(h)=z>0.
\]
If $\theta_*\ge0$, then, for every $u\in[0,h]$,
\[
y^{\theta_*}(u)\ge x+\theta_*u-\bigl(x+\theta_*u-b\bigr)^+=\min\{x+\theta_*u,b\}\ge\min\{x,b\}>0.
\]
Consequently,
\[
    m:=\inf_{0\le u\le h}y^{\theta_*}(u)>0.
\]
Choose $\eta>0$ such that $2\eta<\min\{m,\operatorname{dist}(z,J^c)\}$. On the event
\[
A_\eta:=\left\{\sup_{0\le u\le h}|W_u-\theta_*u|<\eta\right\},
\]
estimate~\eqref{eq:SkorokhodLipschitz} gives $\|Y-y^{\theta_*}\|_\infty<2\eta$. Thus, $Y$ remains strictly positive and $Y_h\in J$. Therefore,
\[
    A_\eta\subseteq\{X_t\in J,\ \tau>t\}.
\]
Since Wiener measure has full support on $C_0([0,h])$, $\mathbb P_{r,x}[A_\eta]>0$, proving the first assertion.\\

For the second assertion, replace the terminal value of the shifted
boundary by setting
\[
\lambda^-(u):=
    \begin{cases}
        \Lambda_{r+u}, & 0\le u<h,\\
        \Lambda_{t-},  & u=h.
    \end{cases}
\]
Then
\[
\inf_{0\le u\le h}\lambda^-(u)=\inf_{s\in[r,t)}\Lambda_s,
\]
so the preceding argument applies to $\lambda^-$. Moreover, for every
continuous $\varphi$,
\[
\sup_{0\le v\le h}\bigl(\varphi(v)-\lambda^-(v)\bigr)^+=\sup_{0\le v<h}\bigl(\varphi(v)-\Lambda_{r+v}\bigr)^+,
\]
because the term at $v=h$ is the limit of the terms as $v\uparrow h$. Hence,
\[
\Gamma_{\lambda^-}(\varphi)(h)=\lim_{u\uparrow h}\Gamma_\lambda(\varphi)(u).
\]
For $\varphi=x+W$, this is precisely the left limit of the restarted reflected process at time $t$. On the positive-probability tube constructed above, the process remains positive on $[r,t)$, so $\tau\ge t$ and the terminal value is $X_{t-}\in J$. This proves the second assertion.
\end{proof}

\begin{lemma}[Absence of upward jumps]
\label{lem:no-upward-jumps}
For every $t\in(0,T)$, $\Lambda_t\leq\Lambda_{t-}$.
\end{lemma}

\begin{proof}
Fix $t_0\in(0,T)$ and suppose, towards a contradiction, that $\delta:=\Lambda_{t_0}-\Lambda_{t_0-}>0$. Choose
\[
    0<\eta<\min\left\{r_*,\frac{\delta}{2}\right\},\quad J:=\left(\frac{\eta}{2},\eta\right).
\]
Set $r:=t_0/2$. By \eqref{eq:mean-state-lower-bound}, $\mathbb P[X_r>0]>0$, while $J\subset\left(0,\inf_{s\in[r,t_0]}\Lambda_s\right)$. Conditioning on $X_r$ and applying Lemma~\ref{lem:interior-reachability} therefore gives
\[
\mathbb P[X_{t_0}\in J,\ \tau>t_0]=\mu_{t_0}(J)>0.
\]
By the right-continuity of $\Lambda$, there exists $h_0\in(0,T-t_0]$ such that
\begin{equation}
\label{eq:boundary-gap-after-upward-jump}
\Lambda_{t_0+u}\geq\Lambda_{t_0-}+\frac{\delta}{2}, \quad 0\leq u\leq h_0.
\end{equation}
An upward jump of the boundary does not increase the regulator, so $L_{t_0}=L_{t_0-}$. Moreover, $X_{t_0}\leq\Lambda_{t_0-}$ on $\{\tau>t_0\}$. Let $W_u:=B_{t_0+u}-B_{t_0}$. For $0<h\leq h_0$, the Skorokhod representation and \eqref{eq:boundary-gap-after-upward-jump} give
\[
L_{t_0+h}-L_{t_0}\leq\mathbf{1}_{\{\tau>t_0\}}\sup_{0\leq u\leq h} \bigl(X_{t_0}+W_u-\Lambda_{t_0+u}\bigr)^+\leq\left(\sup_{0\leq u\leq h}W_u-\frac{\delta}{2}\right)^+.
\]
Set $M_h:=\sup_{0\leq u\leq h}W_u$. Since $M_h\overset{d}{=}\sqrt h\,|N(0,1)|$, the reflection principle gives
\begin{equation}
\label{eq:loss-upper-bound-after-upward-jump}
\mathbb{E}\big[L_{t_0+h}-L_{t_0}\big]\leq\mathbb E\left[M_h\mathbf 1_{\{M_h>\delta/2\}}\right]\leq C\sqrt h\exp\left(-\frac{\delta^2}{8h}\right)
\end{equation}
for all sufficiently small $h>0$.

Recall that $\GaussTail$ denotes the standard Gaussian upper tail. Conditional on $X_{t_0}=x>0$, the event $\left\{\inf_{0\leq u\leq h}(x+W_u)\leq0\right\}$ forces default by time $t_0+h$, since the control can only decrease the state. Hence, the reflection principle and Mills' inequality yield
\begin{align}
\mathbb P[t_0<\tau\leq t_0+h]\geq\int_J2\GaussTail\left(\frac{x}{\sqrt h}\right)\mu_{t_0}(dx)\geq2\mu_{t_0}(J)\cdot\GaussTail\left(\frac{\eta}{\sqrt h}\right)\geq c\sqrt h\exp\left(-\frac{\eta^2}{2h}\right)
\label{eq:default-lower-bound-after-upward-jump}
\end{align}
for all sufficiently small $h>0$ and some $c>0$. Finally, exact balance gives
\begin{equation}
\label{eq:loss-default-increment-identity}
\mathbb{E}\bigl[L_{t_0+h}-L_{t_0}\bigr]=\alpha\mathbb{P}[t_0<\tau\leq t_0+h].
\end{equation}
Combining \eqref{eq:loss-upper-bound-after-upward-jump}--\eqref{eq:loss-default-increment-identity} gives
\[
\alpha c\exp\left(\frac{\delta^2/8-\eta^2/2}{h}\right)\leq C
\]
for all sufficiently small $h>0$, which is impossible because $\eta<\delta/2$. Therefore, $\Lambda_{t_0}\leq\Lambda_{t_0-}$.
\end{proof}

\begin{lemma}[Absence of downward jumps]
\label{lem:no-downward-jumps}
For every $t\in(0,T]$, $\Lambda_t\geq\Lambda_{t-}$.
\end{lemma}
\begin{proof}
The running-supremum representation shows that, before absorption, the continuous part of $dL$ is carried by the contact set $\{X=\Lambda\}$. Moreover, for every $s\in(0,T]$,
\begin{equation}
\label{eq:regulator-jump-formula}
\Delta L_s=\bigl(X_{s-}-\Lambda_s\bigr)^+,\quad X_s=X_{s-}-\Delta L_s=X_{s-}\wedge\Lambda_s.
\end{equation}
By Lemma~\ref{lem:boundary-uniform-positivity}, every increase of $L$ leaves the reflected state at least $r_*>0$. Thus, the regulator cannot cause default, and the continuity of $B$ gives $X_{t-}=X_t=0$ on $\{\tau=t\}$. Fix $t\in(0,T]$. For every integer $n>1/t$, set $r_n:=t-1/n$ and define
\[
A_n:=\left\{\tau=t,\sup_{s\in[r_n,t]}X_s<\frac{r_*}{2}\right\}.
\]
Then $\{\tau=t\}=\bigcup_{n>1/t}A_n$. On $A_n$, the support property of $dL^c$ and the jump formula \eqref{eq:regulator-jump-formula} imply that $L$ is constant on $[r_n,t]$, and therefore $0=X_t=X_{r_n}+B_t-B_{r_n}$. Conditionally on $\mathcal F_{r_n}$, the increment $B_t-B_{r_n}$ has a density, so $\mathbb P[A_n]=0$. Taking the countable union gives
\begin{equation}
\label{eq:default-time-no-atoms-2}
\mathbb{P}[\tau=t]=0,\quad t\in(0,T].
\end{equation}
Since $L$ is non-decreasing, monotone convergence and the defining balance relation give
\[
\mathbb{E}[L_{t-}]=\lim_{s\uparrow t}\mathbb{E}[L_s]=\alpha\lim_{s\uparrow t}\mathbb{P}[\tau\leq s]=\alpha\mathbb{P}[\tau<t].
\]
Consequently, by \eqref{eq:default-time-no-atoms-2},
\begin{equation}
\label{eq:expected-regulator-jump-zero}
\mathbb{E}[\Delta L_t]=\alpha\mathbb{P}[\tau=t]=0.
\end{equation}
Since $\Delta L_t\geq0$, it follows that $\Delta L_t=0$ almost surely. Suppose now, toward a contradiction, that $\delta:=\Lambda_{t-}-\Lambda_t>0$. Choose
\[
    J:=\left(\Lambda_t+\frac{\delta}{3},\Lambda_t+\frac{2\delta}{3}\right)\subset(\Lambda_t,\Lambda_{t-}).
\]
Since $\Lambda$ has a left limit at $t$, there exists $r<t$, sufficiently close to $t$, such that
\[
    \inf_{q\in[r,t)}\Lambda_q>\sup J.
\]
Moreover, \eqref{eq:mean-state-lower-bound} implies $\mathbb{P}[X_r>0]>0$. Applying the left-limit assertion of Lemma~\ref{lem:interior-reachability} conditionally on $X_r$ yields
\begin{equation}
\label{eq:positive-mass-above-post-jump-boundary}
\mathbb{P}\big[X_{t-}\in J,\ \tau\geq t\big]>0.
\end{equation}
On this event, \eqref{eq:regulator-jump-formula} gives $\Delta L_t>\delta/3$, contradicting $\Delta L_t=0$ almost surely. Therefore, $\Lambda_t\geq\Lambda_{t-}$.
\end{proof}

\begin{remark}[Canonical terminal normalization]
Lemma~\ref{lem:no-downward-jumps} gives $\Lambda_T\geq\Lambda_{T-}$. Moreover, by the continuity of $B$ and the existence of the left limit of $\Lambda$,
\[
(X_0+B_T-\Lambda_{T-})^+=\lim_{s\uparrow T}(X_0+B_s-\Lambda_s)^+\leq K_{T-}.
\]
Thus, neither the original terminal value $\Lambda_T$ nor its replacement by $\Lambda_{T-}$ contributes a new maximum to the running-supremum formula for $K_T$. Consequently, replacing $\Lambda_T$ by $\Lambda_{T-}$ leaves $K$, $\widetilde X$, $\tau$, $L$, and $X$ unchanged, and hence preserves the probabilistic-solution property. We henceforth adopt the canonical normalization
\[
    \Lambda_T:=\Lambda_{T-}.
\]
Together with Lemmas~\ref{lem:no-upward-jumps} and \ref{lem:no-downward-jumps}, this implies that $\Lambda$ is actually continuous on $(0,T]$.
\end{remark}

The Skorokhod representation implies that $0\leq X_t\leq \Lambda_t$ a.s. Consequently, $\operatorname{supp}(\mu_t)\subseteq[0,\Lambda_t]$. Now we establish the inverse inclusion.
\begin{lemma}[Identification of the support]
\label{lem:support-identification}
For every $t>0$, $\operatorname{supp}(\mu_t)=[0,\Lambda_t]$.
\end{lemma}
\begin{proof}
Fix $t>0$ and $y\in(0,\Lambda_t)$. Let $G$ be an arbitrary open neighborhood of $y$. Choose an open interval $J$ such that $y\in J$, $\overline{J}\subset G\cap(0,\Lambda_t)$. By the continuity of $\Lambda$, there exists $s<t$, sufficiently close to $t$, such that $\inf_{q\in[s,t]}\Lambda_q>\sup J$. By \eqref{eq:mean-state-lower-bound}, $\mathbb{P}[X_s>0]>0$. Applying Lemma~\ref{lem:interior-reachability} conditionally on $X_s$ gives
\[
    \mathbb{P}[X_t\in J,\, \tau>t]=\int_{(0,\Lambda_s]}\mathbb{P}_{s,x}\big[X_t\in J,\ \tau>t\big]\,\mu_s(\d x)>0.
\]
Therefore, $\mu_t(G)\geq\mu_t(J)>0$. Since $G$ is an arbitrary neighborhood of $y$, we conclude that $(0,\Lambda_t)\subseteq\operatorname{supp}(\mu_t)$. The support is closed, and $\Lambda_t\geq r_*>0$. Hence, it also contains both endpoints, i.e. $[0,\Lambda_t]\subset\operatorname{supp}(\mu_t)$. Together with the observation made before the lemma, this proves the claim.
\end{proof}

The preceding results determine the boundary and the support at every positive time, but they do not yet identify the behavior of the boundary at time zero.  We first identify its lower trace using only the stopped path representation and exact balance.

\begin{lemma}[Initial lower trace]
\label{lem:InitialBoundaryLowerTrace}
Let $(X,\Lambda)$ be a probabilistic solution with initial law $\mu_0$ satisfying Assumption~\ref{ass:initial-law}, and set $a:=\inf\operatorname{supp}(\mu_0)$, $b:=\max\operatorname{supp}(\mu_0)$. Then
\begin{equation}\label{eq:InitialBoundaryLowerTrace}
\liminf_{t\downarrow0}\Lambda_t=a+b.
\end{equation}
\end{lemma}

\begin{proof}
Write
\[
q_t:=\PP[\tau\leq t],\quad\underline\Lambda_t:=\inf_{0<s\leq t}\Lambda_s,\quad\ell:=\lim_{t\downarrow0}\underline\Lambda_t=\liminf_{t\downarrow0}\Lambda_t.
\]
We first show that $\ell\geq b$. Since $L_0=0$, the exact balance equation at
$t=0$ gives $\PP[\tau=0]=0$. Hence, $q_t\downarrow0$ and
\[
    \EE[L_t]=\alpha q_t\longrightarrow0\quad\text{as }t\downarrow0.
\]
Since $L$ is non-decreasing, it follows that $L_{0+}=0$ almost surely.
Moreover, on $\{\tau>0\}$, one has $L_t=K_t$ for all sufficiently small
$t$, and therefore $K_{0+}=0$ almost surely. If $\ell<\infty$, choose $s_n\downarrow0$ such that $\Lambda_{s_n}\to\ell$. Then
\[
0\leq (X_0+B_{s_n}-\Lambda_{s_n})^+\leq K_{s_n}\longrightarrow0\quad\text{almost surely}.
\]
Since $B_{s_n}\to0$, this gives $X_0\leq\ell$ almost surely. Thus $b=\operatorname*{ess\,sup}X_0\leq\ell$; the case $\ell=\infty$ is immediate. Consequently,
\begin{equation}\label{eq:BoundaryRunningInfimumAboveInitialSupport}
\ell=\liminf_{t\downarrow0}\Lambda_t\geq b.
\end{equation}
We next determine the small-time logarithmic rate of $q_t$. For every $\varepsilon>0$,
\[
    p_\varepsilon:=\mu_0([a,a+\varepsilon])>0.
\]
Since allocation can only decrease the state, the reflection principle gives
\[
q_t\geq\PP\!\left[\inf_{0\leq s\leq t}(X_0+B_s)\leq0\right]=2\int_{\mathbb R_+}\GaussTail\!\left(\frac{x}{\sqrt t}\right)\mu_0(dx)\geq2p_\varepsilon\GaussTail\!\left(\frac{a+\varepsilon}{\sqrt t}\right).
\]
Gaussian tail asymptotics, followed by $\varepsilon\downarrow0$, yield
\begin{equation}\label{eq:DefaultLogarithmicLowerBound}
    \liminf_{t\downarrow0}t\log q_t\geq-\frac{a^2}{2}.
\end{equation}
For the converse estimate, set
\[
    \operatorname{osc}_{[0,t]}^-B:=\sup_{0\leq r\leq s\leq t}(B_r-B_s).
\]
Since $\Lambda$, and hence $K$, is continuous on $(0,T]$, default is attained continuously. Thus
\[
    X_0+B_\tau=K_\tau\quad\text{on }\{\tau\leq T\}.
\]
If $K_\tau=0$, then $B_\tau=-X_0\leq-a$. If $K_\tau>0$, choose $r_n\in(0,\tau]$ such that
\[
    X_0+B_{r_n}-\Lambda_{r_n}\longrightarrow K_\tau.
\]
Then
\[
B_{r_n}-B_\tau=X_0+B_{r_n}-\Lambda_{r_n}-K_\tau+\Lambda_{r_n}\geq\underline\Lambda_t+o(1),
\]
and therefore $\operatorname{osc}_{[0,t]}^-B\geq\underline\Lambda_t$. Consequently,
\begin{equation}\label{eq:DefaultDrawdownUpperBound}
q_t\leq\PP\!\left[-\inf_{0\leq s\leq t}B_s\geq a\right]+\PP\!\left[\operatorname{osc}_{[0,t]}^-B\geq\underline\Lambda_t\right].
\end{equation}
If $M_s:=\sup_{0\leq r\leq s}B_r$, then
\[
    \operatorname{osc}_{[0,t]}^-B=\sup_{0\leq s\leq t}(M_s-B_s).
\]
By L\'evy's identity, $(M_s-B_s)_{s\geq0}$ has the law of reflected Brownian motion $(|B_s|)_{s\geq0}$. Hence, for every $d>0$,
\[
\PP[\operatorname{osc}_{[0,t]}^-B\geq d]\leq4\GaussTail\!\left(\frac d{\sqrt t}\right).
\]
If $a>0$, fix $d\in(0,a)$. By \eqref{eq:BoundaryRunningInfimumAboveInitialSupport}, $\underline\Lambda_t\geq d$ for all sufficiently small $t$. Therefore, \eqref{eq:DefaultDrawdownUpperBound} and Gaussian tail asymptotics give
\[
    \limsup_{t\downarrow0}t\log q_t\leq-\frac{d^2}{2}.
\]
Letting $d\uparrow a$ gives the upper bound matching \eqref{eq:DefaultLogarithmicLowerBound}. If $a=0$, the matching upper bound follows from $q_t\leq1$. Thus
\begin{equation}\label{eq:DefaultLogarithmicRate}
\lim_{t\downarrow0}t\log q_t=-\frac{a^2}{2}.
\end{equation}

We next prove
\begin{equation}\label{eq:InitialTraceLowerInequality}
    \liminf_{t\downarrow0}\Lambda_t\geq a+b.
\end{equation}
If $a=0$, this follows from \eqref{eq:BoundaryRunningInfimumAboveInitialSupport}. Suppose $a>0$. If \eqref{eq:InitialTraceLowerInequality} failed, then for some $\delta\in(0,2a)$ and some $t_n\downarrow0$, $\Lambda_{t_n}\leq a+b-\delta$. Set $\eta:=\delta/4$ and
\[
    A_n:=\left\{X_0\geq b-\frac{\delta}{4},\quad B_{t_n}\geq a-\frac{\delta}{2}\right\}.
\]
On $A_n$, $K_{t_n}\geq\eta$. Since $L_{t_n}=K_{t_n}$ on $\{\tau>t_n\}$, the exact balance equation gives
\[
    \alpha q_{t_n}\geq\eta\,\PP[A_n\cap\{\tau>t_n\}]\geq\eta\bigl(\PP[A_n]-q_{t_n}\bigr).
\]
Equivalently,
\[
    q_{t_n}\geq\frac{\eta}{\alpha+\eta}\PP[A_n].
\]
Moreover,
\[
    p_\delta:=\mu_0\!\left(\left[b-\frac{\delta}{4},b\right]\right)>0,\quad\PP[A_n]=p_\delta\,\GaussTail\!\left(\frac{a-\delta/2}{\sqrt{t_n}}\right),
\]
where we used the independence of $X_0$ and $B$. Hence
\[
    \lim_{n\to\infty}t_n\log\PP[A_n]=-\frac{(a-\delta/2)^2}{2}>-\frac{a^2}{2},
\]
contradicting \eqref{eq:DefaultLogarithmicRate}. This proves \eqref{eq:InitialTraceLowerInequality}.

It remains to prove the reverse inequality. If it failed, then for some $\delta>0$ and all sufficiently small $t$, $\underline\Lambda_t\geq a+b+\delta$. Writing $M_t:=\sup_{0\leq s\leq t}B_s$ and using $X_0\leq b$ almost surely, we obtain
\[
    L_t\leq K_t\leq(M_t-a-\delta)^+.
\]
Exact balance therefore implies
\begin{equation}\label{eq:BalanceUpperRegulatorEstimate}
    \alpha q_t\leq\EE[(M_t-a-\delta)^+].
\end{equation}
For every $d>0$, the reflection principle gives
\[
    \EE[(M_t-d)^+]\leq\EE[M_t\mathbf 1_{\{M_t\geq d\}}]=2\sqrt t\,\phi\!\left(\frac d{\sqrt t}\right).
\]
Taking $d=a+\delta$ in \eqref{eq:BalanceUpperRegulatorEstimate} gives
\[
    \limsup_{t\downarrow0}t\log q_t\leq-\frac{(a+\delta)^2}{2},
\]
contradicting \eqref{eq:DefaultLogarithmicRate}. This proves \eqref{eq:InitialBoundaryLowerTrace}.
\end{proof}

Next, we are going to show that the initial lower trace obtained in Lemma \ref{lem:InitialBoundaryLowerTrace} actually gives the correct limit of $\Lambda_t$ as $t\downarrow0$. Towards this goal, we start with proving some technical results about probabilistic solutions.

\begin{lemma}[Protected-path allocation estimate]
\label{lem:ProtectedPathAllocation}
Retain the notation
\[
a:=\inf\operatorname{supp}(\mu_0),\quad b:=\max\operatorname{supp}(\mu_0),\quad q_t:=\PP[\tau\leq t],\quad h_t:=\EE[L_t].
\]
For every $A\in(a,a+b)$, there are constants $c_A>0$ and $t_A>0$ such that, as Stieltjes measures on $(0,t_A]$,
\begin{equation}\label{eq:ProtectedDefaultMeasureLowerBound}
dq_t\geq c_A t^{-3/2}\exp\!\left(-\frac{A^2}{2t}\right)dt.
\end{equation}
Consequently, the exact balance equation gives
\begin{equation}\label{eq:ProtectedAllocationMeasureLowerBound}
dh_t\geq\alpha c_A t^{-3/2}\exp\!\left(-\frac{A^2}{2t}\right)dt.
\end{equation}
\end{lemma}

\begin{proof}
Choose $d\in(A,a+b)$. By Lemma~\ref{lem:InitialBoundaryLowerTrace},
after decreasing $t_A>0$ if necessary, it holds that $\Lambda_t\geq d$ for $t\in(0,t_A]$. Since $a<A$ and $\mu_0(\{0\})=0$, one has $\mu_0((0,A))>0$. We may therefore choose $0<x_-<x_+<A$ such that, for $I:=[x_-,x_+]$, $\mu_0(I)>0$. For Brownian motion $Y^x_t:=x+B_t$ started from $x\in I$, let $T_0$ and $T_d$ denote the first hitting times of $0$ and $d$, respectively. On $\{T_0<T_d,\ T_0\leq t_A\}$, the path stays below $d\leq\Lambda_s$ for $0<s<T_0$. Hence the regulator vanishes up to $T_0$ and $\tau=T_0$. Consequently, for every Borel set $E\subset(0,t_A]$,
\begin{equation}\label{eq:ProtectedPathMeasureDomination}
\PP[\tau\in E]\geq\int_I \PP_x[T_0\in E,\ T_0<T_d]\,\mu_0(dx).
\end{equation}
For $z>0$, let
\[
f_z(t):=\frac{z}{\sqrt{2\pi t^3}}\exp\!\left(-\frac{z^2}{2t}\right),\quad t>0,
\]
be the Brownian first-passage density over distance $z$. By the strong Markov property at $T_d$, and because the law of $T_d$ on $\{T_d<T_0\}$ is a submeasure of the unrestricted hitting-time law with density $f_{d-x}$,
\[
\PP_x[T_0\in dt,\ T_d<T_0]\leq (f_{d-x}*f_d)(t)\,dt=f_{2d-x}(t)\,dt.
\]
It follows that
\[
\frac{\PP_x[T_0\in dt,\ T_0<T_d]}{dt}\geq f_x(t)-f_{2d-x}(t).
\]
Moreover, uniformly for $x\in I$,
\[
\frac{f_{2d-x}(t)}{f_x(t)}=\frac{2d-x}{x}\exp\!\left(-\frac{2d(d-x)}{t}\right)\leq\frac{2d-x_-}{x_-}\exp\!\left(-\frac{2d(d-x_+)}{t}\right)\longrightarrow0
\]
as $t\downarrow0$. After decreasing $t_A$ once more, we therefore have, for $x\in I$ and $0<t\leq t_A$,
\[
f_x(t)-f_{2d-x}(t)\geq\frac12 f_x(t)\geq\frac{x_-}{2\sqrt{2\pi}}\,t^{-3/2}\exp\!\left(-\frac{A^2}{2t}\right).
\]
Combining this estimate with \eqref{eq:ProtectedPathMeasureDomination} proves \eqref{eq:ProtectedDefaultMeasureLowerBound}, with $c_A:=\mu_0(I)x_-/(2\sqrt{2\pi})$. Finally, $h_t=\alpha q_t$ for every $t$, and hence $dh=\alpha\,dq$, proving \eqref{eq:ProtectedAllocationMeasureLowerBound}.
\end{proof}

\begin{remark}[Positive-time support versus the initial trace]
\label{rem:PositiveTimeSupportAndInitialTrace}
The support identity
\[
    \operatorname{supp}(\mu_t)=[0,\Lambda_t]
\]
is asserted only for $t>0$.  The initial law may have gaps, and in general $\operatorname{supp}(\mu_0)$ need not be an interval of the form $[0,\Lambda_0]$. Lemma~\ref{lem:InitialBoundaryLowerTrace} identifies the lower trace of the boundary. We next prove the matching upper trace from the weak evolution of the ordinary call function; no initial support identity is imposed.
\end{remark}

Retain and/or introduce the notations
\[
    \mu_t:=\mathcal L(X_t),\quad\mu_t^+:=\mu_t\big|_{(0,\infty)},\quad h_t:=\EE[L_t],\quad m_0:=\int_{\mathbb R_+}y\,\mu_0(dy),
\]
and define the ordinary call function
\begin{equation}
\label{eq:OrdinaryCallDefinition}
C(t,x):=\int_{\mathbb R_+}(y-x)^+\,\mu_t(dy),\quad (t,x)\in[0,T]\times\mathbb R_+.
\end{equation}

\begin{lemma}[Weak evolution of the ordinary call function]
\label{lem:WeakOrdinaryCallEvolution}
Under Assumption \ref{ass:initial-law}, the function $C$ is jointly continuous on $[0,T]\times\mathbb R_+$, and
\begin{equation}\label{eq:OrdinaryCallMomentIdentity}
C(t,0)=m_0-h_t,\quad t\ge0.
\end{equation}
Moreover,
\[
    \partial_{xx}C(t,\cdot)=\mu_t^+
\]
in the sense of distributions on $(0,\infty)$. In addition, for every $0\leq s<t\leq T$ and every $\eta\in C_c^\infty((0,\infty))$,
\begin{align}
\label{eq:call-function-weak-evolution}
\langle C(t,\cdot)-C(s,\cdot),\eta\rangle=\frac12\int_s^t\langle\mu_u^+,\eta\rangle\,du-\int_{(s,t]}\left(\int_0^{\Lambda_u}\eta(x)\,dx\right)dh_u.
\end{align}
Consequently, as locally finite Radon measures on $(0,T)\times(0,\infty)$,
\begin{equation}
\label{eq:WeakCallEquationForAdvance}
\left(\partial_t-\frac12\partial_{xx}\right)C=-\mathbf 1_{\{0<x<\Lambda_t\}}\,dx\,dh_t.
\end{equation}
Finally, for every $t>0$,
\begin{equation}\label{eq:CallPositiveExactlyBelowBoundary}
C(t,x)>0\quad\Longleftrightarrow\quad0\leq x<\Lambda_t.
\end{equation}
\end{lemma}

\begin{proof}
On $[0,T]$, the absence of atoms in the law of $\tau$ implies that $h$ is continuous. Since $L$ is non-decreasing, $\EE|L_t-L_s|=|h_t-h_s|$. Hence
\[
\EE|X_t-X_s|\leq\EE|B_{t\wedge\tau}-B_{s\wedge\tau}|+|h_t-h_s|\longrightarrow0\quad\text{as }|t-s|\to0.
\]
Thus $t\mapsto\mu_t$ is continuous in $\mathcal P_1(\mathbb R_+)$. Since
\[
|C(t,x)-C(s,y)|\leq W_1(\mu_t,\mu_s)+|x-y|,
\]
$C$ is jointly continuous. Since $X_t\geq0$, optional stopping in the stopped-state identity gives
\[
    C(t,0)=\EE[X_t]=\EE[X_0]-\EE[L_t]=m_0-h_t,
\]
which proves \eqref{eq:OrdinaryCallMomentIdentity}. For any $C^2$ function $\phi$ of at most linear growth with bounded first and second derivatives, It\^o's formula for the stopped reflected process gives
\begin{equation}
\label{eq:weak-forward-for-obstacle}
\langle\phi,\mu_t\rangle-\langle\phi,\mu_s\rangle=\frac12\int_s^t\langle\phi'',\mu_u^+\rangle\,du-\int_{(s,t]}\phi'(\Lambda_u)\,dh_u.
\end{equation}
Indeed, before absorption $dL$ is carried by the contact set $\{X=\Lambda\}$, and its mean Stieltjes measure is $dh$. The identity first follows for compactly supported smooth $\phi$ and then by a cutoff argument.  It is enough here to work on the positive-time interval $[s,t]$, on which the boundary is bounded. Now, given $\eta\in C_c^\infty((0,\infty))$, put
\[
    K_\eta(y):=\int_0^\infty(y-x)^+\eta(x)\,dx =\int_0^y(y-x)\eta(x)\,dx.
\]
Then
\[
    K_\eta'(y)=\int_0^y\eta(x)\,dx,\quad K_\eta''(y)=\eta(y),\quad\langle K_\eta,\mu_t\rangle=\langle C(t,\cdot),\eta\rangle.
\]
Substituting $K_\eta$ into \eqref{eq:weak-forward-for-obstacle} proves \eqref{eq:call-function-weak-evolution}. The distributional identity $\partial_{xx}C(t,\cdot)=\mu_t^+$ follows directly from $\partial_{xx}(y-x)^+=\delta_y$ on $(0,\infty)$, and therefore \eqref{eq:call-function-weak-evolution} is equivalent to \eqref{eq:WeakCallEquationForAdvance}. Finally, $C(t,x)>0$ if and only if $\mu_t((x,\infty))>0$. The support identity $\operatorname{supp}(\mu_t)=[0,\Lambda_t]$ for $t>0$ now gives \eqref{eq:CallPositiveExactlyBelowBoundary}.
\end{proof}

The following positive-set maximum principle is the analytic comparison tool used below. For $Q=(r,S)\times(0,R)$, write its parabolic boundary as
\[
    \partial_pQ:=\{r\}\times[0,R]\,\cup\,[r,S]\times\{0,R\}.
\]

\begin{lemma}[Positive-set parabolic maximum principle]
\label{lem:positive-set-parabolic-maximum-principle}
Let $z\in C(\overline Q)$ and suppose that $\left(\partial_t-\frac12\partial_{xx}\right)z$ is a locally finite signed Radon measure in $Q$. Assume that $\left(\partial_t-\frac12\partial_{xx}\right)z\le0$ on the open set $\{z>0\}$ in the distributional sense, and that $z\le0$ on $\partial_pQ$. Then $z\le0$ in $\overline Q$.
\end{lemma}

\begin{proof}
Fix $\rho>0$. The parabolic Kato inequality gives
\begin{equation}\label{eq:parabolic-kato}
\left(\partial_t-\frac12\partial_{xx}\right)(z-\rho)^+\leq\mathbf{1}_{\{z>\rho\}}\left(\partial_t-\frac12\partial_{xx}\right)z
\end{equation}
in the sense of distributions. For completeness, the local form of this inequality follows by regularizing $z$ in the interior of $Q$ and applying the chain rule to smooth convex approximations $\beta_\varepsilon$ of $r\mapsto r^+$:
\[
\left(\partial_t-\frac12\partial_{xx}\right)\beta_\varepsilon(z)=\beta_\varepsilon'(z)\left(\partial_t-\frac12\partial_{xx}\right)z-\frac12\beta_\varepsilon''(z)|\partial_xz|^2\leq\beta_\varepsilon'(z)\left(\partial_t-\frac12\partial_{xx}\right)z.
\]
Applying this calculation to $z-\rho$, passing first through the regularization, and then letting $\varepsilon\downarrow0$ yields \eqref{eq:parabolic-kato} for measure-valued data. Since $\{z>\rho\}\subset\{z>0\}$, its right-hand side is non-positive by assumption. Thus $(z-\rho)^+$ is a distributional subsolution of the heat equation and has zero parabolic boundary values. The weak parabolic maximum principle gives $(z-\rho)^+=0$. Letting $\rho\downarrow0$ proves $z\leq0$.
\end{proof}

The following argument is in the same spirit of Lemma \ref{lem:cutoff-modulus} as in the discrete-time situation. We present the full proof for completeness.

\begin{lemma}[A probabilistic upper-advance modulus]
\label{lem:ProbabilisticUpperAdvanceModulus}
Let $0<r<r+\rho\leq T$. Suppose that, for some $\gamma>0$,
\begin{equation}\label{eq:AllocationDensityLowerOnInterval}
dh_s\geq\gamma\,ds\quad\text{in the distributional sense on }(r,r+\rho).
\end{equation}
Then
\begin{equation}\label{eq:ProbabilisticUpperAdvanceModulus}
    \Lambda_{r+\rho}-\Lambda_r\leq\Omega_\gamma(\rho),
\end{equation}
where
\begin{equation}\label{eq:DefinitionOmegaGamma}
\Omega_\gamma(\rho):=\inf_{\theta>0}\left\{\frac{\log(1+\theta/\gamma)}{2\theta}+\theta\rho\right\}.
\end{equation}
\end{lemma}

\begin{proof}
Fix $\theta>0$ and set
\[
\beta_{\theta,\gamma}:=\frac{\log(1+\theta/\gamma)}{2\theta},\quad F_{\theta,\gamma}(y):=\frac{\gamma}{2\theta^2}\bigl(e^{2\theta y}-1-2\theta y\bigr),\quad y\geq0.
\]
The function $F_{\theta,\gamma}$ is non-negative, non-decreasing, and
convex, with
\[
F_{\theta,\gamma}(0)=F_{\theta,\gamma}'(0)=0,\quad F_{\theta,\gamma}'(\beta_{\theta,\gamma})=1.
\]
For $r\leq s\leq r+\rho$, define
\[
\xi_s:=\Lambda_r+\beta_{\theta,\gamma}+\theta(s-r),\quad W(s,x):=F_{\theta,\gamma}\bigl((\xi_s-x)^+\bigr).
\]
Because $F_{\theta,\gamma}(0)=F_{\theta,\gamma}'(0)=0$, differentiation across $x=\xi_s$ creates no singular measure. A direct calculation gives
\begin{equation}\label{eq:BarrierHeatOperator}
\left(\partial_s-\frac12\partial_{xx}\right)W=-\gamma\mathbf 1_{\{0<x<\xi_s\}}\,dx\,ds.
\end{equation}
Since $\mu_r$ is supported in $[0,\Lambda_r]$, it is clear that $C(r,x)\leq(\Lambda_r-x)^+$. For $0\leq x\leq\Lambda_r$, convexity and $F_{\theta,\gamma}'(\beta_{\theta,\gamma})=1$ give
\[
W(r,x)=F_{\theta,\gamma}(\beta_{\theta,\gamma}+\Lambda_r-x)\geq F_{\theta,\gamma}(\beta_{\theta,\gamma})+\Lambda_r-x\geq\Lambda_r-x.
\]
For $x\geq\Lambda_r$, the inequality $C(r,x)\leq W(r,x)$ is immediate. Moreover, \eqref{eq:OrdinaryCallMomentIdentity}, the monotonicity of $h$, and the preceding convexity estimate give, for $r\leq s\leq r+\rho$,
\[
C(s,0)=m_0-h_s\leq m_0-h_r=C(r,0)\leq\Lambda_r\leq W(s,0).
\]
Choose $R>\max\left\{\sup_{r\leq s\leq r+\rho}\Lambda_s,\,\xi_{r+\rho}\right\}$. Then $C(s,R)=W(s,R)=0$ throughout the interval. Consider
\[
    Z:=C-W,\quad\mathcal O:=\{(s,x)\in(r,r+\rho)\times(0,R):Z(s,x)>0\}.
\]
Let $\varphi\in C_c^\infty(\mathcal O)$ be non-negative and set
\[
    G(s):=\int_0^R\varphi(s,x)\,dx.
\]
On $\operatorname{supp}(\varphi)$, one has $C>0$, and hence $x<\Lambda_s$ by \eqref{eq:CallPositiveExactlyBelowBoundary}. Equations \eqref{eq:WeakCallEquationForAdvance} and \eqref{eq:BarrierHeatOperator} therefore give
\[
\begin{aligned}
\left\langle\left(\partial_s-\frac12\partial_{xx}\right)Z,\varphi\right\rangle&=-\int_{(r,r+\rho)}G(s)\,dh_s+\gamma\int_r^{r+\rho}\int_0^R\mathbf 1_{\{x<\xi_s\}}\varphi(s,x)\,dx\,ds\\
&\leq-\int_{(r,r+\rho)}G(s)\,dh_s+\gamma\int_r^{r+\rho}G(s)\,ds\leq0,
\end{aligned}
\]
where the last inequality follows from \eqref{eq:AllocationDensityLowerOnInterval}. Therefore,
\[
\left(\partial_s-\frac12\partial_{xx}\right)Z\leq0\qquad\text{in }\mathcal D'(\mathcal O).
\]
Lemma~\ref{lem:positive-set-parabolic-maximum-principle} yields $Z\leq0$, i.e. $C\leq W$. At time $r+\rho$, the barrier vanishes for
\[
x\geq\xi_{r+\rho}=\Lambda_r+\beta_{\theta,\gamma}+\theta\rho.
\]
Equation \eqref{eq:CallPositiveExactlyBelowBoundary} now gives
\[
\Lambda_{r+\rho}\leq\Lambda_r+\beta_{\theta,\gamma}+\theta\rho.
\]
Taking the infimum over $\theta>0$ proves \eqref{eq:ProbabilisticUpperAdvanceModulus}.
\end{proof}

\begin{proposition}[Full initial trace of the free boundary]
\label{prop:FullInitialBoundaryTrace}
Let $(X,\Lambda)$ be a probabilistic solution with initial law $\mu_0$ satisfying Assumption~\ref{ass:initial-law}, and set $a:=\inf\operatorname{supp}(\mu_0)$, $b:=\max\operatorname{supp}(\mu_0)$. Then
\begin{equation}\label{eq:FullInitialBoundaryTrace}
\lim_{t\downarrow0}\Lambda_t=a+b.
\end{equation}
Therefore, the assignment $\Lambda_0:=a+b$ extends $\Lambda$ continuously to $[0,T]$.
\end{proposition}

\begin{proof}
Lemma~\ref{lem:InitialBoundaryLowerTrace} gives
\begin{equation}
\label{eq:LowerTraceUsedInFullTrace}
\liminf_{t\downarrow0}\Lambda_t=a+b.
\end{equation}
Suppose, towards a contradiction, that $\limsup_{t\downarrow0}\Lambda_t>a+b$. Choose $\delta\in(0,b)$ and $t_n\downarrow0$ such that $\Lambda_{t_n}\ge a+b+3\delta$. Choose
\begin{equation}
\label{eq:ChoiceAForProtectedPaths}
A\in(a,a+\delta).
\end{equation}
By Lemma~\ref{lem:ProtectedPathAllocation}, there are $\kappa_A>0$ and $t_A>0$ such that
\begin{equation}
\label{eq:ProtectedAllocationDensityForExcursion}
dh_s\ge\kappa_A s^{-3/2}\exp\!\left(-\frac{A^2}{2s}\right)\,ds,\quad 0<s\le t_A.
\end{equation}
Decrease $t_A$ so that $t_A<A^2/3$. Then, the density on the right-hand side of \eqref{eq:ProtectedAllocationDensityForExcursion} is increasing on $(0,t_A]$. By \eqref{eq:LowerTraceUsedInFullTrace}, for all large $n$ we may choose $\sigma_n\in(0,t_n)$ such that $\Lambda_{\sigma_n}<a+b+\delta$.  Define
\[
    r_n:=\sup\{s\in[\sigma_n,t_n]:\Lambda_s\le a+b+\delta\},\quad\rho_n:=t_n-r_n.
\]
Positive-time continuity of $\Lambda$ gives
\begin{equation}
\label{eq:LastCrossingHighExcursion}
    \Lambda_{r_n}=a+b+\delta,\quad\Lambda_s>a+b+\delta\quad(r_n<s\le t_n),\quad\Lambda_{t_n}-\Lambda_{r_n}\ge2\delta.
\end{equation}
Moreover, $r_n\to0$ and $\rho_n\to0$. For all large $n$, \eqref{eq:ProtectedAllocationDensityForExcursion} and the monotonicity of its density imply
\[
dh_s\ge\gamma_n\,ds\quad\text{on }(r_n,t_n),\quad\gamma_n:=\kappa_A r_n^{-3/2}\exp\!\left(-\frac{A^2}{2r_n}\right).
\]
We claim that
\[
    \liminf_{n\to\infty}\frac{\rho_n}{r_n}>0.
\]
Otherwise, after passing to a subsequence,
\[
    \varepsilon_n:=\frac{\rho_n}{r_n}\longrightarrow0.
\]
Apply Lemma~\ref{lem:ProbabilisticUpperAdvanceModulus} with
\[
\theta_n:=(r_n\rho_n)^{-1/2}=(r_n\sqrt{\varepsilon_n})^{-1}.
\]
Since $\theta_n\rho_n=\sqrt{\varepsilon_n}$ and
\[
\frac{\theta_n}{\gamma_n}=\kappa_A^{-1}\sqrt{\frac{r_n}{\varepsilon_n}}\exp\!\left(\frac{A^2}{2r_n}\right),
\]
the elementary inequalities $1+e^q u\leq e^q(1+u)$ for $q,u\geq0$ and $\log(1+u)\leq u$ give
\begin{align*}
\frac{\log(1+\theta_n/\gamma_n)}{2\theta_n}\leq\frac{A^2}{4}\sqrt{\varepsilon_n}+\frac{r_n\sqrt{\varepsilon_n}}{2}\log\left(1+\kappa_A^{-1}\sqrt{\frac{r_n}{\varepsilon_n}}\right)\leq\frac{A^2}{4}\sqrt{\varepsilon_n}+\frac{r_n^{3/2}}{2\kappa_A}.
\end{align*}
Together with $\theta_n\rho_n=\sqrt{\varepsilon_n}$ and \eqref{eq:DefinitionOmegaGamma}, this implies
\[
\Lambda_{t_n}-\Lambda_{r_n}\leq\left(1+\frac{A^2}{4}\right)\sqrt{\varepsilon_n}+\frac{r_n^{3/2}}{2\kappa_A}\longrightarrow0,
\]
contradicting \eqref{eq:LastCrossingHighExcursion}. Hence, there are $\eta>0$ and $n_0$ such that
\[
\rho_n\ge\eta r_n,\quad n\ge n_0.
\]

Now, choose $\vartheta$ such that
\begin{equation}
\label{eq:ChoiceVarthetaExcursion}
\max\left\{\frac1{1+\eta},\frac{A^2}{(a+\delta)^2}\right\}<\vartheta<1. 
\end{equation}
Then $[\vartheta t_n,t_n]\subset(r_n,t_n]$.  Integrating
\eqref{eq:ProtectedAllocationDensityForExcursion} over this interval gives
\begin{equation}
\label{eq:HighExcursionAllocationLowerBound}
    h_{t_n}-h_{\vartheta t_n}
    \ge c_1t_n^{-1/2}
       \exp\!\left(-\frac{A^2}{2\vartheta t_n}\right)
\end{equation}
for some $c_1>0$.

For the reverse estimate, let $M_t:=\sup_{0\le s\le t}B_s$. On the event $\{L_{t_n}-L_{\vartheta t_n}>0\}$, the fact that $dL_u$ is carried by $\{u<\tau:X_u=\Lambda_u\}$ gives a time $s\in(\vartheta t_n,t_n]$ such that
\[
X_s=\Lambda_s,\quad B_s=\Lambda_s-X_0+L_s\ge a+\delta.
\]
Moreover, $X_0\le b$ and $\Lambda>0$ imply $L_{t_n}\le K_{t_n}\le b+M_{t_n}$. Thus, pathwise,
\[
    L_{t_n}-L_{\vartheta t_n}\le(b+M_{t_n})\mathbf 1_{\{M_{t_n}\ge a+\delta\}}.
\]
The reflection principle and Mills' inequality therefore give
\begin{align}
\label{eq:HighExcursionAllocationUpperBound}
h_{t_n}-h_{\vartheta t_n}\le2b\,\GaussTail\left(\frac{a+\delta}{\sqrt{t_n}}\right)+2\sqrt{t_n}\,\phi\left(\frac{a+\delta}{\sqrt{t_n}}\right)\le c_2\sqrt{t_n}\exp\!\left(-\frac{(a+\delta)^2}{2t_n}\right)
\end{align}
for some $c_2<\infty$. By \eqref{eq:ChoiceVarthetaExcursion}, $A^2/\vartheta<(a+\delta)^2$, so \eqref{eq:HighExcursionAllocationLowerBound} and \eqref{eq:HighExcursionAllocationUpperBound} jointly lead to a contradiction as $n\to\infty$. We conclude that
\[
    \limsup_{t\downarrow0}\Lambda_t\le a+b.
\]
Together with \eqref{eq:LowerTraceUsedInFullTrace}, this concludes the proof of \eqref{eq:FullInitialBoundaryTrace}.
\end{proof}

\begin{remark}[Canonical endpoint]
\label{rem:CanonicalInitialBoundaryEndpoint}
The full initial trace is obtained without assuming a density trace at
the moving upper boundary.  The canonical endpoint is
\[
    \Lambda_0=\inf\operatorname{supp}(\mu_0)+\max\operatorname{supp}(\mu_0),
\]
which need not equal $\max\operatorname{supp}(\mu_0)$. No support identity is asserted at time zero: the initial law may be disconnected or singular. The auxiliary convention $\Lambda_0^\Delta=\max\operatorname{supp}(\mu_0)$ in the discrete scheme in Section \ref{sec:continuous} may be safely retained, because it is dynamically irrelevant and the compactness argument is deliberately local on $(0,T]$.
\end{remark}

\subsection{The augmented call-function formulation}
\label{subsec:augmented-call-obstacle}

The weak evolution of the ordinary call function was established in
Lemma~\ref{lem:WeakOrdinaryCallEvolution}. Adding the cumulative loss $h$
moves its source measure from the continuation region to the coincidence
region. We now formulate the resulting obstacle problem without reference to
a free boundary.

\begin{definition}[The continuous weak augmented-call class]
\label{def:continuous-weak-augmented-call}
Fix an initial law $\mu_0\in\mathcal P_1(\mathbb R_+)$ satisfying $\mu_0(\{0\})=0$, and write
\[
    m_0:=\int_{\mathbb R_+}x\,\mu_0(dx).
\]
We call a pair
\[
    H\in C([0,T]\times\mathbb R_+),\quad h\in C([0,T])
\]
a \emph{continuous weak augmented-call solution with initial law $\mu_0$} if the following conditions hold:
\begin{enumerate}
\item
The function $h$ is non-decreasing.

\item
There exists a family $(\mu_t)_{0\leq t\leq T}\subset\mathcal P_1(\mathbb R_+)$, with initial value $\mu_0$, such that, for every $t\in[0,T]$,
\begin{equation}\label{eq:augmented-call-representation}
\begin{aligned}
H(t,x)-h_t&=C_{\mu_t}(x):=\int_{\mathbb R_+}(y-x)^+\,\mu_t(dy),
&&x\geq0,\\
h_t&=\alpha\mu_t(\{0\}).
\end{aligned}
\end{equation}

\item
The conserved first-moment boundary condition holds:
\begin{equation}\label{eq:augmented-call-left-boundary}
    H(t,0)=m_0,\quad 0\leq t\leq T.
\end{equation}

\item
The distribution
\[
\nu_H:=\left(\partial_t-\frac12\partial_{xx}\right)H
\]
is a locally finite non-negative Radon measure on $(0,T)\times(0,\infty)$, and
\[
    (H-h)\,\nu_H=0.
\]
Here, $h$ is identified with the function $(t,x)\mapsto h_t$. Equivalently,
\[
\operatorname{supp}(\nu_H)\subseteq\{(t,x):H(t,x)=h_t\}.
\]
\end{enumerate}
\end{definition}

\begin{remark}
Condition 2 gives $h_0=0$. For each $t$, the measure $\mu_t$ is uniquely determined by $(H,h)$. Indeed,
\begin{equation}\label{eq:augmented-call-second-derivative}
\partial_{xx}H(t,\cdot)=\mu_t\big|_{(0,\infty)},\quad \mu_t(\{0\})=\frac{h_t}{\alpha},
\end{equation}
where the first identity holds in the sense of distributions on $(0,\infty)$. Moreover, conditions 2--3 imply
\begin{equation}\label{eq:augmented-call-bounds}
\begin{gathered}
0\leq h_t\leq\alpha,\quad\int_{\mathbb R_+}y\,\mu_t(dy)=m_0-h_t,\\
0\leq h_t\leq H(t,x)\leq m_0,\quad (t,x)\in[0,T]\times\mathbb R_+.
\end{gathered}
\end{equation}
Furthermore,
\begin{equation}\label{eq:augmented-call-obstacle-at-infinity}
\lim_{x\to\infty}H(t,x)=h_t.
\end{equation}
For every $\mu\in\mathcal P_1(\mathbb R_+)$, the right derivative of its call function satisfies $\partial_x^+C_\mu(x)=-\mu((x,\infty))$. Consequently,
\begin{equation}\label{eq:augmented-call-closure}
\partial_x^+H(t,0)=-\mu_t((0,\infty))=-(1-\mu_t(\{0\})),\quad h_t=\alpha\bigl(1+\partial_x^+H(t,0)\bigr).
\end{equation}
Note that neither a free boundary $\Lambda$ nor a representation of $\nu_H$ in terms of such a boundary is included in the definition.
\end{remark}

Let $(X,\Lambda)$ be a probabilistic solution under Assumption~\ref{ass:initial-law}, retain the notation $\mu_t,h_t$, and $C$ from Lemma~\ref{lem:WeakOrdinaryCallEvolution}, and define
\begin{equation}\label{eq:augmented-call-definition}
H(t,x):=C(t,x)+h_t,\quad (t,x)\in[0,T]\times\mathbb R_+.
\end{equation}

\begin{proposition}[Probabilistic solutions induce weak augmented-call
solutions]\label{prop:weak-augmented-call}
The pair $(H,h)$ defined by \eqref{eq:augmented-call-definition} belongs to the continuous weak augmented-call class. More precisely, its obstacle measure admits the sharper representation
\begin{equation}\label{eq:weak-augmented-call-measure}
\nu_H:=\left(\partial_t-\frac12\partial_{xx}\right)H=\mathbf 1_{\{x\geq\Lambda_t\}}\,dx\,dh_t
\end{equation}
as locally finite Radon measures on $(0,T)\times(0,\infty)$. In particular,
\begin{equation}\label{eq:weak-augmented-call-obstacle}
H\geq h,\quad\nu_H\geq0,\quad(H-h)\nu_H=0.
\end{equation}
Moreover,
\begin{equation}\label{eq:obstacle-coincidence-set}
\{x\geq0:H(t,x)=h_t\}=[\Lambda_t,\infty),\quad t>0.
\end{equation}
\end{proposition}

\begin{proof}
Exact balance shows that $h$ is non-decreasing and that $h_t=\alpha\mu_t(\{0\})$. Lemma~\ref{lem:WeakOrdinaryCallEvolution} gives the required continuity. The call-function representation holds by construction, while \eqref{eq:OrdinaryCallMomentIdentity} gives the boundary condition. Since the distributional time derivative of the spatially constant function $(t,x)\mapsto h_t$ is $dx\,dh_t$, equation \eqref{eq:WeakCallEquationForAdvance} yields
\[
\left(\partial_t-\frac12\partial_{xx}\right)H=-\mathbf 1_{\{0<x<\Lambda_t\}}\,dx\,dh_t+dx\,dh_t=\mathbf 1_{\{x\geq\Lambda_t\}}\,dx\,dh_t,
\]
which proves \eqref{eq:weak-augmented-call-measure}. This measure is locally finite because, for every $R<\infty$ and $0<s<t<T$,
\[
    \nu_H\bigl((s,t]\times(0,R)\bigr)\leq R(h_t-h_s)\leq R\alpha.
\]
Finally, $H-h=C\geq0$, and \eqref{eq:CallPositiveExactlyBelowBoundary} gives \eqref{eq:obstacle-coincidence-set}. The measure in \eqref{eq:weak-augmented-call-measure} is therefore non-negative and carried by $\{H=h\}$, proving \eqref{eq:weak-augmented-call-obstacle} and completing the verification of the definition.
\end{proof}

\subsection{Stability and Uniqueness in the Augmented-Call Topology}
\label{subsec:stability-and-uniqueness}
We now compare two continuous weak augmented-call solutions with possibly different initial laws. The natural distance must record not only the initial call function, but also its short-time response to absorption at zero.

For $\mu\in\mathcal{P}_1(\mathbb{R}_+)$, write
\[
    C_\mu(x):=\int_{\mathbb{R}_+}(y-x)^+\,\mu(dy).
\]
Let $(S_t)_{t\geq0}$ denote the absorbed Brownian semigroup and define the control-free default curve
\begin{equation}\label{eq:free-default-curve}
q_\mu^{\mathrm{fr}}(t):=(S_t\mu)(\{0\})=\int_{\mathbb{R}_+}2\GaussTail\!\left(\frac{x}{\sqrt{t}}\right)\mu(dx),\quad t>0,
\end{equation}
where $\GaussTail$ is the standard Gaussian tail.  For the fixed horizon $T>0$, set
\begin{equation}\label{eq:augmented-call-distance}
d_T^\sharp(\mu,\nu):=\|C_\mu-C_\nu\|_{L^\infty(\mathbb{R}_+)}+\sup_{0<t\leq T}\bigl|q_\mu^{\mathrm{fr}}(t)-q_\nu^{\mathrm{fr}}(t)\bigr|.
\end{equation}

\begin{theorem}[Stability of continuous weak augmented-call solutions]
\label{thm:pairwise-FBP-stability}
Fix $\alpha>0$, $T>0$, and $a_0>0$. For $i=1,2$, let
$\mu_0^i\in\mathcal P_1(\mathbb R_+)$ satisfy
\[
    \mu_0^i(\{0\})=0,\quad m_i:=\int_{\mathbb R_+}x\,\mu_0^i(dx)\geq\alpha+a_0,
    \qquad i=1,2,
\]
and let $(H^i,h^i)$ be continuous weak augmented-call solutions with
initial law $\mu_0^i$, $i=1,2$, respectively. Let $(\mu_t^i)_{0\leq t\leq T}$ be the associated family from Definition~\ref{def:continuous-weak-augmented-call}, and write
\[
    C^i(t,x):=H^i(t,x)-h_t^i=C_{\mu_t^i}(x),\quad q_t^i:=\mu_t^i(\{0\})=\frac{h_t^i}{\alpha}.
\]
Then there exists $C=C(\alpha,a_0,T)<\infty$ such that
\begin{equation}\label{eq:augmented-call-stability}
\sup_{0\leq t\leq T}\left(\|H^1(t,\cdot)-H^2(t,\cdot)\|_{L^\infty(\mathbb R_+)}+|h_t^1-h_t^2|\right)\leq C\,d_T^\sharp(\mu_0^1,\mu_0^2).
\end{equation}
After changing $C$ if necessary,
\begin{equation}\label{eq:call-and-default-stability}
\sup_{0\leq t\leq T}\left(\|C^1(t,\cdot)-C^2(t,\cdot)\|_{L^\infty(\mathbb R_+)}+|q_t^1-q_t^2|\right)\leq C\,d_T^\sharp(\mu_0^1,\mu_0^2).
\end{equation}
\end{theorem}

\begin{proof}
Set
\begin{equation}\label{eq:stability-epsilon-beta}
\varepsilon:=\|C_{\mu_0^1}-C_{\mu_0^2}\|_{L^\infty(\mathbb R_+)},\quad\beta_T:=\sup_{0<t\leq T}\bigl|q_{\mu_0^1}^{\mathrm{fr}}(t)-q_{\mu_0^2}^{\mathrm{fr}}(t)\bigr|.
\end{equation}
Thus $d_T^\sharp(\mu_0^1,\mu_0^2)=\varepsilon+\beta_T$. Set $r:=a_0/2$.\\

\emph{Step 1: A common continuation strip and obstacle comparison.}
By \eqref{eq:augmented-call-bounds}, $0\leq h_t^i\leq\alpha$ and $C^i(t,0)=m_i-h_t^i$. Since every call function is $1$-Lipschitz in space,
\begin{equation}\label{eq:common-continuation-strip}
C^i(t,x)\geq C^i(t,0)-x=m_i-h_t^i-x\geq a_0-x\geq \frac{a_0}{2}>0,\quad (t,x)\in[0,T]\times[0,r].
\end{equation}
The complementarity condition therefore implies that $\nu_{H^i}$ vanishes on $(0,T)\times(0,r)$, so
\begin{equation}\label{eq:heat-equation-common-strip}
\partial_tH^i=\frac12\partial_{xx}H^i\quad\text{in }\mathcal D'((0,T)\times(0,r)).
\end{equation}
Introduce
\[
    w:=H^1-H^2,\quad\Delta(t):=\sup_{0\leq s\leq t}|h_s^1-h_s^2|.
\]
We claim that, for every $t\in[0,T]$,
\begin{equation}\label{eq:obstacle-comparison-different-data}
\sup_{\substack{0\leq s\leq t\\x\geq0}}|H^1(s,x)-H^2(s,x)|\leq\max\{\varepsilon,\Delta(t)\}.
\end{equation}
The assertion is immediate for $t=0$. Fix $t\in(0,T]$, put $M:=\max\{\varepsilon,\Delta(t)\}$, and, for $\eta>0$, define
\[
    z_\eta(s,x):=H^1(s,x)-H^2(s,x)-M-\eta(1+x).
\]
On the open set $\{z_\eta>0\}$,
\[
    H^1>H^2+M\geq h^2+M\geq h^1.
\]
Hence $\nu_{H^1}$ vanishes there, whereas $\nu_{H^2}\geq0$, and thus
\[
\left(\partial_s-\frac12\partial_{xx}\right)z_\eta=\nu_{H^1}-\nu_{H^2}=-\nu_{H^2}\leq0\quad\text{on }\{z_\eta>0\}.
\]
At $s=0$,
\[
z_\eta(0,x)\leq\varepsilon-M-\eta(1+x)\leq0,
\]
while \eqref{eq:augmented-call-left-boundary} and $|m_1-m_2|\leq\varepsilon$ give
\[
z_\eta(s,0)=m_1-m_2-M-\eta\leq0.
\]
Finally, \eqref{eq:augmented-call-bounds} gives $0\leq H^i\leq m_i$. Choose $R_\eta<\infty$ so that $\eta(1+R_\eta)>\max\{m_1,m_2\}$. Then $z_\eta(s,R_\eta)<0$ for $0\leq s\leq t$. Applying Lemma~\ref{lem:positive-set-parabolic-maximum-principle} on $(0,t)\times(0,R_\eta)$ yields $z_\eta\leq0$. Letting $\eta\downarrow0$ and then interchanging the two solutions proves
\eqref{eq:obstacle-comparison-different-data}.\\

\emph{Step 2: Removing the freely evolved initial discrepancy.}
Let
\[
    f:=C_{\mu_0^1}-C_{\mu_0^2},\quad\delta m:=f(0)=m_1-m_2,\quad v(t,x):=(S_tf)(x),
\]
where the same notation is used for the action of $S_t$ on functions. For absorbed Brownian motion $Z$, symmetry of the killed heat kernel and optional stopping give the put--call identity
\[
    \mathbb E_x[(y-Z_t)^+]=\mathbb E_y[(Z_t-x)^+],\quad x,y\geq0.
\]
Fubini's theorem therefore yields the identity $S_tC_\mu=C_{S_t\mu}$. Consequently,
\[
v(t,x)=C_{S_t\mu_0^1}(x)-C_{S_t\mu_0^2}(x).
\]
Since $f$ is bounded and Lipschitz, $S_tf\to f$ uniformly as $t\downarrow0$. Moreover, $S_t$ is a Markov contraction, and therefore
\begin{equation}\label{eq:free-call-basic-properties}
v(0,\cdot)=f,\quad v(t,0)=\delta m,\quad \sup_{t\geq0}\|v(t,\cdot)\|_{L^\infty(\mathbb R_+)}\leq\varepsilon.
\end{equation}
Moreover, the call-function derivative formula gives, for $t>0$,
\begin{equation}\label{eq:free-call-boundary-derivative}
\partial_x^+v(t,0)
=q_{\mu_0^1}^{\mathrm{fr}}(t)
-q_{\mu_0^2}^{\mathrm{fr}}(t).
\end{equation}
The function $v$ is smooth and satisfies the heat equation $\partial_tv=\frac12\partial_{xx}v$ on $(0,\infty)\times(0,\infty)$. Now, set $Y:=w-v$ on $[0,T]\times[0,r]$. Equations \eqref{eq:heat-equation-common-strip} and \eqref{eq:free-call-basic-properties} show that
\[
\partial_tY=\frac12\partial_{xx}Y\quad\text{in }\mathcal D'((0,T)\times(0,r)),\quad Y(0,\cdot)=0,\quad Y(\cdot,0)=0.
\]
Therefore, for $Y$, the only nonzero parabolic boundary datum is $b(t):=Y(t,r)$.\\

\emph{Step 3: The Dirichlet-to-Neumann estimate and boundary feedback.}
For $k>0$, set
\[
    \lambda_k:=\frac{k^2}{2},\quad\rho_k:=\alpha\frac{k}{\sinh(kr)},\quad\|g\|_{\lambda,T}:=\sup_{0\leq t\leq T}e^{-\lambda t}|g(t)|,
\]
and put $\Delta_{\lambda,T}:=\|h^1-h^2\|_{\lambda,T}$. By \eqref{eq:obstacle-comparison-different-data},
\[
e^{-\lambda t}|w(t,r)|\leq\max\{\varepsilon,\Delta_{\lambda,T}\},\quad 0\leq t\leq T.
\]
Together with \eqref{eq:free-call-basic-properties}, this gives
\begin{equation}\label{eq:b-weighted-bound}
\|b\|_{\lambda,T}\leq\varepsilon+\max\{\varepsilon,\Delta_{\lambda,T}\}\leq2\varepsilon+\Delta_{\lambda,T}.
\end{equation}
Consider the auxiliary function
\[
\Psi(t,x):=\|b\|_{\lambda_k,T}e^{\lambda_k t}\frac{\sinh(kx)}{\sinh(kr)},
\]
which satisfies
\[
\partial_t\Psi=\frac12\partial_{xx}\Psi,\quad\Psi(t,0)=0,\quad\Psi(t,r)\geq|b(t)|.
\]
Applying Lemma~\ref{lem:positive-set-parabolic-maximum-principle} to $Y-\Psi$ and $-Y-\Psi$ gives $|Y|\leq\Psi$. Since the right derivatives at zero exist by \eqref{eq:augmented-call-closure} and \eqref{eq:free-call-boundary-derivative}, dividing by $x$ and letting $x\downarrow0$ yields
\begin{equation}\label{eq:weighted-DtN}
\sup_{0<t\leq T}e^{-\lambda_kt}|\partial_x^+Y(t,0)|\leq\frac{k}{\sinh(kr)}\|b\|_{\lambda_k,T}.
\end{equation}
Subtracting the two closure relations \eqref{eq:augmented-call-closure} and using $w=v+Y$, we obtain, for $t>0$,
\[
h_t^1-h_t^2=\alpha\bigl(\partial_x^+v(t,0)+\partial_x^+Y(t,0)\bigr).
\]
Since $h_0^1=h_0^2=0$, equations \eqref{eq:free-call-boundary-derivative}, \eqref{eq:b-weighted-bound}, and \eqref{eq:weighted-DtN} give
\begin{equation}\label{eq:key-boundary-feedback-contraction}
\Delta_{\lambda_k,T}\leq\alpha\beta_T+\rho_k\bigl(2\varepsilon+\Delta_{\lambda_k,T}\bigr).
\end{equation}
Consequently, whenever $\rho_k<1$,
\begin{equation}\label{eq:weighted-boundary-stability}
\|h^1-h^2\|_{\lambda_k,T}\leq\frac{\alpha\beta_T+2\rho_k\varepsilon}{1-\rho_k}.
\end{equation}
Because $r>0$, $\rho_k\to0$ as $k\to\infty$. Choose $k=k(\alpha,a_0)$ so that $\rho_k\leq1/2$. Then
\[
\Delta_{\lambda,T}\leq2\alpha\beta_T+2\varepsilon,
\]
and hence
\[
\sup_{0\leq t\leq T}|h_t^1-h_t^2|\leq e^{k^2T/2}\bigl(2\alpha\beta_T+2\varepsilon\bigr).
\]
Together with \eqref{eq:obstacle-comparison-different-data}, this proves
\eqref{eq:augmented-call-stability}. Finally,
\[
    C^1-C^2=H^1-H^2-(h^1-h^2),\quad q^1-q^2=\frac{h^1-h^2}{\alpha},
\]
which proves \eqref{eq:call-and-default-stability}.
\end{proof}

\begin{corollary}[Uniqueness of probabilistic solution]
\label{cor:probabilistic-FBP-uniqueness}
For every initial law satisfying Assumption \ref{ass:initial-law}, there is at most one probabilistic free boundary solution in law.  More precisely, if $(X^i,\Lambda^i)$, $i=1,2$, are two probabilistic solutions with initial law $\mu_0$, then $\Lambda^1=\Lambda^2$ on $[0,T]$ and
\[
\mathcal{L}\bigl((X_t^1)_{0\leq t\leq T}\bigr)=\mathcal{L}\bigl((X_t^2)_{0\leq t\leq T}\bigr).
\]
If the two solutions are constructed from the same $X_0$ and Brownian motion, then $X^1=X^2$ almost surely.
\end{corollary}

\begin{proof}
For $i=1,2$, let $(H^i,h^i)$ be the augmented-call pair associated with $(X^i,\Lambda^i)$ through \eqref{eq:augmented-call-definition}. By Proposition~\ref{prop:weak-augmented-call}, both pairs belong to the continuous weak augmented-call class. Set
\[
a_0:=\int_{\mathbb R_+}x\,\mu_0(dx)-\alpha>0.
\]
Since $d_T^\sharp(\mu_0,\mu_0)=0$, Theorem~\ref{thm:pairwise-FBP-stability} gives $H^1=H^2,\qquad h^1=h^2$. Since $q_t^i=\frac{h_t^i}{\alpha}$ and $\quad(\mu_t^i)^+=\partial_{xx}H^i(t,\cdot)$ in the sense of distributions on $(0,\infty)$, it follows that $\mu_t^1=\mu_t^2$ for every $t\in[0,T]$. For $t>0$, \eqref{eq:obstacle-coincidence-set} gives
\[
    \Lambda_t^i=\inf\{x\geq0:H^i(t,x)=h_t^i\}.
\]
Hence $\Lambda_t^1=\Lambda_t^2$ for every $t>0$. By Proposition~\ref{prop:FullInitialBoundaryTrace}, their canonical continuous extensions satisfy
\[
    \Lambda_0^1=\Lambda_0^2=\inf\operatorname{supp}(\mu_0)+\max\operatorname{supp}(\mu_0).
\]
Once the deterministic boundary is fixed, the stopped Skorokhod construction makes $X$ a measurable functional of $(X_0,B)$. This proves uniqueness of the path law. If the two solutions use the same $X_0$ and $B$, the same representation yields pathwise equality.
\end{proof}

\begin{remark}[Comparison with an absorption-flux distance]
\label{rem:comparison-with-flux-distance}
For every $x\geq0$, the function $y\mapsto(y-x)^+$ is $1$-Lipschitz. Consequently, $\|C_\mu-C_\nu\|_{L^\infty(\mathbb{R}_+)}\leq W_1(\mu,\nu)$. Extend the free default curves continuously to time zero by setting $q_\mu^{\mathrm{fr}}(0):=\mu(\{0\})$ and $q_\nu^{\mathrm{fr}}(0):=\nu(\{0\})$. Suppose, in addition, that these extensions are absolutely continuous on $[0,T]$ and write
\[
    j_\mu(t):=\frac{d}{dt}(S_t\mu)(\{0\}),\quad j_\nu(t):=\frac{d}{dt}(S_t\nu)(\{0\}).
\]
Then
\[
q_\mu^{\mathrm{fr}}(t)-q_\nu^{\mathrm{fr}}(t)=\mu(\{0\})-\nu(\{0\})+\int_0^t\bigl(j_\mu(s)-j_\nu(s)\bigr)\,ds,
\]
and therefore
\begin{align*}
d_T^\sharp(\mu,\nu)\leq{}&W_1(\mu,\nu)+|\mu(\{0\})-\nu(\{0\})|+\int_0^T|j_\mu(t)-j_\nu(t)|\,dt.
\end{align*}
As a result, any topology controlling the Wasserstein distance, the initial absorbed mass, and the $L^1$ distance between the free absorption fluxes also controls $d_T^\sharp$.
\end{remark}

\subsection{Positive-time classical regularity}
\label{subsec:positive-time-classical-regularity}
We now show that probabilistic solutions actually satisfy the free boundary system~\eqref{eq:FBP1} classically at positive times. Throughout this subsection, let $(X,\Lambda)$ be a probabilistic solution under Assumption~\ref{ass:initial-law}, and retain the notation $\mu_t$, $h_t$, $C$, and $H$ from Lemma~\ref{lem:WeakOrdinaryCallEvolution} and \eqref{eq:augmented-call-definition}. Set
\[
    r_*:=m_0-\alpha>0,\quad
    \mathcal D:=\{(t,x):0<t<T,\ 0<x<\Lambda_t\}.
\]
By Lemma~\ref{lem:boundary-uniform-positivity}, $\Lambda_t\ge r_*$. We first obtain smoothness of the loss flux at the fixed boundary, then prove a positive density bound near the moving boundary. This bound excludes singular points of the associated parabolic obstacle problem; boundary Harnack estimates then yield smoothness of the moving boundary.

\begin{lemma}[Smooth loss flux and density representation]
\label{lem:smooth-loss-flux}
The function $h$ belongs to $C^\infty((0,T))$. For every $t\in(0,T)$,
\begin{equation}\label{eq:positive-time-density-representation}
\mu_t=\frac{h_t}{\alpha}\delta_0+u(t,x)\mathbf 1_{\{0<x<\Lambda_t\}}\,\d x,
\end{equation}
where $u=\partial_{xx}H$ is smooth and strictly positive in $\mathcal D$ and satisfies $\partial_tu=\frac12\partial_{xx}u$ there. It is also smooth up to $x=0$, locally at positive times, with
\begin{equation}\label{eq:smooth-positive-loss-flux}
u(t,0)=0,\quad g(t):=\dot h_t=\frac{\alpha}{2}\partial_xu(t,0)>0,\quad 0<t<T.
\end{equation}
Moreover, $C_x$ is jointly continuous on $(0,T)\times(0,\infty)$, $C_{xx}$ and $C_t$ are locally bounded there, and
\begin{equation}\label{eq:call-spatial-smooth-fit}
C(t,\Lambda_t)=C_x(t,\Lambda_t)=0,\qquad 0<t<T.
\end{equation}
\end{lemma}

\begin{proof}
Equations~\eqref{eq:weak-augmented-call-measure} and
\eqref{eq:augmented-call-left-boundary} imply
\[
    (\partial_t-\frac12\partial_{xx})(H-m_0)=0\quad\text{on }(0,T)\times(0,r_*),\quad H(t,0)-m_0=0.
\]
To justify odd reflection, write $w=H-m_0$ and note that
\[
    |w(t,x)|=|C(t,x)-C(t,0)|\le x.
\]
For a smooth test function $\varphi$ compactly supported in $(0,T)\times(-r_*,r_*)$, put $\psi(t,x)=\varphi(t,x)-\varphi(t,-x)$ for $x>0$. Test the equation for $w$ against $\chi_\delta(x)\psi(t,x)$, where $\chi_\delta$ vanishes for $x\le\delta$, equals one for $x\ge2\delta$, and has first and second derivatives bounded by constants times $\delta^{-1}$ and $\delta^{-2}$, respectively. The terms containing $\chi_\delta'\psi_x$ and $\chi_\delta''\psi$ tend to zero: their integrals are $O(\delta)$ because $|w|\le x$ and $\psi(t,x)=O(x)$. Letting $\delta\downarrow0$ proves that the odd extension of $w$ solves the heat equation distributionally across zero. Interior heat regularity therefore makes $H$ smooth up to $x=0$, locally at positive times. By \eqref{eq:augmented-call-closure},
\[
    h_t=\alpha\bigl(1+H_x(t,0)\bigr),
\]
so $h\in C^\infty((0,T))$.

By Proposition~\ref{prop:weak-augmented-call}, $(\partial_t-\frac12\partial_{xx})H=0$ in $\mathcal D$. Therefore, $u:=H_{xx}$ is smooth there and satisfies the heat equation. Convexity gives $u\ge0$. For every nonempty open interval $J$ compactly contained in $(0,\Lambda_t)$, the spatial distributional identity $H_{xx}(t,\cdot)=\mu_t^+$ and Lemma~\ref{lem:support-identification} give
\[
    \int_Ju(t,x)\,dx=\mu_t(J)>0.
\]
If $u$ vanished at an interior point, the strong parabolic minimum principle on a fixed cylinder contained in $\mathcal D$ would force it to vanish on a spatial interval at an earlier time, contradicting this identity. Thus, $u>0$ in $\mathcal D$. The odd reflection above makes $u$ odd across zero, so $u(t,0)=0$. Fix $t_0\in(0,T)$ and choose $b\in(0,r_*)$ and $d>0$ such that $[t_0-d,t_0+d]\subset(0,T)$. On the fixed cylinder $(t_0-d,t_0+d)\times(0,b)$, $u$ is smooth up to zero, positive in the interior, and zero at $x=0$. The parabolic Hopf lemma therefore gives $u_x(t_0,0)>0$. Since $t_0$ was arbitrary, differentiation of the closure relation yields
\[
\partial_xu(t,0)>0,\quad\dot h_t=\alpha H_{xt}(t,0)=\frac{\alpha}{2}H_{xxx}(t,0)=\frac{\alpha}{2}\partial_xu(t,0).
\]
Equation~\eqref{eq:WeakCallEquationForAdvance} now reads
\begin{equation}\label{eq:call-obstacle-smooth-source}
(\partial_t-\frac12\partial_{xx})C=-g(t)\mathbf 1_{\{C>0\}}\quad\text{on }(0,T)\times(0,\infty).
\end{equation}
Its right-hand side is locally bounded, so interior parabolic $L^p$ estimates give $C\in W^{2,1}_{p,\mathrm{loc}}$ for every finite $p>1$. Parabolic Sobolev embedding with $p>3$, the parabolic dimension being $1+2=3$, yields joint continuity of $C_x$. Since $C(t,x)=0$ for $x\ge\Lambda_t$, this proves \eqref{eq:call-spatial-smooth-fit}. For every $a>0$, the one-sided derivatives of a call function satisfy
\[
    C_x^+(t,a)=-\mu_t((a,\infty)),\quad C_x^-(t,a)=-\mu_t([a,\infty)).
\]
Their continuity at $a=\Lambda_t$ excludes an atom there. The interior density identity, the support identification, and $\mu_t(\{0\})=h_t/\alpha$ now prove \eqref{eq:positive-time-density-representation}.

We obtain the bounded-derivative estimate by reducing the equation to one with a constant source. Fix a cylinder compactly contained in $(0,T)\times(0,\infty)$, set $\Omega_C:=\{C>0\}$, and define a smooth function $\rho$ on its time interval with $\rho'=g-1$. For $W(t,x):=C(t,x)+\rho(t)$, equation \eqref{eq:call-obstacle-smooth-source} gives
\[
    \frac12W_{xx}-W_t=1\quad\text{a.e. on }\Omega_C.
\]
On the open set $x>\Lambda_t$, both $C_{xx}$ and $C_t$ vanish. The graph $x=\Lambda_t$ has two-dimensional Lebesgue measure zero by Fubini's theorem. Hence, almost everywhere on the complementary set,
\[
    W_{xx}=0,\qquad W_t=g(t)-1,
\]
so both derivatives are bounded there. The strong Sobolev regularity of $W$ follows from that of $C$. After translation and parabolic rescaling, all hypotheses of \cite[Theorem~1.1]{FigalliShahgholian2015} therefore hold. That theorem requires no sign condition on $W$ and no identification of $\Omega_C$ with its nonzero set. It gives $W_{xx},W_t\in L^\infty_{\mathrm{loc}}$, and therefore $C_{xx},C_t\in L^\infty_{\mathrm{loc}}$. In particular, $C$ is locally Lipschitz in time and locally $C^{1,1}$ in space.
\end{proof}

\begin{lemma}[Non-degeneracy near the free boundary]
\label{lem:positive-time-edge-density-bounds}
Set $\ell:=r_*/2$. For every $0<\varepsilon<S<T$, there exist constants $0<c_{\varepsilon,S}\le A_{\varepsilon,S}<\infty$ such that
\begin{equation}\label{eq:positive-time-edge-density-bounds}
c_{\varepsilon,S}\le u(t,x)\le A_{\varepsilon,S},\quad \varepsilon\le t\le S,\quad \ell\le x<\Lambda_t.
\end{equation}
Consequently,
\begin{equation}\label{eq:positive-time-quadratic-call-bound}
C(t,\Lambda_t-z)\ge\frac{c_{\varepsilon,S}}{2}z^2,\quad \varepsilon\le t\le S,\quad 0\le z\le\ell.
\end{equation}
\end{lemma}

\begin{proof}
First, $H$ is non-decreasing in time. Indeed, \eqref{eq:augmented-call-second-derivative} and \eqref{eq:weak-augmented-call-measure} give the non-negative distribution
\begin{equation}\label{eq:augmented-call-time-monotonicity}
\partial_tH=\frac12\mu_t^+(dx)\,dt+\mathbf 1_{\{x\ge\Lambda_t\}}\,dx\,dh_t\ge0.
\end{equation}
Testing against a non-negative smooth spatial function shows that its integral against $H(t,\cdot)$ is non-decreasing in time. Joint continuity and localization of the spatial test function then give this monotonicity for every fixed $x$. Choose $0<s<\varepsilon$ and $\delta_0>0$ with $S+\delta_0<T$, and set
\[
R>\max_{t\in[s,S+\delta_0]}\Lambda_t,\quad\gamma:=\min_{t\in[s,S+\delta_0]}g(t)>0.
\]
For $0<\delta<\delta_0$, define on $[s,S]\times[0,R]$
\[
    D_\delta(t,x):=\frac{H(t+\delta,x)-H(t,x)}{\delta}\ge0.
\]
On the coincidence set $C(t,x)=0$, the inequality $H\ge h$ gives
\begin{equation}\label{eq:loss-time-quotient-contact-bound}
D_\delta(t,x)\ge\frac{h_{t+\delta}-h_t}{\delta}\ge\gamma.
\end{equation}
Globally on the open rectangle, equation \eqref{eq:weak-augmented-call-measure} gives
\[
    (\partial_t-\frac12\partial_{xx})D_\delta=\frac{g(t+\delta)\mathbf 1_{\{x\ge\Lambda_{t+\delta}\}}-g(t)\mathbf 1_{\{x\ge\Lambda_t\}}}{\delta}.
\]
In particular, this right-hand side is bounded for each fixed $\delta$. Restricted to the open set $C(t,x)>0$, the second term vanishes, so
\begin{equation}\label{eq:loss-time-quotient-supercaloric}
    (\partial_t-\frac12\partial_{xx})D_\delta(t,x)
    =\frac{g(t+\delta)}{\delta}
       \mathbf 1_{\{x\ge\Lambda_{t+\delta}\}}\ge0.
\end{equation}
Let $v$ be the solution to
\[
\begin{cases}
    (\partial_t-\frac12\partial_{xx})v=0,&s<t<S,\quad 0<x<R,\\
    v(s,x)=0,&0\le x\le R,\\
    v(t,0)=0,\quad v(t,R)=\gamma(1-e^{-(t-s)}),&s\le t\le S.
\end{cases}
\]
The maximum principle gives $0\le v\le\gamma$ and $v(t,x)>0$ for $t>s$ and $0<x\le R$. Set $w_\delta:=v-D_\delta$. At the initial time, $w_\delta(s,\cdot)\le0$ because $D_\delta\ge0$. At $x=0$, both terms vanish, while at $x=R$,
\[
    D_\delta(t,R)=\frac{h_{t+\delta}-h_t}{\delta}\ge\gamma\ge v(t,R).
\]
Thus $w_\delta$ is non-positive on the parabolic boundary. Moreover, \eqref{eq:loss-time-quotient-contact-bound} implies
\[
    \{v-D_\delta>0\}\subseteq\{C>0\}.
\]
For each fixed $\delta$, $w_\delta$ has local parabolic Sobolev regularity and $(\partial_t-\frac12\partial_{xx})w_\delta\in L^\infty$. The Kato inequality and \eqref{eq:loss-time-quotient-supercaloric} therefore give
\[
    (\partial_t-\frac12\partial_{xx})(w_\delta)^+\le\mathbf 1_{\{w_\delta>0\}}(\partial_t-\frac12\partial_{xx})w_\delta\le0.
\]
The weak maximum principle yields $(w_\delta)^+=0$, hence $D_\delta\ge v$. This comparison uses only bounded right-hand sides. For a fixed $(t,x)\in\mathcal D$, continuity of $\Lambda$ keeps $(t+\delta,x)$ in $\mathcal D$ for all sufficiently small $\delta$. Interior smoothness of $H$ then gives, on letting $\delta\downarrow0$,
\[
    u(t,x)=2H_t(t,x)\ge2v(t,x).
\]
The minimum of $v$ on $[\varepsilon,S]\times[\ell,R]$ is strictly positive, proving the lower bound in \eqref{eq:positive-time-edge-density-bounds}. For the upper bound, Lemma~\ref{lem:smooth-loss-flux} gives a finite essential bound for $C_{xx}$ on the ambient compact rectangle $[\varepsilon,S]\times[\ell,R]$, which crosses the free boundary. Continuity of $u=C_{xx}$ inside $\mathcal D$ upgrades that bound to a pointwise bound there. Finally,
\eqref{eq:positive-time-density-representation} gives
\[
C(t,\Lambda_t-z)=\int_{\Lambda_t-z}^{\Lambda_t}(y-\Lambda_t+z)\,u(t,y)\,dy\ge\frac{c_{\varepsilon,S}}{2}z^2,
\]
because $z\le\ell$ and $\Lambda_t\ge2\ell$.
\end{proof}

\begin{theorem}[Positive-time classical regularity]
\label{thm:positive-time-classical-regularity}
Under Assumption~\ref{ass:initial-law}, every probabilistic solution satisfies $\Lambda\in C^\infty((0,T))$. The density in \eqref{eq:positive-time-density-representation} is smooth up to both spatial boundaries, locally at positive times; or equivalently,
\[
    (t,z)\longmapsto u(t,z\Lambda_t)\quad\text{belongs to }C^\infty((0,T)\times[0,1]).
\]
The pair $(u,\Lambda)$ satisfies \eqref{eq:FBP1} classically for $0<t<T$. In particular,
\begin{equation}\label{eq:classical-positive-time-boundary-identities}
\begin{aligned}
u(t,0)&=0,\\
u(t,\Lambda_t)&=2\dot h_t=\alpha\partial_xu(t,0)>0,\\
\frac12\partial_xu(t,\Lambda_t)&=-\dot\Lambda_t\,u(t,\Lambda_t),
\end{aligned}
\end{equation}
where the traces at $\Lambda_t$ are taken from the interior.
\end{theorem}

\begin{proof}
Fix $0<\varepsilon_0<\varepsilon<S<S_0<T$, put $\ell=r_*/2$, and let $c,A>0$ be the constants in Lemma~\ref{lem:positive-time-edge-density-bounds} on $[\varepsilon_0,S_0]$. We work locally near the graph $\Gamma:=\{(t,\Lambda_t):\varepsilon_0<t<S_0\}$, away from $x=0$. By Lemma~\ref{lem:smooth-loss-flux}, $C\in W^{2,1}_{\infty,\mathrm{loc}}$, $g=\dot h$ is smooth and strictly positive, and \eqref{eq:call-obstacle-smooth-source} is the parabolic obstacle equation
\[
    C\ge0,\quad\left(\partial_t-\frac12\partial_{xx}\right)C=-g(t)\mathbf 1_{\{C>0\}}.
\]
Moreover, \eqref{eq:call-spatial-smooth-fit} and \eqref{eq:positive-time-edge-density-bounds} give, for $\varepsilon_0\le t\le S_0$ and $0\le z\le\ell$,
\[
    \frac c2z^2\le C(t,\Lambda_t-z)\le\frac A2z^2,\quad cz\le -C_x(t,\Lambda_t-z)\le Az.
\]

\smallskip
\noindent\textbf{Step 1: Every point of $\Gamma$ is regular.}
Fix $(t_0,x_0)\in\Gamma$ and consider the parabolic rescalings
\[
    C_r(s,y):=r^{-2}C(t_0+r^2s,x_0+ry).
\]
The local bounds on $C_t$ and $C_{xx}$, together with $C(t_0,x_0)=C_x(t_0,x_0)=0$, give locally uniform bounds
\[
    0\le C_r(s,y)\le \frac A2y^2+L|s|
\]
for some $L<\infty$, and locally uniform subsequential convergence, including on the time slice $s=0$. The blow-up classification for the parabolic obstacle problem applies after normalizing the diffusion coefficient and $g(t_0)$; see \cite[Theorem~1.2 and equation~(1.3)]{LindgrenMonneau2015}. In particular, a singular blow-up has the form $P(s,y)=ay^2+bs$ for $s\le0$, whereas a degenerate blow-up is zero. However, every blow-up limit $C_0$ satisfies
\[
    C_0(0,y)=0\quad\text{for }y\ge0,\quad C_0(0,y)\ge\frac c2y^2\quad\text{for }y<0.
\]
These identities exclude the zero blow-up. They also exclude every singular polynomial: the first forces $a=0$, contradicting the second. Thus $(t_0,x_0)$ is a regular point, with blow-up $g(t_0)(y_-)^2$, where $y_-:=\max\{-y,0\}$. By \cite[Theorem~1.7]{LindgrenMonneau2015}, $\Gamma$ is locally $C_p^1$, where $C_p^\beta$ denotes parabolic H\"older regularity with time having weight two. In addition, \cite[Corollary~6.4]{BlanchetDolbeaultMonneau2006} gives
\[
    \lim_{\substack{(t,x)\to(t_0,\Lambda_{t_0})\\(t,x)\in\mathcal D}}C_t(t,x)=0.
\]
The latter result allows a positive H\"older continuous source depending on both space and time, so its hypotheses hold for $g(t)$. Consequently, both $C_t$ and $C_x$ have continuous zero traces on $\Gamma$. No sign condition on $C_t$ is needed.

\smallskip
\noindent\textbf{Step 2: Boundary Harnack and the boundary velocity.}
Set $p:=-C_x$ and $q:=C_t$ in $\mathcal D$. Differentiating \eqref{eq:call-obstacle-smooth-source} in the interior gives
\[
\left(\partial_t-\frac12\partial_{xx}\right)p=0,\quad\left(\partial_t-\frac12\partial_{xx}\right)q=-g'(t).
\]
Both functions vanish continuously on $\Gamma$ and are locally bounded up to it. In a sufficiently small neighborhood of any point of $\Gamma$, the density lower bound yields
\[
    p(t,x)=\int_x^{\Lambda_t}u(t,y)\,dy\ge c(\Lambda_t-x)\ge c\,d_p(t,x),
\]
where $d_p$ is the parabolic distance to the local free boundary. The last inequality follows by using the same-time boundary point $(t,\Lambda_t)$. Therefore, \cite[Theorem~1.2]{Kukuljan2022}, applied with its parameter $\varepsilon=1/2$, gives
\[
    R:=\frac{q}{p}=\frac{C_t}{-C_x}\in C_p^{1/2}
\]
locally up to $\Gamma$.

To identify the boundary trace of $R$ without presupposing differentiability of $\Lambda$, fix a compact time interval $J$ in such a neighborhood. For every sufficiently small $a>0$, the equation
\[
    C(t,\lambda_a(t))=a,\qquad t\in J,
\]
defines a unique smooth curve in the interior collar. The preceding quadratic bound and the implicit function theorem give
\[
    0\le\Lambda_t-\lambda_a(t)\le\sqrt{2a/c},\quad\lambda_a'(t)=R(t,\lambda_a(t)).
\]
Hence, $\lambda_a\to\Lambda$ and $\lambda_a'\to R(\cdot,\Lambda_\cdot)$ uniformly on $J$. Passing to the limit in the integral identity for $\lambda_a$ proves that $\Lambda\in C^1(J)$ and
\begin{equation}\label{eq:obstacle-boundary-velocity-quotient}
\dot\Lambda_t=R(t,\Lambda_t),
\end{equation}
where the right-hand side denotes the interior trace. Since $\Lambda$ is now locally Lipschitz and $R\in C_p^{1/2}$, this trace is $1/4$-H\"older continuous in time. Therefore, $\Lambda\in C^{1,1/4}_{\mathrm{loc}}((0,T))$.

\smallskip
\noindent\textbf{Step 3: Higher regularity.}
Suppose inductively that $\Lambda\in C^{k,1/4}_{\mathrm{loc}}((0,T))$ for an integer $k\ge1$. Then $\Gamma$ is locally $C_p^{2k+1/2}$. The coefficients in the equations for $p,q$ are constant, their sources $0,-g'$ are smooth, and the lower bound for $p$ remains valid. Applying \cite[Theorem~1.3]{Kukuljan2022} with $\beta=2k+1/2$ therefore gives $R\in C_p^{2k+1/2}$ locally up to $\Gamma$. Its trace $t\mapsto R(t,\Lambda_t)$ belongs to $C^{k,1/4}$. Indeed, $k$ time derivatives of this composition involve only derivatives of $R$ of parabolic order at most $2k$ and derivatives of $\Lambda$ of order at most $k$; the top-order terms are $1/4$-H\"older continuous along the locally Lipschitz graph. The chain rule at the boundary follows by first working on the interior curves $x=\Lambda_t-\eta$ and then letting $\eta\downarrow0$. Equation~\eqref{eq:obstacle-boundary-velocity-quotient} now yields $\Lambda\in C^{k+1,1/4}_{\mathrm{loc}}((0,T))$. Induction proves $\Lambda\in C^\infty((0,T))$.

\smallskip
\noindent\textbf{Step 4: Density regularity and the boundary identities.}
Near the now smooth moving boundary, $C$ solves the linear equation
\[
    C_t-\frac12C_{xx}=-g(t)\quad\text{in }\mathcal D,\quad C(t,\Lambda_t)=0.
\]
Local boundary Schauder estimates imply that $C$, and hence $u=C_{xx}$, is smooth up to that boundary from the interior. Lemma~\ref{lem:smooth-loss-flux} gives smoothness up to $x=0$, and interior heat regularity covers the region between the two boundary collars. A finite cover on each compact positive-time interval therefore gives
\[
    (t,z)\longmapsto u(t,z\Lambda_t)\in C^\infty((0,T)\times[0,1]).
\]
Taking the interior trace in \eqref{eq:call-obstacle-smooth-source} and using the zero trace of $C_t$ gives
\[
    u(t,\Lambda_t)=C_{xx}(t,\Lambda_t)=2g(t)=\alpha\partial_xu(t,0).
\]
Differentiating $C_x(t,\Lambda_t)=0$ gives
\[
    C_{tx}(t,\Lambda_t)+\dot\Lambda_t C_{xx}(t,\Lambda_t)=0.
\]
Since $C_{tx}=\frac12C_{xxx}$ in $\mathcal D$, this is the last identity in \eqref{eq:classical-positive-time-boundary-identities}. The heat equation for $u$ and its zero trace at $x=0$ were established in Lemma~\ref{lem:smooth-loss-flux}, completing the proof.
\end{proof}

\begin{remark}[Time endpoints]
\label{rem:classical-regularity-time-endpoints}
No classical regularity at $t=0$ is asserted: the initial law in Assumption~\ref{ass:initial-law} may be singular. The preceding conclusions also hold up to $t=T$ with one-sided time derivatives. Indeed, one may construct a solution from the same initial law on a longer horizon using Section~\ref{sec:continuous}. By Corollary~\ref{cor:probabilistic-FBP-uniqueness}, its restriction agrees with the given solution, including the canonically normalized boundary. The terminal time $T$ is then an interior time for the extended solution.
\end{remark}

\section{Back to the Mean-Field Control Problem}\label{sec:Back2MFC}
\subsection{Stability of the value function}
\begin{proposition}[Stability of the value function]
\label{prop:survival-value-stability}
Under the assumptions of Theorem~\ref{thm:pairwise-FBP-stability}, let $f$ be any payoff allowed in item~2 of Definition~\ref{def:stochastic-orders}, and let flows $(\mu_t^i)_{t\in[0,T]}$ be corresponding to the unique probabilistic solutions to \eqref{eq:FBP1} such that
\[
    V_f(0,\mu_0^i)=\langle f,\mu_T^i\rangle,\quad i=1,2.
\]
Then there exists $C=C(\alpha,a_0,T)<\infty$ such that
\begin{equation}\label{eq:value-function-stability}
\bigl|V_f(0,\mu_0^1)-V_f(0,\mu_0^2)\bigr|\leq C\bigl(f(1)-f(0)\bigr)d_T^\sharp(\mu_0^1,\mu_0^2).
\end{equation}
\end{proposition}

\begin{proof}
Set
\[
d:=d_T^\sharp(\mu_0^1,\mu_0^2),\quad r:=\frac{a_0}{2},\quad w:=H^1-H^2,\quad\delta m:=m_1-m_2.
\]
Since $|\delta m|\leq d$, estimate \eqref{eq:augmented-call-stability} gives, after enlarging the constant $C$ if necessary,
\[
    \|w-\delta m\|_{L^\infty([0,T]\times\mathbb R_+)}\leq Cd.
\]
On the common continuation strip $[0,T]\times[0,r]$, \eqref{eq:heat-equation-common-strip} and the identity $H^i(t,0)=m_i$ show that $z:=w-\delta m$ solves
\[
    \partial_tz=\frac12\partial_{xx}z\quad\text{on }(0,T)\times(0,r),\quad z(t,0)=0.
\]
Odd reflection across $x=0$, followed by the interior gradient estimate for the heat equation on $[T/2,T)\times[-r/2,r/2]$, yields
\[
    \sup_{\substack{T/2\leq t<T\\|x|\leq r/2}}|\partial_xz(t,x)|\leq Cd.
\]
Consequently,
\[
    |z(T,a)|\leq Cda,\quad 0<a\leq \frac r2.
\]
Here, the last estimate is passed from $t<T$ to $T$ by continuity. Moreover,
\[
    D_{\mu_T^i}(a)=m_i-H^i(T,a),
\]
and hence
\[
    D_{\mu_T^1}(a)-D_{\mu_T^2}(a)=-z(T,a).
\]
Combining the local estimate with the global bound on $z$, and enlarging $C=C(\alpha,a_0,T)$ once more, gives
\begin{equation}\label{eq:terminal-truncated-moment-stability}
\bigl|D_{\mu_T^1}(a)-D_{\mu_T^2}(a)\bigr|\leq Cd(a\wedge1),\quad a>0.
\end{equation}
Also, \eqref{eq:call-and-default-stability} gives
\[
    |q_T^1-q_T^2|+|m_{\mu_T^1}-m_{\mu_T^2}|\leq Cd.
\]
Write the decomposition \eqref{eq:payoff-mixture-decomposition} as
\[
f(x)=f(0)+J_f\mathbf 1_{\{x>0\}}+\ell x+\int_{(0,\infty)}(x\wedge a)\,\gamma(da).
\]
Then, by Fubini's theorem,
\[
\langle f,\mu_T^i\rangle=f(0)+J_f(1-q_T^i)+\ell m_{\mu_T^i}+\int_{(0,\infty)}D_{\mu_T^i}(a)\,\gamma(da).
\]
It follows from \eqref{eq:terminal-truncated-moment-stability} that
\begin{align*}
\bigl|V_f(0,\mu_0^1)-V_f(0,\mu_0^2)\bigr|\leq Cd\left(J_f+\ell+\int_{(0,\infty)}(a\wedge1)\,\gamma(da)\right)=Cd\bigl(f(1)-f(0)\bigr),
\end{align*}
where the last identity follows by evaluating \eqref{eq:payoff-mixture-decomposition} at $x=1$.
\end{proof}

\begin{proposition}[Positive-time stability of the free boundary]
\label{prop:support-boundary-stability}
Fix $\alpha>0$, $T>0$, and an initial law $\mu_*$ satisfying Assumption~\ref{ass:initial-law}. For every $\tau\in(0,T]$, there exist $\rho,K>0$, depending only on $\alpha,\mu_*,\tau,T$, such that the following holds. If $\mu_0^1,\mu_0^2$ satisfy Assumption~\ref{ass:initial-law} and
\[
    d_T^\sharp(\mu_0^i,\mu_*)<\rho,\qquad i=1,2,
\]
then their corresponding free boundaries satisfy
\begin{equation}\label{eq:support-boundary-stability}
    \sup_{\tau\le t\le T}|\Lambda_t^1-\Lambda_t^2|
    \le Kd_T^\sharp(\mu_0^1,\mu_0^2).
\end{equation}
Consequently, if $\mu_0^n,\mu_0$ satisfy Assumption~\ref{ass:initial-law} and $W_1(\mu_0^n,\mu_0)\to0$, then the corresponding free boundaries satisfy
\[
    \sup_{\tau\le t\le T}|\Lambda_t^n-\Lambda_t|\longrightarrow0\quad\text{for every }\tau\in(0,T].
\]
\end{proposition}

\begin{proof}
We first consider any two initial laws satisfying Assumption~\ref{ass:initial-law} with a common reserve margin
\[
    m_i:=\int_{\mathbb R_+}x\,\mu_0^i(dx)\ge\alpha+a_0,\quad i=1,2,
\]
where $a_0>0$. Write
\[
    d:=d_T^\sharp(\mu_0^1,\mu_0^2),\quad C^i(t,x):=H^i(t,x)-h_t^i,\quad g^i(t):=\dot h_t^i.
\]
By Lemma~\ref{lem:smooth-loss-flux}, the functions $g^i$ are smooth and strictly positive at positive times. Moreover, \eqref{eq:call-obstacle-smooth-source} and \eqref{eq:CallPositiveExactlyBelowBoundary} give
\[
    (\partial_t-\frac12\partial_{xx})C^i=-g^i(t)\mathbf1_{\{C^i>0\}},\quad\{x\ge0:C^i(t,x)>0\}=[0,\Lambda_t^i).
\]
Regularity at $t=T$ is understood as in Remark~\ref{rem:classical-regularity-time-endpoints}.

\medskip
\noindent
\emph{Step 1: Stability of the loss flux.}
Set $r:=a_0/2$ and
\[
    w:=H^1-H^2-(m_1-m_2).
\]
By \eqref{eq:heat-equation-common-strip} and \eqref{eq:augmented-call-left-boundary},
\[
    (\partial_t-\frac12\partial_{xx})w=0\quad\text{on }(0,T)\times(0,r),\quad w(t,0)=0.
\]
Since $|m_1-m_2|\le d$, estimate \eqref{eq:augmented-call-stability} gives
\[
    \|w\|_{L^\infty([0,T]\times[0,r])}\le K_0d,
\]
where $K_0$ depends only on $\alpha,a_0,T$. Odd reflection across $x=0$, as in the proof of Lemma~\ref{lem:smooth-loss-flux}, followed by interior heat estimates yields, for every $s\in(0,T)$,
\[
    \sup_{s\le t<T}|\partial_t\partial_xw(t,0)|\le K_sd,
\]
with $K_s=K_s(\alpha,a_0,s,T)$. Differentiating \eqref{eq:augmented-call-closure} gives $g^1-g^2=\alpha\partial_t\partial_xw(\cdot,0)$. Passing to $t=T$ by continuity, we obtain
\begin{equation}\label{eq:positive-time-loss-flux-stability}
\sup_{s\le t\le T}|g^1(t)-g^2(t)|\le K_sd.
\end{equation}
Together with \eqref{eq:call-and-default-stability}, this shows that
\[
    D_s:=\sup_{\substack{s\le t\le T\\x\ge0}}|C^1(t,x)-C^2(t,x)|\le K_0d,\quad B_s:=\sup_{s\le t\le T}|g^1(t)-g^2(t)|\le K_sd.
\]

\medskip
\noindent
\emph{Step 2: A comparison estimate for the support boundaries.}
Fix $0<s<\theta\le T$. For this step, regard the first solution as a reference and abbreviate
\[
(C,g,\Lambda):=(C^1,g^1,\Lambda^1),\quad(\widetilde C,\widetilde g,\widetilde\Lambda):=(C^2,g^2,\Lambda^2).
\]
Suppose that, for some $R,\gamma>0$,
\[
    \Lambda_t\le R,\quad g(t)\ge\gamma,\quad s\le t\le T.
\]
These bounds are imposed only on the reference solution. Set
\[
    a:=\frac{a_0}{4},\quad p:=\frac{a_0}{2R},\quad Q(t,x):=-C_x(t,x)=\mu_t^1((x,\infty)).
\]
The support bound gives $C(t,x)\le RQ(t,x)$ for $x>0$. Also, by \eqref{eq:augmented-call-bounds} and the spatial $1$-Lipschitz property of call functions,
\[
Q(t,2a)\ge\frac{C(t,2a)}R\ge\frac{a_0-2a}{R}=p.
\]
Define
\begin{equation}\label{eq:support-boundary-comparison-parameters}
\kappa_s:=\frac{B_s+D_s/(\theta-s)}{\gamma},\quad\sigma_s:=\frac{\kappa_s R+D_s/p}{1-\kappa_s}.
\end{equation}
Assume $D_s+B_s$ is small enough that
\[
    \kappa_s<\frac12,\quad\sigma_s<a.
\]
For $s\le t\le T$ and $x\ge a$, set
\[
    e_s(t):=D_s\left(\frac{\theta-t}{\theta-s}\right)^+,
\]
and introduce the barriers
\begin{equation}\label{eq:support-boundary-call-barriers}
\begin{aligned}
    U_s(t,x)&:=(1-\kappa_s)C(t,x-\sigma_s)+e_s(t),\\
    V_s(t,x)&:=(1+\kappa_s)C(t,x+\sigma_s)-e_s(t).
\end{aligned}
\end{equation}
We claim that
\[
    V_s\le\widetilde C\le U_s\quad\text{on }[s,T]\times[a,\infty).
\]
We first verify the initial and lateral inequalities. Since $Q(t,\cdot)$ is non-increasing,
\[
    C(t,x-\sigma_s)-C(t,x)\ge\sigma_sQ(t,x),\quad C(t,x)-C(t,x+\sigma_s)\ge\sigma_s Q(t,x+\sigma_s).
\]
Consequently,
\begin{align*}
    (1-\kappa_s)C(t,x-\sigma_s)-C(t,x)&\ge\bigl((1-\kappa_s)\sigma_s-\kappa_sR\bigr)Q(t,x)=\frac{D_s}pQ(t,x)\ge0,\\
    C(t,x)-(1+\kappa_s)C(t,x+\sigma_s)&\ge(\sigma_s-\kappa_s R)Q(t,x+\sigma_s)\ge0.
\end{align*}
Here, the definition of $\sigma_s$ also gives $\sigma_s-\kappa_s R\ge D_s/p$. At $x=a$, both differences are at least $D_s$, because
\[
    Q(t,a)\ge p,\quad Q(t,a+\sigma_s)\ge Q(t,2a)\ge p.
\]
It follows that
\[
    V_s(t,a)\le\widetilde C(t,a)\le U_s(t,a).
\]
At $t=s$, the required ordering follows from $e_s(s)=D_s$ and $|\widetilde C-C|\le D_s$.

Next, on the open set $\{\widetilde C>U_s\}$, we have $\widetilde C>0$ because $U_s\ge0$. Thus,
\begin{align*}
(\partial_t-\frac12\partial_{xx})(\widetilde C-U_s)&=-\widetilde g
      +(1-\kappa_s)g\,\mathbf1_{\{C(t,x-\sigma_s)>0\}}-e_s'\\
    &\le B_s-\kappa_s\gamma+\frac{D_s}{\theta-s}=0.
\end{align*}
Similarly, on $\{V_s>\widetilde C\}$, we have $C(t,x+\sigma_s)>0$, so
\begin{align*}
(\partial_t-\frac12\partial_{xx})(V_s-\widetilde C)&=-(1+\kappa_s)g-e_s'+\widetilde g\,\mathbf1_{\{\widetilde C>0\}}\\
    &\le B_s-\kappa_s\gamma+\frac{D_s}{\theta-s}=0.
\end{align*}
These inequalities hold distributionally; the function $e_s$ is Lipschitz and satisfies $-e_s'\le D_s/(\theta-s)$ almost everywhere. Now, choose a finite right endpoint lying beyond both support boundaries and the shifted reference boundary throughout $[s,T]$. Such an endpoint exists by continuity of the individual boundaries. At this endpoint,
\[
    V_s=-e_s\le0=\widetilde C\le e_s=U_s.
\]
Lemma~\ref{lem:positive-set-parabolic-maximum-principle}, applied to $\widetilde C-U_s$ and $V_s-\widetilde C$, proves the claimed comparison. The size of this rectangle does not enter the estimates. For $t\ge\theta$, we have $e_s(t)=0$, and therefore
\[
    (1+\kappa_s)C(t,x+\sigma_s)
    \le\widetilde C(t,x)
    \le(1-\kappa_s)C(t,x-\sigma_s).
\]
Since $\Lambda_t\ge a_0=4a$ by Lemma~\ref{lem:boundary-uniform-positivity} and $\sigma_s<a$, identification of the positivity sets gives
\[
    \Lambda_t-\sigma_s\le\widetilde\Lambda_t\le\Lambda_t+\sigma_s,\quad \theta\le t\le T.
\]
Using $\kappa<1/2$ in \eqref{eq:support-boundary-comparison-parameters}, we conclude that
\begin{equation}\label{eq:support-boundary-comparison-bound}
\sup_{\theta\le t\le T}|\widetilde\Lambda_t-\Lambda_t|\le2\left[\frac R\gamma B_s+\left(\frac R{\gamma(\theta-s)}+\frac1p\right)D_s\right].
\end{equation}

\medskip
\noindent
\emph{Step 3: Uniformity in a neighborhood of the reference law.}
Let
\[
m_*:=\int_{\mathbb R_+}x\,\mu_*(dx),\quad a_0:=\frac{m_*-\alpha}{2}>0.
\]
For every initial law $\nu$,
\[
    \left|\int_{\mathbb R_+}x\,\nu(dx)-m_*\right|\le d_T^\sharp(\nu,\mu_*).
\]
Hence, all sufficiently close initial laws have the common reserve margin $a_0$.

Write $(\Lambda^*,g^*)$ for the boundary and loss flux associated with $\mu_*$. First apply Step~2 with this solution as reference, taking $s=\tau/4$ and $\theta=\tau/2$. The reference boundary is bounded on $[\tau/4,T]$, and the reference flux has a strictly positive minimum there. Step~1 therefore implies, after choosing $\rho>0$ sufficiently small, that
\[
\sup_{\tau/2\le t\le T}|\Lambda_t^\nu-\Lambda_t^*|\le K_*d_T^\sharp(\nu,\mu_*)
\]
whenever $\nu$ satisfies Assumption~\ref{ass:initial-law} and $d_T^\sharp(\nu,\mu_*)<\rho$. In particular, these boundaries have a common upper bound on $[\tau/2,T]$.

By \eqref{eq:positive-time-loss-flux-stability}, shrinking $\rho$ further also gives a common strictly positive lower bound for $g^\nu$ on $[\tau/2,T]$. We may now apply Step~2 to any two initial laws in this neighborhood, with $s=\tau/2$ and $\theta=\tau$. The constants $R,\gamma,a_0,p$ can be chosen uniformly. Moreover, Step~1 and
\[
    d_T^\sharp(\mu_0^1,\mu_0^2)\le d_T^\sharp(\mu_0^1,\mu_*)+d_T^\sharp(\mu_0^2,\mu_*)<2\rho
\]
ensure $\kappa<1/2$ and $\sigma<a$ after one final reduction of $\rho$. Estimate~\eqref{eq:support-boundary-comparison-bound} now proves \eqref{eq:support-boundary-stability}.

\medskip
\noindent
\emph{Step 4: Continuity under Wasserstein convergence.}
Suppose $W_1(\mu_0^n,\mu_0)\to0$. Since
\[
    \|C_{\mu_0^n}-C_{\mu_0}\|_{L^\infty(\mathbb R_+)}\le W_1(\mu_0^n,\mu_0),
\]
it remains to control the free default curves in \eqref{eq:augmented-call-distance}. For every $\eta\in(0,T)$, formula~\eqref{eq:free-default-curve} gives
\[
    \sup_{\eta\le t\le T}|q_{\mu_0^n}^{\mathrm{fr}}(t)-q_{\mu_0}^{\mathrm{fr}}(t)|\le\sqrt{\frac{2}{\pi\eta}}\,W_1(\mu_0^n,\mu_0).
\]
For $0<t\le\eta$, monotonicity of the free default curves gives
\[
|q_{\mu_0^n}^{\mathrm{fr}}(t)-q_{\mu_0}^{\mathrm{fr}}(t)|\le q_{\mu_0^n}^{\mathrm{fr}}(\eta)+q_{\mu_0}^{\mathrm{fr}}(\eta).
\]
Letting $n\to\infty$ and then $\eta\downarrow0$ proves uniform convergence on $(0,T]$, because $q_{\mu_0}^{\mathrm{fr}}(\eta)\to\mu_0(\{0\})=0$. Thus, $d_T^\sharp(\mu_0^n,\mu_0)\to0$, and the asserted boundary convergence follows from \eqref{eq:support-boundary-stability}.
\end{proof}

\subsection{Uniqueness of the optimal control}
\label{subsec:continuous-optimal-control-uniqueness}
Throughout this subsection, let Assumption~\ref{ass:initial-law} hold. Let $(X,\Lambda)$ be the probabilistic solution constructed in Section~\ref{sec:continuous}, which is unique by Corollary~\ref{cor:probabilistic-FBP-uniqueness}, and retain the notations
\[
\mu_t:=\mathcal L(X_t),\quad q_t:=\mu_t(\{0\}),\quad h_t:=\alpha q_t,\quad D_t(a):=D_{\mu_t}(a),\quad m_0:=m_{\mu_0}.
\]
For an arbitrary admissible system $(\widetilde X,\widetilde B,\widetilde L,\widetilde\tau)\in\mathfrak A(\mu_0)$, use the analogous notation $\widetilde\mu_t$, $\widetilde q_t$, $\widetilde h_t$, and $\widetilde D_t$, and set the difference $d_t:=D_t-\widetilde D_t$.

\begin{lemma}[Dominance and strict propagation]
\label{lem:dynamic-maximality-strict-propagation}
For every admissible system and every $t\in[0,T]$,
\begin{equation}\label{eq:dynamic-optimal-dominance}
\mu_t\succsim_2\widetilde\mu_t,\quad d_t(a)\ge0\quad(a>0),\quad q_t\le\widetilde q_t.
\end{equation}
Moreover, if $0<s<t\le T$ and $0<r<\inf_{v\in[s,t]}\Lambda_v$, then
\begin{equation}\label{eq:optimal-gap-killed-propagation}
d_t(a)\ge\int_0^r p_r(t-s,a,y)\,d_s(y)\,\d y,\quad 0<a<r,
\end{equation}
where $p_r$ is the transition density of Brownian motion killed upon exiting $(0,r)$. Consequently, if $d_s\not\equiv0$ on $(0,r)$, then $d_t(a)>0$ for every $a\in(0,r)$ and $\widetilde q_t>q_t$.
\end{lemma}

\begin{proof}
Fix $t>0$. Applying \eqref{eq:lemma32-inductive-comparison} on $[0,t]$ with meshes $\Delta_k=t/k$ gives $\mu_t^{\Delta_k}\succsim_2\widetilde\mu_t$. The compactness and identification argument in the proof of Theorem~\ref{thm:VDist_to_Cont_Conv}, together with Corollary~\ref{cor:probabilistic-FBP-uniqueness}, identifies every subsequential cutoff limit with the restriction of $\mu$ to $[0,t]$. Passing to the limit in the bounded continuous functions $x\mapsto x\wedge a$ and using \eqref{eq:integrated-quantile-characterization} proves \eqref{eq:dynamic-optimal-dominance}.

For nonnegative $\eta\in C_c^\infty((0,\infty))$, set
\[
F_\eta(x):=\int_0^\infty(x\wedge a)\,\eta(a)\,\d a.
\]
Extend $\eta$ by zero at $0$. Then
\[
F_\eta'(x)=\int_x^\infty\eta(a)\,da\ge0,\quad F_\eta''(x)=-\eta(x).
\]
Write $\widetilde L^c$ for the continuous part of $\widetilde L$. Since $\Delta\widetilde X_z=-\Delta\widetilde L_z$, It\^o's formula with jumps gives, for $0\le u<v\le T$,
\begin{align*}
F_\eta(\widetilde X_v)-F_\eta(\widetilde X_u)={}&\int_u^v F_\eta'(\widetilde X_{z-})\,d\widetilde B_{z\wedge\widetilde\tau}-\frac12\int_u^v\eta(\widetilde X_z)\,dz-\int_{(u,v]}F_\eta'(\widetilde X_z)\,d\widetilde L_z^c\\
&\quad+\sum_{u<z\le v}\bigl(F_\eta(\widetilde X_z)-F_\eta(\widetilde X_{z-})\bigr).
\end{align*}
Here, the jump part of the allocation integral has been combined with It\^o's jump correction, and the survival indicator in the second-order term is omitted because $\eta(0)=0$. The continuous allocation term is non-positive, and each jump contribution satisfies
\[
F_\eta(\widetilde X_z)-F_\eta(\widetilde X_{z-})=-\int_{\widetilde X_z}^{\widetilde X_{z-}}\left(\int_y^\infty\eta(a)\,da\right)dy\le0.
\]
The stochastic integral has mean zero because $F_\eta'$ is bounded, while the allocation terms are integrable by $\EE[\widetilde L_T]<\infty$. Taking expectations therefore yields
\[
\langle\widetilde D_v-\widetilde D_u,\eta\rangle\le-\frac12\int_u^v\langle\widetilde\mu_z^+,\eta\rangle\,dz=\frac12\int_u^v\langle\widetilde D_z,\eta''\rangle\,dz,
\]
where $\widetilde\mu_z^+:=\widetilde\mu_z|_{(0,\infty)}$. The last equality follows from $\partial_{aa}\widetilde D_z=-\widetilde\mu_z^+$ in the sense of distributions on $(0,\infty)$. Hence, $\widetilde D$ is a distributional subsolution of the heat equation. Since $D=m_0-H$, equation~\eqref{eq:weak-augmented-call-measure} shows that $D$ solves the heat equation on $(s,t)\times(0,r)$. Therefore, $d$ is a nonnegative distributional supersolution there, with $d_v(0)=0$.

We justify comparison despite possible jumps of the competitor. The uniform spatial Lipschitz bound and c\`adl\`agness of $\widetilde X$ imply that $d$ is c\`adl\`ag with values in $C([0,r])$. Moreover, $d_v\ge d_{v-}$, since $D$ is continuous and $\widetilde X$ has only downward jumps. Define
\[
d^\varepsilon(v,a):=\frac1\varepsilon\int_0^\varepsilon d(v+\theta,a)\,d\theta,\quad s\le v\le t-\varepsilon.
\]
These are continuous supersolutions with nonnegative lateral boundary values. Let $P^r$ be the killed semigroup with kernel $p_r$, and choose $\chi_n\in C_c^\infty((0,r))$ with $0\le\chi_n\uparrow1$ on $(0,r)$. Applying Lemma~\ref{lem:positive-set-parabolic-maximum-principle} to
\[
P^r_{v-s}\bigl(\chi_n d^\varepsilon(s,\cdot)\bigr)(a)-d^\varepsilon(v,a)
\]
gives $d^\varepsilon(v,\cdot)\ge P^r_{v-s}(\chi_n d^\varepsilon(s,\cdot))$. Letting $n\to\infty$ and then $\varepsilon\downarrow0$ proves the comparison for $v<t$. Letting $v\uparrow t$ and using $d_t\ge d_{t-}$ proves \eqref{eq:optimal-gap-killed-propagation}, including at $t=T$.

Strict positivity in $(0,r)$ follows from positivity of $p_r$. Finally,
\[
\lim_{a\downarrow0}\frac{d_t(a)}a=\widetilde q_t-q_t.
\]
Dividing \eqref{eq:optimal-gap-killed-propagation} by $a$, applying Fatou's lemma, and using $\partial_a p_r(t-s,0+,y)>0$ for $0<y<r$ gives the strict default inequality $q_t<\widetilde q_t$.
\end{proof}

\begin{theorem}[Uniqueness of the continuous-time optimal control]
\label{thm:continuous-optimal-control-uniqueness}
Let $f$ be any non-constant payoff allowed in item~2 of Definition~\ref{def:stochastic-orders}. If $(\widetilde X,\widetilde B,\widetilde L,\widetilde\tau)\in\mathfrak A(\mu_0)$ satisfies
\[
\EE[f(\widetilde X_T)]=V_f(0,\mu_0),
\]
then
\begin{equation}\label{eq:unique-optimal-law-flow}
\widetilde\mu_t=\mu_t,\qquad 0\le t\le T,
\end{equation}
and, almost surely,
\begin{equation}\label{eq:unique-optimal-skorokhod-control}
\widetilde L_t=\sup_{0<s\le t\wedge\widetilde\tau}\bigl(\widetilde X_0+\widetilde B_s-\Lambda_s\bigr)^+,\quad 0\le t\le T,
\end{equation}
where the supremum is understood to be zero at $t=0$. Consequently, the optimal flow and the joint law of the optimal state and control are unique. Optimal systems driven by the same initial state and Brownian motion are indistinguishable.
\end{theorem}

\begin{proof}
\textit{Equality of the flow before $T$.}
Suppose that $\widetilde\mu_s\ne\mu_s$ for some $s\in(0,T)$. Then $d_s(a_0)>0$ for some $a_0\in(0,\Lambda_s)$. Indeed, if $d_s=0$ on this interval, taking the right derivative at zero gives $\widetilde q_s=q_s$, and \eqref{eq:continuous-resource-invariant} gives equality of the means. Letting $a\uparrow\Lambda_s$ then shows that $\widetilde\mu_s$ has no mass above $\Lambda_s$. The truncated moments therefore agree at every strike, which identifies the two laws.

By the continuity of $\Lambda$, choose $u\in(s,T)$ and $r>a_0$ with $r<\inf_{v\in[s,u]}\Lambda_v$. Lemma~\ref{lem:dynamic-maximality-strict-propagation} gives $\widetilde q_u>q_u$, so $d_u$ is strictly positive near zero. Set $r_*:=m_0-\alpha>0$. Since $\Lambda\ge r_*$ by Lemma~\ref{lem:boundary-uniform-positivity}, applying the same lemma on $(0,r_*/2)$ gives $\widetilde q_v>q_v$, $u<v\le T$. In fact, it must hold that
\begin{equation}\label{eq:strict-terminal-truncated-moments}
d_T(a)>0,\quad a>0.
\end{equation}
For $a<\Lambda_T$, choose $v\in(u,T)$ close enough to $T$ and $r>a$ with $r<\inf_{w\in[v,T]}\Lambda_w$. The strict default gap at $v$ makes $d_v$ positive near zero, so \eqref{eq:optimal-gap-killed-propagation} yields \eqref{eq:strict-terminal-truncated-moments}. For $a\ge\Lambda_T$, the support identity and resource conservation give
\[
D_T(a)=m_{\mu_T},\quad\widetilde D_T(a)\le m_{\widetilde\mu_T}=m_{\mu_T}-\alpha(\widetilde q_T-q_T)<m_{\mu_T}.
\]
Using the coefficients $J_f,\ell,\gamma$ from \eqref{eq:payoff-mixture-decomposition}, we obtain
\begin{equation}\label{eq:strict-terminal-payoff-gap}
\langle f,\mu_T\rangle-\langle f,\widetilde\mu_T\rangle=(J_f+\alpha\ell)(\widetilde q_T-q_T)+\int_{(0,\infty)}d_T(a)\,\gamma(da)>0.
\end{equation}
The inequality is strict because $f$ is nonconstant. This contradicts the assumed optimality of $(\widetilde X,\widetilde B,\widetilde L,\widetilde\tau)$, so $\widetilde\mu_t=\mu_t$ for all $t<T$.\\

\textit{Equality at $T$.}
C\`adl\`agness of $\widetilde X$ and continuity of the canonical flow give $\mathcal L(\widetilde X_{T-})=\mu_T$. Condition~\eqref{eq:continuous-no-allocation-default}, together with absorption at default, implies $\{\widetilde X_T=0\}=\{\widetilde X_{T-}=0\}$ almost surely, and hence $\widetilde q_T=q_T$. Exact balance and continuity of $h$ now give
\[
\EE[\Delta\widetilde L_T]=\widetilde h_T-\widetilde h_{T-}=h_T-h_{T-}=0.
\]
Therefore, $\widetilde X_T=\widetilde X_{T-}$ almost surely, proving \eqref{eq:unique-optimal-law-flow}.\\

\textit{Identification of the control.}
By Lemma~\ref{lem:support-identification}, flow identification, and Proposition~\ref{prop:FullInitialBoundaryTrace}, a countable dense set of times containing $T$, together with c\`adl\`agness of $\widetilde X$ and continuity of $\Lambda$, gives
\[
0\le\widetilde X_t\le\Lambda_t\quad(0\le t\le T),\quad\widetilde X_{t-}\le\Lambda_t\quad(0<t\le T)
\]
simultaneously almost surely. In particular, both systems are bounded by $M:=\sup_{[0,T]}\Lambda<\infty$.

It\^o's formula for $x^2$ gives
\begin{align*}
\EE[\widetilde X_T^2]-\EE[\widetilde X_0^2]&=\int_0^T(1-q_t)\,dt-2\EE\!\int_{(0,T]}\widetilde X_{t-}\,d\widetilde L_t+\EE\!\sum_{0<t\le T}(\Delta\widetilde L_t)^2,\\
\EE[X_T^2]-\EE[X_0^2]&=\int_0^T(1-q_t)\,dt-2\int_{(0,T]}\Lambda_t\,dh_t,
\end{align*}
where the second identity uses continuity of the canonical regulator and its boundary-contact property. All terms are integrable because $\EE[\widetilde L_T]<\infty$ and $\sum(\Delta\widetilde L_t)^2\le M\widetilde L_T$. Since the marginal laws agree and $\EE[d\widetilde L_t]=dh_t$, subtraction yields
\begin{equation}\label{eq:optimal-control-quadratic-identity}
\EE\!\left[\int_{(0,T]}(\Lambda_t-\widetilde X_{t-})\,d\widetilde L_t+\frac12\sum_{0<t\le T}(\Delta\widetilde L_t)^2\right]=0.
\end{equation}
Both terms are non-negative. Hence $\widetilde L$ is continuous and increases only on $\{\widetilde X=\Lambda\}$. Together with $\widetilde L_0=0$ and stopping at default, these are the stopped upper Skorokhod conditions, proving \eqref{eq:unique-optimal-skorokhod-control}. The resulting state and control are measurable functionals of $(\widetilde X_0,\widetilde B)$, which proves the remaining uniqueness assertions.
\end{proof}

\section{Potential Extensions and Further Directions}\label{sec:further-directions}

The cutoff principle rests on two structural properties: conservation of total resources and preservation of increasing-concave order by the uncontrolled dynamics. These properties suggest extensions beyond the present specification. We describe three directions, distinguishing adaptations of the discrete comparison argument from questions requiring further analysis.

\paragraph{Nonlinear cumulative losses.}
The balance constraint can be replaced by
\[
    \mathbb E[L_t]=\Gamma(q_t),\quad q_t=\mathbb P[\tau\le t],
\]
where $\Gamma(0)=0$, $\Gamma$ is smooth, and $0<g_-\le\Gamma'\le g_+<\infty$ on $[0,1]$. For example, $\Gamma(q)=\alpha q+\beta q^2$, with $\beta\ge0$, allows the loss generated by an additional default to increase with cumulative defaults. Assuming $m_0>\Gamma(1)$, the conserved resource is $m_\mu+\Gamma(q_\mu)$, and the discrete allocation budget becomes $\Gamma(q_{S_\Delta\mu})-\Gamma(q_\mu)$. With these substitutions, the truncated-moment comparison and backward induction still give cutoff optimality for the same payoff class; convexity of $\Gamma$ is unnecessary. For a classical reflected limit, the balance condition would read
\[
    u(t,\Lambda_t)=\Gamma'(q_t)u_x(t,0).
\]
The initial mean condition retains a positive resource margin, while the derivative bounds make allocation budgets comparable to default increments. These conditions provide a concrete route to adapting the compactness and boundary estimates, although the continuous-time extension remains to be proved.

\paragraph{State-dependent volatility.}
Another extension replaces Brownian motion by the driftless diffusion $\mathrm dX_t=\sigma(X_t)\,\mathrm dB_t-\mathrm dL_t$ before absorption, with a common coefficient $\sigma\in C_b^\infty([0,\infty))$ satisfying $0<\underline\sigma\le\sigma\le\overline\sigma$. The absorbed uncontrolled diffusion preserves the mean and increasing-concave order, so the discrete cutoff comparison remains applicable. The call-function formulation also retains a local parabolic structure: writing $H(t,x)=\mathbb E[(X_t-x)^+]+\alpha q_t$, a smooth reflected candidate satisfies
\[
    H_t=\tfrac12\sigma^2(x)H_{xx},\qquad
    \sigma^2(\Lambda_t)u(t,\Lambda_t)
      =\alpha\,\partial_x(\sigma^2u)(t,0).
\]
The first identity holds inside the continuation region; the second couples the reflection rate to the default flux. Establishing the continuous limit and its regularity would require variable-coefficient counterparts of the heat-kernel and free boundary estimates used here.

\paragraph{Finite populations.}
For $N$ banks driven by independent Brownian motions, a natural particle approximation allocates each default loss $\alpha$ by capping the surviving capitals. Its feasibility can be controlled through the empirical resource $R_t^N=N^{-1}\sum_iX_t^{i,N}+\alpha q_t^N$, where $q_t^N$ is the default fraction. Stop the system at $\eta_N=\inf\{t:R_t^N\le\alpha\}$. Before this time all allocations are feasible, and the stopped resource is a continuous martingale with quadratic variation at most $t/N$. For $\delta>0$, the condition $R_0^N\ge\alpha+\delta$ almost surely therefore implies $\mathbb P(\eta_N\le T)\le T/(N\delta^2)$. This vanishing probability provides a starting point for connecting the finite system to the optimal mean-field flow. Empirical convergence, identification of the limiting allocation boundary, and asymptotic optimality remain separate questions requiring additional estimates.

\appendix
\section{Appendix}
\label{sec:appendix}
\subsection*{Stochastic dominance}
\begin{lemma}[Decomposition of second-order dominance]
\label{lem:icv-decomposition}
Let $\mu,\nu\in\mathcal P_1(\mathbb R_+)$. Suppose that $\mu \succsim_2 \nu$. Then there exists $\eta\in\mathcal P_1(\mathbb R_+)$ such that $\eta \succsim_1\nu$ and that $\eta$ dominates $\mu$ in convex order: for every convex function $\psi:\mathbb R_+\to\mathbb R$ for which the integrals are well-defined,
\[
    \int_{\mathbb R_+}\psi\,\d\mu\leq\int_{\mathbb R_+}\psi\,\d\eta .
\]
\end{lemma}

\begin{proof}
The supermartingale-coupling characterization is
\cite[Theorem~4.A.5]{ShakedShanthikumar2007}. It states that the assumption $\mu\succsim_2\nu$ is equivalent to the existence of random variables $X$ and $Y$, defined on a common probability space, such that
\begin{align*}
X \sim \nu,\quad Y \sim \mu,\quad\mathbb E[X\mid Y] \le Y\quad \text{a.s.}
\end{align*}
Since $\mu,\nu\in\mathcal P_1(\mathbb R_+)$, we may take $X,Y\ge 0$ almost surely. Define $d(Y):=Y-\mathbb E[X\mid Y]$. Then $d(Y)\ge 0$ almost surely. Now set $Z:=X+d(Y)$ and define $\eta:=\mathrm{Law}(Z)$.  Since $X\ge 0$ and $d(Y)\ge 0$ almost surely, we have $Z\ge 0$ a.s. Thus $\eta\in\mathcal P_1(\mathbb R_+)$. Moreover, since $X$ and $Y$ have finite first moments and $0\le d(Y)\le Y+\mathbb E[X\mid Y]$, we also have $\mathbb E[Z]<\infty$. As $d(Y)\ge0$ a.s., it is immediate that $\eta=\mathrm{Law}(Z)\succsim_1\mathrm{Law}(X)=\nu$. It remains to prove the convex-order relation. By construction,
\begin{align*}
    \mathbb E[Z\mid Y]=\mathbb E[X+d(Y)\mid Y]=\mathbb E[X\mid Y]+d(Y)=\mathbb E[X\mid Y]+Y-\mathbb E[X\mid Y] =Y.
\end{align*}
Thus, $(Y,Z)$ is a martingale coupling from $\mu$ to $\eta$. Hence, for every convex function $\psi:\mathbb R_+\to\mathbb R$ for which the integrals are well-defined, Jensen's inequality gives
\begin{align*}
\psi(Y)=\psi\bigl(\mathbb E[Z\mid Y]\bigr)\le\mathbb E[\psi(Z)\mid Y].
\end{align*}
Taking expectations, we obtain
\begin{align*}
\int_{\mathbb R_+} \psi(x)\,\mu(dx)=\mathbb E[\psi(Y)]\le\mathbb E[\psi(Z)]=\int_{\mathbb R_+} \psi(x)\,\eta(dx).
\end{align*}
This is precisely the asserted convex-order relation. Together with $\eta\succsim_1\nu$, it proves the claim.
\end{proof}


{\footnotesize
\bibliography{Main}
}

\bigskip\bigskip\bigskip

\end{document}